\documentclass[reqno,11pt]{article}
\usepackage{bbm}
\usepackage{mathrsfs}
\usepackage{amssymb}

\usepackage[usenames,dvipsnames]{xcolor}
\usepackage{amsthm}
\usepackage{amsmath}
\usepackage{amsfonts}
\usepackage{enumerate}
\usepackage{txfonts}
\usepackage{bm}

\usepackage{psfrag}

\usepackage{graphicx}

\numberwithin{equation}{section}

\newtheorem{theorem}{Theorem}[section]
\newtheorem{lemma}[theorem]{Lemma}
\newtheorem{corollary}[theorem]{Corollary}
\newtheorem{proposition}[theorem]{Proposition}

\newtheorem{remark}[theorem]{Remark}

\theoremstyle{definition}
\newtheorem{definition}[theorem]{Definition}
\newtheorem{assumption}[theorem]{Assumption}
\newtheorem{example}[theorem]{Example}

\begin{document}

\def\be{\begin{eqnarray}}
\def\ee{\end{eqnarray}}
\def\p{\partial}
\def\no{\nonumber}
\def\eps{\epsilon}
\def\de{\delta}
\def\De{\Delta}
\def\om{\omega}
\def\Om{\Omega}
\def\f{\frac}
\def\th{\theta}
\def\la{\lambda}
\def\lab{\label}
\def\b{\bigg}
\def\var{\varphi}
\def\na{\nabla}
\def\ka{\kappa}
\def\al{\alpha}
\def\La{\Lambda}
\def\ga{\gamma}
\def\Ga{\Gamma}
\def\ti{\tilde}
\def\wti{\widetilde}
\def\wh{\widehat}
\def\ol{\overline}
\def\ul{\underline}
\def\Th{\Theta}
\def\si{\sigma}
\def\Si{\Sigma}
\def\oo{\infty}
\def\q{\quad}
\def\z{\zeta}
\def\co{\coloneqq}
\def\eqq{\eqqcolon}
\def\bt{\begin{theorem}}
\def\et{\end{theorem}}
\def\bc{\begin{corollary}}
\def\ec{\end{corollary}}
\def\bl{\begin{lemma}}
\def\el{\end{lemma}}
\def\bp{\begin{proposition}}
\def\ep{\end{proposition}}
\def\br{\begin{remark}}
\def\er{\end{remark}}
\def\bd{\begin{definition}}
\def\ed{\end{definition}}
\def\bpf{\begin{proof}}
\def\epf{\end{proof}}
\def\bex{\begin{example}}
\def\eex{\end{example}}
\def\bq{\begin{question}}
\def\eq{\end{question}}
\def\bas{\begin{assumption}}
\def\eas{\end{assumption}}
\def\ber{\begin{exercise}}
\def\eer{\end{exercise}}
\def\mb{\mathbb}
\def\mbR{\mb{R}}
\def\mbZ{\mb{Z}}
\def\mbf{\mathbf}
\def\u{\mbf{u}}
\def\U{\mbf{U}}
\def\v{\mbf{v}}
\def\e{\mbf{e}}
\def\M {\mbf{M}}
\def\x{\mbf{x}}
\def\mc{\mathcal}
\def\mcS{\mc{S}}
\def\ms{\mathscr}
\def\lan{\langle}
\def\ran{\rangle}
\def\lb{\llbracket}
\def\rb{\rrbracket}
\def\div{\textrm{div}\,}

\title{Smooth axisymmetric transonic irrotational flows to the steady Euler equations with an external force}

\author{Shangkun Weng\thanks{School of Mathematics and Statistics, Wuhan University, Wuhan, Hubei Province, China, 430072. Email: skweng@whu.edu.cn}\and  \and Yan Zhou\thanks{School of Mathematics and Statistics, Wuhan University, Wuhan, Hubei Province, China, 430072. Email: yanz0013@whu.edu.cn}}
\date{}


\pagestyle{myheadings} \markboth{Smooth axi-symmetric transonic irrotational flows}{Smooth axi-symmetric transonic irrotational flows}\maketitle

\begin{abstract}
  For a class of external forces, we prove the existence and uniqueness of smooth transonic flows to the one dimensional steady Euler system with an external force, which is subsonic at the inlet and flows out at supersonic speed after smoothly accelerating through the sonic point. We then investigate the structural stability of the one-dimensional smooth transonic flows with positive acceleration under axisymmetric perturbations of suitable boundary conditions, and establish the first existence and uniqueness result for smooth axisymmetric transonic irrotational flows. The key point lies on the analysis of a linear second order elliptic-hyperbolic mixed differential equation of Keldysh type with a singular term. Some weighted Sobolev spaces $H_r^m(D) (m=2,3,4)$ are introduced to deal with the singularities near the axis. Compared with the stability analysis in the two dimensional case by Weng and Xin (arXiv:2309.07468), there are several interesting new observations about the structure of the linear mixed type equation with a singular term which play crucial roles in establishing the $H^4_r(D)$ estimate.


\end{abstract}

\begin{center}
\begin{minipage}{5.5in}
\end{minipage}
\end{center}

\section{Introduction and main results}\label{1df}\noindent

In this paper, we concern the smooth axisymmetric transonic irrotational flows to the steady isentropic compressible inviscid flow with an external force:
\begin{eqnarray}\label{2deuler-force}
\begin{cases}
\div (\rho \u) =0,\\
\div (\rho \u \otimes \u + P  I_3) = \rho \na \Phi,
\end{cases}
\end{eqnarray}
where $\u = (u_1, u_2, u_3)$, $\rho $ represent the velocity and density respectively, $P = \rho^{\ga}$ with $\ga >1$ is the pressure, $\Phi $ is the potential force. Denote the Bernoulli's quantity $B = \f 12 |\u|^2 + h (\rho) - \Phi $ with the enthalpy $h (\rho) = \f {\ga }{\ga - 1 } \rho^{\ga -1} $. 

Recently, Weng and Xin \cite{WX23} studied smooth transonic flows with nonzero vorticity in de Laval nozzles for a quasi two dimensional steady Euler flow model which generalizes the classical quasi one dimensional model. They first proved the existence and uniqueness of smooth transonic flows to the quasi one-dimensional model, which start from a subsonic state and accelerate to reach a sonic state at the throat and then become supersonic. These flows may have positive or zero acceleration at their sonic points and the degeneracy types near the sonic point are classified precisely. They further proved the existence and uniqueness of smooth transonic flow with nonzero vorticity to the quasi two dimensional model adjacent to the one dimensional smooth transonic flows with positive accelerations. The sonic curve may not locate at the throat of the nozzle.

\subsection{Transonic flows to one dimensional steady Euler system with an external force}\noindent

Motivated by \cite{WX23}, we study smooth transonic flows with nonzero vorticity to steady Euler system with nontrivial external forces. First, we identify a class of external forces and establish the existence and uniqueness of one dimensional smooth transonic flows to \eqref{2deuler-force}. Namely, we solve the problem
\begin{eqnarray}\label{1df}\begin{cases}
(\bar{\rho} \bar{u})'(x_1)=0, \ \ \forall x_1\in [L_0,L_1],\\
\bar{\rho} \bar{u} \bar{u}'+ \frac{d}{dx_1} P(\bar{\rho})= \bar{\rho} \bar{f}(x_1), \ \ \forall x_1\in [L_0,L_1],\\
\bar{\rho}(L_0)=\rho_0>0,\ \ \bar{u}(L_0)= u_0>0,
\end{cases}\end{eqnarray}
where  $\bar{f} (x_1) $ is a given infinitely differentiable function on $[L_0,L_1]$ $(L_0 < 0 < L_1)$ and the initial state at $x_1=L_0$ is subsonic, that is $u_0^2<c^2(\rho_0)=\gamma \rho_0^{\gamma-1}$. 

Our goal is to find suitable assumptions on $\bar{f}$ and boundary data $(\rho_0, u_0)$ such that there exists a smooth accelerating transonic flow with the sonic state occurring on the point $x_1=0$. Inspired by \cite{WX23}, we assume that the external force satisfies
\be\label{1df04}\begin{cases}
\bar{f}(x_1)<0,\ \forall x\in [L_0,0),\\
\bar{f}(0)=0,\\
\bar{f}(x_1)>0, \ \ \forall x_1\in (0,L_1].
\end{cases}\ee
Denote $J= \bar{\rho} \bar{u}=\rho_0 u_0>0$. The Bernoulli's law yields that
\begin{eqnarray}\label{1df2}
\frac{1}{2}(\bar{u}(x_1))^2+ \frac{\gamma}{\gamma-1} \bar{\rho}^{\gamma-1} -\int_{L_0}^{x_1} \bar{f}(\tau) d\tau\equiv B_0 \co \frac{1}{2} u_0^2 + \frac{\gamma}{\gamma-1} \bar{\rho}_0^{\gamma-1} .
\end{eqnarray}
Suppose the flow becomes sonic at $x_1=0$, i.e. $\bar{u}^2({0})=c^2(\bar{\rho}(0))=\gamma \bar{\rho}(0)^{\gamma-1}=\gamma (\frac{J}{\bar{u}(0)})^{\gamma-1}$, then
\be\no
\bar{u}(0)=\gamma^{\frac{1}{\gamma+1}} J^{\frac{\gamma-1}{\gamma+1}}=: c_*.
\ee
Therefore one can conclude from \eqref{1df2} that
\be\label{1df90}
\int_{L_0}^{0} \bar{f}(\tau) d\tau= \frac{\gamma+1}{2(\gamma-1)} \gamma^{\frac{2}{\gamma+1}} J^{\frac{2(\gamma-1)}{\gamma+1}}- B_0.
\ee


\begin{proposition}\label{1dfge}({\bf General accelerating transonic flows.}) Assume that $(u_0, \rho_0)$ is subsonic and the external force $\bar{f}$ satisfies \eqref{1df04} and \eqref{1df90}. Then there exists a unique accelerating transonic flow $(\bar{u}(x_1), \bar{\rho} (x_1)) \in C ([L_0, L_1])$ which is subsonic in $[L_0, 0)$, supersonic in $(0, L_1]$ with a sonic state at $x_1 = 0$. Furthermore, $(\bar{u}(x_1), \bar{\rho} (x_1))$ is smooth and satisfies the equations \eqref{1df} on $[L_0, 0) \cup (0, L_1]$.
\end{proposition}
\begin{proof}
  For smooth solutions, the problem \eqref{1df} is equivalent to
  \be\no
  F (x_1, \bar{u}(x_1); J) = 0, \q \bar{\rho } (x_1) = \f {J}{\bar{u} (x_1)},
  \ee
  where
  \be\no
  F(x_1, t; J) = \f 12 t^2 + \f {\ga J^{\ga - 1}}{\ga - 1} \f 1{t^{\ga - 1}} - \int_{L_0}^{x_1} \bar{f} (\tau) d \tau - B_0.
  \ee
  Fixing $x_1 \in [L_0, L_1]$, then on $(0, + \oo)$, $F(x_1, t;J)$ attains its minimum value at $t = t_* (x_1) = (\ga J^{\ga - 1})^{\f 1{\ga + 1}}$. For any $x_1 \in [L_0, 0) \cup (0, L_1]$, \eqref{1df04} leads to
  \be\no
  F (x_1, t_* (x_1); J)  =  \f {\ga + 1}{2 (\ga - 1)} (\ga J^{\ga - 1})^{\f 2{\ga + 1}} - \int_{L_0}^{x_1} \bar{f} (\tau) d \tau - B_0
   =  \int_{x_1}^0 \bar{f} (\tau) d \tau < 0.
  \ee
  From $\ga > 1$, one has $\lim_{ t \to 0^+} F (x_1, t;J) = \lim_{t \to +\oo} F (x_1, t; J) = + \oo$, and that $F (x_1, t; J)$ is monotone decreasing in $(0, t_* (x_1)]$ and monotone increasing in $[t_* (x_1), +\oo)$. Thus for each $x_1 \in [L_0, 0) \cup (0, L_1]$, $F (x_1, t; J) = 0$ has exactly two solutions $0 < t_{sub} (x_1) < t_* < t_{sup} (x_1) < + \oo$. For $x_1 = 0$, $F (0, t; J) = 0$ has exactly one solution $c_* = t_* (0)$. Define the function $\bar{u} (x_1)$ as follows:
  \be\no
  \bar{u} (x_1) = \begin{cases}
                    t_{sub} (x_1), & \forall x_1 \in [L_0, 0), \\
                    t_* (0), & x_1 = 0, \\
                    t_{sup} (x_1), & \forall x_1 \in (0, L_1].
                  \end{cases}
  \ee
  For any $\de \in (0, 1)$, one has
  \begin{eqnarray*}
  F (x_1, (1 \pm \de)t_* (x_1);J) & = & (\ga J^{\ga - 1})^{\f 2{\ga + 1}} \b( \f {(1 \pm \de)^2}2 + \f {(1 \pm \de)^{1 - \ga}}{\ga - 1} \b) - (\ga J^{\ga - 1})^{\f 2 {\ga + 1}} \b(\f 12 + \f 1{\ga - 1}\b) \\\no
  & = & (\ga J^{\ga - 1})^{\f 2{\ga + 1}} (\f {\ga + 1}2 \de^2 + O(|\de|^3)) > 0.
  \end{eqnarray*}
  This implies that
  \be\no
  (1 - \de) t_* (x_1) < t_{sub } (x_1) < t_* (x_1) < t_{sup} (x_1) < (1 + \de) t_*(x_1).
  \ee
  Hence $\lim_{x_1 \to 0} t_{sub} (x_1) = \lim_{x_1 \to 0} t_{sup} (x_1) = \lim_{x_1 \to 0} t_* (x_1) = t_* (0)$, and $(\bar{u} (x_1), \bar{\rho} (x_1) = \f {J}{\bar{u} (x_1)}) \in (C([L_0, L_1]))^2$ and is subsonic in $[L_0, 0)$, supersonic in $(0, L_1]$, and the sonic point is $x_1 = 0$. Furthermore, for each $x_1 \in [L_0, 0) \cup (0, L_1]$, one has
  \begin{equation*}
  \p_t F (x_1, t_{sub} (x_1);J) < 0, \q \p_t F (x_1, t_{sup} (x_1);J) > 0.
  \end{equation*}
  By the implicit function theorem, $(\bar{u}(x_1), \bar{\rho} (x_1))$ is smooth and satisfies the equations \eqref{1df} on $[L_0, 0) \cup (0, L_1]$. However, we have no information on the differentiability properties for $(\bar{u}(x_1), \bar{\rho} (x_1))$ at the sonic point $x_1 = 0$.
\end{proof}


We then try to improve the regularity of the flow $(\bar{u}(x_1), \bar{\rho} (x_1))$ obtained in Proposition \ref{1dfge} at the sonic point $x_1 = 0$. It follows from \eqref{1df} that
\be\label{1df1}\begin{cases}
\bar{\rho}(x_1)=\frac{J}{\bar{u}(x_1)},\\
(\bar{u}^{2}-c^2(\bar{\rho})) \bar{u}'=(\bar{u}^2- c_*^{\gamma+1}\bar{u}^{-\gamma+1}) \bar{u}'= \bar{u}\bar{f}.
\end{cases}\ee
Suppose that $(\bar{u}(x_1), \bar{\rho} (x_1))$ is smooth at $x_1=0$, then it satisfies \eqref{1df} at $x_1 = 0$ as well. Differentiating \eqref{1df1} and evaluating at $x_1=0$ yields that
\begin{equation*}
  \bar{f}'(0)=(\gamma+1)(\bar{u}'(0))^2\geq 0.
\end{equation*}

Consider the case where the smooth transonic flows have positive acceleration at the sonic point:
\be\label{1df040}
\bar{f}'(0)>0.
\ee


\begin{proposition}\label{1dfpo}({\bf Smooth transonic flows with positive acceleration at the sonic point.}) Assume the initial data $(u_0, \rho_0)$ is subsonic and the external force satisfies \eqref{1df04}, \eqref{1df90} and \eqref{1df040}. The problem \eqref{1df} has a unique smooth solution $(\bar{\rho}(x_1), \bar{u}(x_1))\in
C^{\infty}([L_0,L_1])$ which is subsonic in $[L_0,0)$, supersonic in $(0, L_1]$ with a sonic state at $x_1=0$.
\end{proposition}

\begin{proof}


It remains to show the solution given in Proposition \ref{1dfge} passes through the sonic point $x_1 = 0$ smoothly. This will be proved by a reduction of degeneracy near $x_1=0$ and the implicit function theorem. It is easy to see that $F(0,c_*;J)=\frac{\p F}{\p t}(0,c_*;J)=\frac{\p F}{\p x_1}(0,c_*; J)=0$ and
\begin{eqnarray}\no
\frac{\p^2 F}{\p t^2}(0,c_*;J)=1+\gamma, \ \  \frac{\p^2 F}{\p t \p x_1}(0,c_*;J)= 0,\ \  \frac{\p^2 F}{\p x_1^2}(0,c_*;J)=- \bar{f}'(0)<0.
\end{eqnarray}
According to Taylor's expansion, there holds
\begin{equation}\label{1df6}
F(x_1, t ;J)= \frac{1}{2}(1+\gamma) (t-c_*)^2- \frac{1}{2} \bar{f}'(0) x_1^2+ G(x_1,t-c_*),
\end{equation}
where
\begin{equation*}
|G(x_1, t-c_*)|\leq C_1 (|t-c_*|^3+ |x_1|^3),\ \  \text{for any } |t-c_1|+|x_1|\leq \sigma_1,
\end{equation*}
with some positive constants $C_1$ and $\sigma_1$.

We write $t = c_* + x_1 y$ and reformulate $F (x_1, t; J) =0$ as
\begin{equation*}
y^2-\frac{\bar{f}'(0)}{\gamma+1}  + \frac{2}{\gamma+1} G_1(x_1, y)=0,
\end{equation*}
where
\begin{equation}\label{ldf8}
|G_1(x_1,y)|=\left|\frac{1}{x_1^2}G(x_1,x_1y)\right|\leq C_1(|x_1| |y|^3+ |x_1|),\ \ \text{for any }|x_1|\leq \sigma_1.
\end{equation}
Thus
\begin{equation}\label{ldf9}
H (x_1, y) \co y-\sqrt{\nu^2 - \frac{2}{\gamma+1} G_1(x_1, y)}=0,
\end{equation}
where $\nu = \sqrt{\frac{\bar{f}'(0)}{\gamma+1}}$. According to \eqref{ldf8}, $H (0, \nu) = 0$. Since $\p_y G_1 (x_1, y) = \f 1{x_1} \p_t G (x_1, x_1 y)$, where
\begin{equation*}
\p_t G (x_1, t - c_*) = - \ga \nu x_1 + c_* - \f {\ga J^{\ga - 1}}{t^{\ga }},
\end{equation*}
then
\be\no
&& \f {\p H}{\p y} (0, \nu) = 1 + \f 1 {(\ga + 1) \nu} \lim_{x_1 \to 0^+} (\f 1 {x_1} \f{\p G}{\p t} (x_1, \nu x_1)) \\\no
&& = 1 + \f 1 {(\ga + 1) \nu} [- \ga \nu + \lim_{x_1 \to 0^+} \f {c_* - (\ga J^{\ga - 1}) (\nu x_1 + c_*)^{- \ga }}{x_1}] = 1.
\ee
Thus the implicit function theorem implies that the existence and uniqueness of a smooth function $y=y(x_1)$ defined on the
interval $[-\sigma_2, \sigma_2]$ for some $0 < \si_2 \le \si_1$ such that \eqref{ldf9} holds. Moreover, the function $\bar{u}_1(x_1):= c_* + x_1 y(x_1) \in C^{\oo} ([L_0, L_1])$ is the solution of the equation \eqref{1df2} on the interval $[-\sigma_2, \sigma_2]$ and $(\bar{u} (x_1), \f J{\bar{u} (x_1)})$ is subsonic in $[- \si_2, 0)$, supersonic in $(0, \si_2]$. By the uniqueness of solutions to \eqref{1df2}, one has $( \bar{u}, \bar{\rho}) \equiv (\bar{u}_1 (x_1), \f J{\bar{u}_1 (x_1)})$ on $[- \si_2, \si_2]$. The proof is completed.

\end{proof}

The following Lemma gives the properties of the smooth transonic flows constructed in Proposition \ref{1dfpo}, which play a crucial role in the subsequent stability analysis. Define
\begin{equation}\label{5}
\bar{k}_{11} (x_1) =
\f {c^2(\bar{\rho }) - \bar{u}^2}{c^2 (\bar{\rho })},\ \ \ \bar{k}_1 (x_1) = \f {\bar{f} - (\ga +1) \bar{u} \bar{u}^{\prime}}{c^2 (\bar{\rho })} = \f {\bar{f} (c^2 (\bar{\rho }) + \ga \bar{u}^2)}{c^2 (\bar{\rho }) (c^2 (\bar{\rho }) - \bar{u}^2)}.
\end{equation}

\begin{lemma}\label{bkg-coe}
  For the flows given in Proposition \ref{1dfpo}, there exists a positive number $\kappa_* > 0$ such that
  \begin{equation}\label{6}
  2 \bar{k}_1 (x_1) + (2j -1) \bar{k}_{11}^{\prime} (x_1) \le - \kappa_*, \, j = 0, 1, 2,3, \ \forall x_1\in [L_0,L_1].
  \end{equation}
  Thus there also exists another positive number $d_0 >0$, such that $d(x_1) = 6 (x_1 - d_0) < 0$ and
  \begin{equation}\label{7}
  (\bar{k}_1 + j \bar{k}_{11}^{\prime}) d - \f 12 (\bar{k}_{11}d)^{\prime} \ge 4, \, \forall j = 0, 1, 2, 3,\ \forall x_1\in [L_0,L_1].
  \end{equation}
\end{lemma}
\begin{proof}
  Since for any $x_1 \in [L_0, 0) \cup (0, L_1]$,
\begin{equation*}
2 \bar{k}_1 + (2j -1 ) \bar{k}_{11}^{\prime}
 = \f {\bar{f} (2 c^2 (\bar{\rho }) + ((2\ga + 2)j+ \ga -1) \bar{u}^2 )}{c^2 (\bar{\rho }) (c^2 (\bar{\rho }) - \bar{u}^2)} < 0,
\end{equation*}
and
\begin{equation*}
\lim_{x_1 \to 0} \f {\bar{f}}{c^2 (\bar{\rho }) - \bar{u}^2} = - \sqrt{\f {\bar{f}^{\prime} (0)}{(\ga + 1)c_{*}^2}} < 0.
\end{equation*}
Then a positive constant $\kappa_* > 0$ exists such that \eqref{6} holds. Let $d(x_1) = 6 (x_1 - d_0)$, one can select a large $d_0 > L_1$ such that $d(x_1) < 0$ for every $x_1 \in [L_0, L_1]$ and
\be\no
&& (\bar{k}_1 + j \bar{k}_{11}^{\prime}) d - \f 12 (\bar{k}_{11}d)^{\prime}
= \f 12 (2 \bar{k}_1 + (2j -1) \bar{k}_{11}^{\prime})d - \f 12 d^{\prime} \bar{k}_{11}\\\no
&& = 3 (2 \bar{k}_1 + (2j -1) \bar{k}_{11}^{\prime}) (x_1 - d_0) - 3 \bar{k}_{11} (x_1)\\\no
&& \ge 3 (d_0 - x_1) \kappa_* - 3 \bar{k}_{11} (x_1) \ge 4, \q \forall j = 0, 1, 2, 3.
\ee
\end{proof}

Next, we consider the case where $\bar{f}' (0) = 0$. Suppose that the transonic flow is smooth near $x_1=0$, then $\bar{u}'(0)=0$. This, together with the fact $\bar{u}'(x_1)>0$ for any $x_1\in [L_0,0)\cup (0, L_1]$, implies that $\bar{u}''(0)=0$. The second equation in \eqref{1df1} can be rewritten as $\bar{f}(x_1)= D(x_1) \bar{u}'(x_1)$, where $D(x_1)=\bar{u}- c_*^{\gamma+1}\bar{u}^{-\gamma}$. This gives $D(0)=D'(0)=D''(0)=0, D^{(3)}(0)=(\gamma+1)\bar{u}^{(3)}(0)$ by simple calculations and further $\bar{f}''(0)=\bar{f}^{(3)}(0)=\bar{f}^{(4)}(0)=0$ holds. In addition, one has
\begin{equation*}
\bar{f}^{(5)}(0)=10(\gamma+1)(\bar{u}^{(3)}(0))^2\geq 0.
\end{equation*}

If $\bar{f}^{(5)}(0)>0$, there exists a unique smooth accelerating transonic flow $(\bar{u},\bar{\rho})$ to \eqref{1df} with \eqref{1df1}, \eqref{1df04} and $\bar{f}'(0)=\cdots=\bar{f}^{(4)}(0)=0$. Since the proof is similar, we omit it.

\begin{proposition}\label{1dfzero1}({\bf Smooth transonic flows with zero acceleration at the sonic point: case 1.}) Assume the initial data $(u_0, \rho_0)$ is subsonic, the external force satisfies \eqref{1df04}, \eqref{1df90} and for some nonnegative integer $m\in \mathbb{N}$ the following holds
\begin{equation}\label{1df200}
\bar{f}'(0)=\bar{f}''(0)=\cdots =\bar{f}^{(4m)}(0)=0,\ \ \ \bar{f}^{(4m+1)}(0)>0.
\end{equation}
Then the problem \eqref{1df} has a unique smooth accelerating transonic solution $(\bar{\rho}(x_1), \bar{u}(x_1))\in C^{\infty}([L_0,L_1])$, which is subsonic in $[L_0,0)$, supersonic in $(0, L_1]$ with a sonic state at $x_1=0$. The velocity can be represented as $\bar{u}(x_1)=c_*+ x_1^{2m+1} y(x_1)$ with a positive
smooth function $y\in C^{\infty}([L_0,L_1])$ and
\be\no
&&\bar{u}'(0)=\bar{u}''(0)=\cdots= \bar{u}^{(2m)}(0)=0,\\\no
&&\bar{u}^{(2m+1)}(0)=(2m+1)! y(0)=(2m+1)!\sqrt{\frac{2\bar{f}^{(4m+1)} (0)}{(\gamma+1)(4m+2)!}}>0.
\ee
\end{proposition}

For the case $\bar{f}'(0)=0$, there is another possibility that $\bar{u}$ is only one order differentiable at $x_1=0$, thus $\bar{u}''(x_1)$ is discontinuous at $x_1=0$. Yet, the following existence theorem holds.

\begin{proposition}\label{1dfzero2}({\bf Smooth transonic flows with zero acceleration at the sonic point: case 2.}) Assume the initial data $(u_0, \rho_0)$ is subsonic, the external forcing satisfies \eqref{1df04}, \eqref{1df90} and for any nonnegative integer $m\geq 1\in \mathbb{N}$ the following holds
\begin{equation*}
\bar{f}'(0)=\bar{f}''(0)=\cdots =\bar{f}^{(4m-2)}(0)=0,\ \ \ \bar{f}^{(4m-1)}(0)>0.
\end{equation*}
Then there exists a unique $C^{2m-1,1}$ smooth accelerating transonic solution $(\bar{\rho}(x_1), \bar{u}(x_1))\in C^{2m-1,1}([L_0,L_1])$ to \eqref{1df} such that the
solution is subsonic in $[L_0,0)$, supersonic in $(0, L_1]$ with a sonic state at $x_1=0$. The velocity can be represented as $\bar{u}(x_1)=c_*+ x_1^{2m} y(x_1)$,
where the function $y$ is defined on $[L_0,L_1]$ with a discontinuity at $x_1=0$:
\begin{equation*}
y(x_1)=\begin{cases}
y_-(x_1)<0,\ \ x_1\in [L_0, 0),\\
y_+(x_1)>0,\ \ x_1\in (0, L_1],
\end{cases}
\end{equation*}
with $y_1\in C^{\infty}([L_0,0])$ and $y_2\in C^{\infty}([0,L_1])$.

Furthermore,
\be\no
&&\bar{u}'(0)=\bar{u}''(0)=\cdots= \bar{u}^{(2m-1)}(0)=0,\\\no
&&\bar{u}^{(2m)}(0-)= (2m)! y_-(0)=-(2m)!\sqrt{\frac{2}{(\gamma+1)}\frac{1}{(4m)!} \bar{f}^{(4m-1)}(0)}<0 ,\\\no
&&\bar{u}^{(2m)}(0+)= (2m)! y_+(0)=(2m)!\sqrt{\frac{2}{(\gamma+1)}\frac{1}{(4m)!} \bar{f}^{(4m-1)}(0)}>0.
\ee
\end{proposition}

Finally, if $\bar{f}(x_1)$ or higher even order derivatives of ${\bar f}(x_1)$ is not continuous at $x_1=0$, one has

\begin{proposition}\label{1dfzero3} Assume that $(u_0, \rho_0)$ is subsonic, the external force satisfies \eqref{1df04}, \eqref{1df90} and
\begin{equation*}\begin{cases}
\displaystyle\bar{f}'(0)=\bar{f}''(0)=\cdots =\bar{f}^{(2m-1)}(0)=0,\ \ \ \bar{f}^{(2m)}(0-)=\lim_{x_1\to 0+}\bar{f}^{(2m)}(x_1)<0,\\
\displaystyle\bar{f}^{(2m)}(0+)=\lim_{x_1\to 0+}\bar{f}^{(2m)}(x_1)>0,\ \ \ \text{for some integer $m\geq 0$}.
\end{cases}\end{equation*}
Then there exists a unique $C^{m,\frac12}$ smooth accelerating transonic solution $(\bar{\rho}(x_1), \bar{u}(x_1))\in C^{m,\frac12}([L_0,L_1])$ to \eqref{1df} such that the
solution is subsonic in $[L_0,0)$, supersonic in $(0, L_1]$ with a sonic state at $x_1=0$. The velocity can be represented as $\bar{u}(x_1)=c_*+ |x_1|^{m+\frac12} y(x_1)$,
where the function $y$ is defined on $[L_0,L_1]$ with a discontinuity at $x_1=0$:
\begin{equation*}
y(x_1)=\begin{cases}
y_-(x_1)<0,\ \ x_1\in [L_0, 0),\\
y_+(x_1)>0,\ \ x_1\in (0, L_1],
\end{cases}\end{equation*}
with $y_-\in C^{\infty}([L_0,0])$ and $y_+\in C^{\infty}([0,L_1])$.
\end{proposition}

\subsection{Smooth axisymmetric transonic irrotational flows}\noindent

From now on, the one dimensional smooth transonic flows with positive acceleration given in Proposition \ref{1dfpo} will be called as the background transonic flow. We will focus on the structural stability of the background transonic flows in a cylinder $\Om = \{x =  (x_1, x_2, x_2); L_0 < x_1 < L_1, \, x_2^2 + x_3^2 < 1 \}$ under axisymmetric perturbations of suitable boundary conditions at the entrance and exit of the cylinder.

Introduce the cylindrical coordinates $(x_1, r, \theta)$
\begin{equation*}
x_1 = x_1, \quad r = \sqrt{x_2^2 + x_3^2}, \quad \theta = \arctan \f{x_3}{x_2},
\end{equation*}
and decompose the velocity as ${\bf u}= u_1 {\bf e}_1 + u_r {\bf e}_r + u_{\th} {\bf e}_{\th}$, where
\be\no
{\bf e}_1=(0,0,1)^t,\ \ {\bf e}_r=(\cos\theta,\sin\theta,0)^t, \  \ {\bf e}_{\theta}=(-\sin\theta,\cos\theta,0)^t.
\ee


Consider the axi-symmetric flow, that is
\begin{equation*}
u_{1} = u_{1} ({x_1},r), \q u_r = u_r ({x_1},r), \q u_{\th} =u_{\th} (x_1, r), \q \rho = \rho (x_1,r),\q \Phi=\Phi(x_1,r).
\end{equation*}
Then the steady Euler system \eqref{2deuler-force} can be simplified as
\begin{equation}\label{cy1}
\begin{cases}
  \p_{x_1} (\rho u_{1}) + \p_r (\rho u_r) + \f {\rho u_r}r   = 0, & \\
  \rho ( u_{1} \p_{x_1} + u_r \p_r  ) u_{1}+ \p_{x_1} P(\rho ) = \rho \p_{x_1} \Phi, & \\
  \rho  (u_{1} \p_{x_1} + u_r \p_r ) u_r  + \p_r P(\rho ) = \rho \p_r \Phi , & \\
  \rho ( u_1 \p_{x_1}  + u_r \p_r ) u_{\th } + \f { \rho u_r u_{\th}}r = 0. &
\end{cases}
\end{equation}
The cylinder $\Omega$ is reduced to $D \co \{ (x_1, r), L_0 < x_1 < L_1, 0 < r < 1 \}$.

To avoid a lengthy paper, we concentrate only on the irrotational flows (i.e. $u_{\theta}=\p_{x_1} u_r - \p_r u_1 \equiv 0$ in $D$). The existence and uniqueness of smooth transonic flows with nonzero swirl velocity and vorticity to \eqref{cy1} will be reported in a forthcoming paper. Within the irrotational flows, the axisymmetric Euler equations \eqref{cy1} can be simplified as
\begin{equation}\label{cy2}
\begin{cases}
  \p_{x_1} (\rho u_{1}) + \p_r (\rho u_r) + \f {\rho u_r}r   = 0, & \\
  \p_{x_1} u_r - \p_r u_1 = 0, & \\
  B_0 = \f 12 |\u|^2 + h (\rho) - \bar{\Phi},
\end{cases}
\end{equation}
where we assume $\Phi=\bar{\Phi} (x_1)= \int_{L_0}^{x_1} \bar{f} (\tau) d \tau $ for simplicity. Introduce a potential function $\phi$ such that
\begin{equation}\no
u_1 = \p_{x_1} \phi, \ \ \  u_r = \p_r \phi, \ \ \ \phi(L_0,0)=0.
\end{equation}
The potential function $\bar{\phi}(x_1)$ for the background transonic flows is $ \bar{\phi}(x_1)=\int_{L_0}^{x_1} \bar{u}(s) ds$. The last equation in \eqref{cy2} implies that
\begin{equation}\no
\rho = \rho (|\nabla \phi|^2 , \bar{\Phi}) = \b[ \f{\ga -1}{\ga } (B_0 + \bar{\Phi} - \f 12 |\nabla \phi|^2) \b]^{\f 1{\ga -1}}, \ \ \nabla\phi=(\p_{x_1}\phi, \p_r\phi)
\end{equation}
and
\begin{equation*}
 c^2 (\rho) =  \ga \rho^{\ga -1} = (\ga -1 ) (B_0 + \bar{\Phi} - \f 12 |\nabla \phi|^2).
\end{equation*}
Thus the steady Euler system (\ref{cy2}) is equivalent to
\begin{equation}\label{cy3}
(c^2 (\rho) - (\p_{x_1} \phi)^2) \p_{x_1}^2 \phi +  (c^2 (\rho) - (\p_{r} \phi)^2) \p_r^2 \phi - 2\p_{x_1} \phi \p_r\phi \p_{x_1r}^2 \phi
+\f {c^2(\rho)}{r} \p_r\phi = - {\bar f}(x_1)\p_{x_1} \phi .
\end{equation}

We prescribe the following boundary conditions
\begin{equation}\label{2dbcs1}
\begin{cases}
  \p_r \phi (L_0, r) = \eps h_1 (r), &  \forall r \in [0, 1], \\
  \p_r \phi (x_1, 1) = \p_r \phi (x_1, 0) = 0, & \forall x_1 \in [L_0, L_1],
\end{cases}
\end{equation}
Here $ \f {h_1(r)}r \in H^2_r ([0,1])$, $h_1' (r) \in H^2_r ([0, 1])$ satisfy the following compatibility conditions
\begin{equation}\label{cp2}
h_1 (0) =  h_1'' (0) = 0,
\end{equation}
and for some positive constant $\beta_0\in (0,1)$,
\begin{equation}\label{bch}
h_1 (r) \equiv 0, \ \ \forall r \in [1 - \beta_0, 1].
\end{equation}

Note here that we pose some restrictions on the flow angle at the inlet (i.e. the first equation in \eqref{2dbcs1}), which are physically acceptable and experimentally controllable. The second condition in \eqref{2dbcs1} is the slip boundary condition on the wall. From a mathematical perspective, these two boundary conditions are also admissible for the linearized potential equation (see Lemma \ref{H1estimate}). No boundary conditions need to be specified at the exit of the cylinder.

Before stating the main theorem, we introduce the following weighted norms on $D$:
\be\no
\| \psi \|_{L^2_r (D)}^2 & \co & \int_{L_0}^{L_1} \int_0^1 |\psi (x_1, r)|^2 r dr dx_1, \\\no
\| \psi \|_{H^1_r (D)}^2 & \co & \| \psi \|_{L^2_r (D)}^2 + \|\na \psi \|_{L^2_r (D)}^2 , \ \ \ \nabla=(\p_{x_1}, \p_r), \\\label{norm}
\| \psi \|^2_{H^2_r (D)} & \co & \| \psi \|_{H^1_r (D)}^2 + \|\na^2 \psi \|_{L^2_r (D)}^2 + \b\|\f 1r \p_r \psi \b\|_{L^2_r (D)}^2, \\\no
\| \psi \|^2_{H^3_r (D)} & \co & \| \psi \|^2_{H^2_r (D)} + \|\na^3 \psi \|_{L^2_r (D)}^2 + \b\|\na \left(\f 1r \p_r \psi \right) \b\|_{L^2_r (D)}^2, \\\no
\| \psi \|^2_{H^4_r (D)} & \co & \| \psi \|^2_{H^3_r (D)} + \|\na^4 \psi \|_{L^2_r (D)}^2 + \b\|\na^2 \left(\f 1r \p_r \psi\right) \b\|_{L^2_r (D)}^2 + \b\| \f 1r \p_r \left(\f 1r \p_r \psi \right) \b\|_{L^2_r (D)}^2.
\ee
Note that for any $\psi\in H^4_r(D)$, one has $\p_r\left(\f 1r \p_r \psi \right)\in L_r^2(D)$. However, $\frac{1}{r}\p_r^2 \psi\not\in L_r^2(D)$ and $\p_r\psi\not\in H_r^3(D)$ since in general $\p_r^2\psi(x_1,0)\neq 0$ for any $x_1\in [L_0,L_1]$.

The following theorem establishes the structural stability of the background transonic flows under the axisymmetric perturbations as above.
\begin{theorem}\label{irro}
  Let $(\bar{u}, \bar{\rho })$ be the background transonic flow with positive acceleration given in Proposition \ref{1dfpo}. Assuming that $\ga >1$, $ \f {h_1 (r)}r \in H^2_r ([0, 1])$, $h_1' \in H^2_r ([0, 1])$ satisfy \eqref{cp2} and \eqref{bch}, there exist positive constants $C_0,\epsilon_0$ depending on the background flow and the boundary datum $h_1$, such that for any $0 < \epsilon < \epsilon_0$, the problem \eqref{cy3} with \eqref{2dbcs1} has a unique smooth axisymmetric transonic irrotational solution $\phi\in H^4_r(D)$ with
  \begin{equation}\label{ir1}
  \| \phi - \bar{\phi} \|_{H^4_r (D)}  \le C_0 \epsilon.
  \end{equation}
  Moreover, all the sonic points form a $C^1$ smooth axisymmetric front $x_1 = \xi (r) \in C^1 ([0, 1])$, which is closed to the background sonic disc $x_1=0$ in the sense that
  \begin{equation}\label{so1}
  \| \xi(r) \|_{C^1 ([0, 1])} \le C_0 \eps.
  \end{equation}
\end{theorem}

\begin{remark}
{\it To the best of the authors' knowledge, Theorem \ref{irro} is the first existence and uniqueness result on the smooth axisymmetric transonic flows.
}
\end{remark}


There have been many studies on smooth subsonic-sonic and transonic flows. For flows past a profile, Gilbarg and Shiffman \cite{gs54} showed that for a smooth subsonic-sonic flow past a profile, the sonic points must occur on the profile. Morawetz in \cite{mora56} proved that smooth transonic flows with a supersonic bubble attached to the profile does not exist in general and is unstable with respect to small changes in the shape of the profile. Kuzmin \cite{Kuzmin2002} had investigated the perturbation problem of the von Karman equation which can be used to describe an accelerating transonic flow when the velocity is close to the sound speed and the flow angle is small. In a converging nozzle with straight solid walls, the existence and uniqueness of a two dimensional irrotational continuous subsonic-sonic flow with a singularity at the sonic point was proved in \cite{WX2013} and the acceleration blows up at the sonic curve.



In his studies of possible continuation of a flow across a sonic curve when a subsonic-sonic flow was assumed to be given, Bers \cite{bers58} found that sonic points should be classified into two classes: exceptional and nonexceptional. A sonic point in a $C^2$ transonic flow is exceptional if and only if the velocity is orthogonal to the sonic curve at this point. Bers further proved that if there are no exceptional points on a sonic curve, then the flow can be continued locally as a supersonic flow without discontinuity in a unique way across the sonic curve. Wang and Xin \cite{WX15,WX2016,WX2019,WX2020} established the existence and uniqueness of smooth transonic irrotational flows of Meyer type on two dimensional De Laval nozzles, and all the sonic points are exceptional and must locate at the throat. See some further developments in this direction in \cite{wang19,wang22}. Weng and Xin \cite{WX23} proved the existence and uniqueness of smooth transonic flows with nonzero vorticity in de Laval nozzles for a quasi two dimensional steady Euler flow model, where one of the key issues is the analysis of a linear second order elliptic-hyperbolic mixed equation of Keldysh type. One may see the related stability analysis for smooth transonic flows to the steady Euler-Poisson system \cite{bdx}.

Courant and Friedrichs \cite[Section 104]{Courant1948} had found a class of spiral flows which may change smoothly from subsonic to supersonic or vice verse. The authors in \cite{WXY21a} had further studied this class of radially symmetric transonic flows with nonzero angular velocity in an annulus and gave a complete classification of all possible flow patterns for inviscid transonic flow with or without shocks. Different from \cite{WX2019,WX2020,WX23}, the sonic points of the smooth transonic spiral flow constructed in \cite{WXY21a} are all nonexceptional and noncharacteristic degenerate. The existence and uniqueness of smooth transonic spiral flows with nonzero vorticity satisfying suitable boundary conditions were established in \cite{WXY21b} by the analysis of a linear second order elliptic-hyperbolic mixed equation of Tricomi type and the deformation-curl decomposition to the steady Euler equations \cite{WengXin19,weng2019}.

Now we explain some new ingredients in our proof of Theorem \ref{irro}. We basically follow the strategy developed in \cite{WX23}, yet there are several important differences. The linear second order elliptic-hyperbolic mixed differential equation in \cite{WX23} is as follows
\begin{equation}\label{linear-2d}
\begin{cases}
   k_{11} \p_{x_1}^2 \phi + 2k_{12}\p_{x_1x_2}^2 \phi +\p_{x_2}^2 \phi + k_1 \p_{x_1} \phi = G_0(x_1 , x_2), & (x_1 , x_2) \in (L_0, L_1) \times (-1, 1), \\
  \phi (L_0, x_2) = 0, & \forall x_2 \in [-1, 1], \\
  \p_{x_2} \phi (x_1, \pm 1) = 0, & \forall x_1 \in [L_0, L_1].
\end{cases}
\end{equation}
While the linear second order elliptic-hyperbolic mixed differential equation in current case reads as
\begin{equation}\label{linear-irro}
\begin{cases}
  k_{11}  \p_{x_1}^2 \psi + 2 k_{12} \p_{x_1 r}^2 \psi + \p_r^2 \psi + \f 1r \p_r \psi + k_1 \p_{x_1} \psi + k_2 \p_r \psi = F_0(x_1,r), &  (x_1, r) \in (L_0, L_1) \times (0, 1), \\
 \psi (L_0, r) =  0, & \forall r \in [0, 1],   \\
  \p_r \psi (x_1, 0) = \p_r \psi (x_1, 1) = 0,  & \forall x_1 \in [L_0, L_1].
\end{cases}
\end{equation}
Both the coefficients $k_{11}$ in \eqref{linear-2d} and \eqref{linear-irro} change signs when crossing the sonic front, thus both \eqref{linear-2d} and \eqref{linear-irro} are Keldysh type mixed equations. Since \eqref{linear-irro} contains a singular operator $\f 1r \p_r$, several new difficulties must be overcome:
\begin{enumerate}[(i)]
  \item To deal with the artificial singularities near the symmetry axis, we introduce the weighted norms $\|f\|_{H^m_r(D)} (m=0,\cdots,4)$ for a function $f(x_1,r)$ defined on $D$, which correspond to the standard Sobolev norms $\|\check{f}\|_{H^m(\Omega)}(m=0,\cdots,4)$ for the function $\check{f}(x_1,x')=f(x_1,|x'|)$ defined on $\Omega$. Due to the different roles of $x_1,r$ and the singular operator $\frac{1}{r}\p_r$ appeared in the norms $\|\cdot\|_{H^m_r(D)} (m=2,3,4)$, the coefficients $k_{12}$ and $k_2$ in \eqref{linear-irro} behave quite different from $k_{11}$ and $k_1$ (See Lemma \ref{coe-estimate}). The estimates in Lemma \ref{coe-estimate} are of great importance in the a priori estimates to \eqref{linear-irro}.

  \item In order to show the existence of the $H^2_r(D)$ strong solution of \eqref{linear-irro}, similar to \cite{WX23}, we modify the strategy in \cite{Kuzmin2002} and add a third order dissipation term $\si \p_{x_1}^3 \psi$ and two additional boundary conditions $\p_{x_1}^2 \psi (L_0, \cdot) = \p_{x_1}^2 \psi (L_1, \cdot ) = 0$, instead of the one $\p_{x_1} \psi (L_0, \cdot ) = \p_{x_1} \psi (L_1, \cdot) = 0$ used in \cite{Kuzmin2002}. The positive acceleration of the one dimensional transonic solutions plays an important role in searching for an appropriate multiplier for the equation \eqref{linear-irro}. To obtain the $H^2_r(D)$ estimate, it is also crucial to observe that multiplying  \eqref{linear-irro} by $-d (x_1) \p_{x_1} (\p_r^2 \psi + \f 1r \p_r \psi)$ with $d(x_1)= 6(x_1-d_0)$ and integrating by parts will yield the estimate $\|\p_r^2 \psi \|_{L^2_r (D)} + \| \f 1r \p_r \psi \|_{L^2_r (D)}$. These enable us to obtain a uniform $H^2_r(D)$ estimate with respect to $\sigma$. The Galerkin method with an orthonormal basis for $L^2_r(D)$ consisting of all the eigenfunctions of the operator $\p_r^2+\frac{1}{r}\p_r$ with Neumann boundary condition is used to construct the approximate solutions.

  \item For the $H^4_r$ estimate of $\psi$ on elliptic region, the symmetric extension technique used in \cite{WX23} can not be applied to the problem \eqref{linear-irro} due to the term $\frac{1}{r}\p_r \psi$ and the fact that $\p_r^3\psi(x_1,1)=0$ for any $x_1\in [L_0,L_1]$ does {\it not} hold in general. The $H^4_r$ estimates near the nozzle wall $r = 1$ and near the axis $r = 0$ will be discussed separately. Since $\p_r\psi$ satisfies homogeneous boundary conditions at the entrance and nozzle wall, one can use \cite[Theorem 3.1.3.1]{Grisvard11} to derive the $H^3_r$ estimate near the nozzle wall $r = 1$. For the $H^4_r$ estimate near $r = 1$, a key observation is the function $\hat{w}_2:= \p_{x_1}^2 \psi - \f 1{k_{11}} (F_0 - k_1 \p_{x_1} \psi)$ satisfies the homogeneous mixed boundary conditions on the entrance and nozzle wall (see \eqref{w21}), then the symmetric extension and the cut-off techniques can be applied to derive the estimate. To obtain the $H_r^4$ estimate near the axis, we transform the problem \eqref{linear-irro} back to the Cartesian coordinates and the singularities disappear. The standard elliptic theory in \cite{gt} yields the $H^4_r$ estimate near $r = 0$. The $H^4_r$ estimate of $\psi$ on the transonic region adopts the method similar to \cite{WX23} by extending the problem to a larger region so that the auxiliary equation is elliptic at the exit.
      Since the properties of the coefficients $k_{11}, k_{12}, k_1$ and $k_2$ in \eqref{linear-irro} are more complicated than those for \eqref{linear-2d}, the $H^4_r(D)$ norm estimate for $\psi$ in the transonic region is more involved than those in \cite{WX23}.
\end{enumerate}

The structure of the paper will be arranged as follows. In section \ref{H2estimate}, we use the multiplier method and the Galerkin's method to prove the existence and uniqueness of the $H^2_r$ strong solution to the linearized mixed potential equation. In section \ref{h4estimate}, we prove higher order energy estimates in the elliptic and the transonic domain, respectively. The proof of Theorem \ref{irro} by the fixed point theorem is also accomplished in this section. The appendix shows some estimates to the coefficients of the linearized mixed equation.

\section{The $H^2_r$ strong solution to the linear second order elliptic-hyperbolic mixed equation}\label{H2estimate} \noindent



In this section, we first linearized the problem \eqref{cy3} with \eqref{2dbcs1}. We rewrite \eqref{cy3} as
\begin{equation}\no
\f {c^2 (\rho) - (\p_{x_1} \phi)^2}{c^2 (\rho) - (\p_r \phi )^2 } \p_{x_1 }^2 \phi
 -  \f {2 \p_{x_1} \phi \p_r \phi}{c^2 (\rho) - (\p_r \phi )^2 } \p_{x_1 r}^2 \phi
 +  \p_r^2 \phi
 + \f {c^2 (\rho)  \p_r \phi }{r (c^2 (\rho) - (\p_r \phi )^2 )}+ \f {{\bar f}(x_1)\p_{x_1} \phi}{c^2 (\rho) - (\p_r \phi )^2}=0.
\end{equation}
Denote $\psi_1= \phi - \bar{\phi}$, then $\psi_1$ satisfies
\begin{equation}\label{2d-force-cy17}
\begin{cases}
  k_{11} (\na \psi_1 ) \p_{x_1}^2 \psi_1 + 2 k_{12} (\na \psi_1 ) \p_{x_1 r}^2 \psi_1 + \p_r^2 \psi_1 + \f 1r \p_r \psi_1 + k_1 (\na \psi_1 ) \p_{x_1} \psi_1 + k_2 (\na \psi_1 ) \p_r \psi_1 = F(\na \psi_1), \\
  \p_r \psi_1 (L_0, r) = \epsilon h_1 (r), \q \forall r \in [0, 1], \q  \psi_1 (L_0, 0) = 0, \\
  \p_r \psi_1 ({x_1},1) = \p_r \psi_1 ({x_1},0 ) = 0, \q \forall {x_1} \in [L_0, L_1],
\end{cases}
\end{equation}
where
\begin{equation}\label{coe}
\begin{cases}
  k_{11} (\na \psi_1 ) = \f {c^2 (\rho) - (\p_{x_1} \psi_1 + \bar{u})^2}{c^2 (\rho) - (\p_r \psi_1 )^2}, \,
  k_{12} (\na \psi_1 ) = - \f {(\p_{x_1} \psi_1 + \bar{u}) \p_r \psi_1 }{c^2 (\rho) - (\p_r \psi_1 )^2}, \\
  k_1 (\na \psi_1 ) = \f {\bar{f} - (\ga +1) \bar{u} \bar{u}'}{c^2 (\rho) - (\p_r \psi_1 )^2},  \,
  k_2 (\na \psi_1 ) = \f {(\p_r \psi_1)^2  }{r (c^2 (\rho) - (\p_r \psi_1 )^2 )}, \\
  F (\na \psi_1 ) =   \f {\bar{u}^{\prime}}{ c^2 (\rho) - (\p_r \psi_1 )^2 } [
\f {\ga +1}2 (\p_{x_1} \psi_1)^2 + \f{\ga -1}2 (\p_r \psi_1)^2
 ],\\
c^2 (\rho) = (\ga -1) [B_0 + \Phi - \f 12 ((\p_{x_1} \psi_1 + \bar{u})^2 + (\p_r \psi_1)^2 )].
\end{cases}
\end{equation}

Choose a monotonic decreasing function $\eta_0 (x_1) \in C^{\oo} ([L_0, L_1])$ satisfying
\begin{equation}\label{eta}
\eta_0 (x_1) = \begin{cases}
    1, & L_0 \le x_1 \le \f {15}{16} L_0 , \\
    0, & \f 78 L_0 \le x_1 \le L_1.
    \end{cases}
\end{equation}
Set $\psi(x_1,r) = \psi_1 (x_1 ,r) -  \epsilon \psi_0(x_1,r)$, where $\psi_0(x_1,r)=\eta_0 (x_1)
\int_0^r  h_1 (t) dt$. Then $\psi$ satisfies
\begin{equation}\label{linearized2}
\begin{cases}
  k_{11} (\na \psi + \eps \na \psi_0 ) \p_{x_1}^2 \psi + 2 k_{12} (\na \psi + \eps \na \psi_0 ) \p_{x_1 r}^2 \psi + \p_{r}^2 \psi + \f 1r \p_r \psi  & \\
  \q\q  + k_1 (\na \psi + \eps \na \psi_0 ) \p_{x_1} \psi+ k_2 (\na \psi + \eps \na \psi_0 ) \p_r \psi  = F_0 (\na \psi), & (x_1, r ) \in D, \\
  \psi (L_0, r) =0, & \forall r \in [0, 1], \\
  \p_r \psi (x_1, 1 ) = \p_r \psi (x_1, 0 ) = 0, & \forall x_1 \in [L_0, L_1],
\end{cases}
\end{equation}
where
\be\label{f0}
&&F_0 (\na \psi) \co  F (\na \psi + \eps \na \psi_0 )+ \mc{F}(\nabla \psi),\\\no
&&\mc{F}(\nabla \psi)\co - \epsilon(k_{11} (\na \psi + \eps \na \psi_0 ) \p_{x_1}^2 \psi_0
+ k_1 (\na \psi + \eps \na \psi_0 )\p_{x_1}\psi_0)
\\\no
&&\q\q\q \q- \eps  (2 k_{12} (\na \psi + \eps \na \psi_0 ) \p_{x_1r}^2 \psi_0 + k_2 (\na \psi + \eps \na \psi_0 ) \p_r \psi_0) - \epsilon(\p_r^2+\frac{1}{r}\p_r)\psi_0.
\ee

Denote the function space $\Sigma_{\de_0}$ by consisting of the functions $\psi \in H^4_r (D)$ satisfying $\| \psi \|_{H^4_r (D )} \le \de_0$ with $\de_0 > 0$ to be specified later and the compatibility conditions
\begin{equation}\no\begin{cases}
\psi (L_0, r) = 0,\ \ &\forall r\in [0,1],\\
\p_r \psi(x_1, 1) = \p_r\psi(x_1,0)=\p_r^3 \psi (x_1, 0) =0, \ \ &\forall x_1\in [L_0,L_1].
\end{cases}\end{equation}

For any given $\hat{\psi} \in \Sigma_{\de_0} $, we define an operator $\mc{T}$ mapping $ \hat{\psi } \in \Sigma_{\de_0}$ to $ \psi \in \Sigma_{\de_0} $, where $\psi$ is obtained by solving the following linear mixed type equation
\begin{equation}\label{linearized1}
\begin{cases}
  \mc{L} \psi \co k_{11} (\na \hat{\psi}+\epsilon \na \psi_0) \p_{x_1 }^2 \psi + 2 k_{12} (\na \hat{\psi}+\epsilon \na \psi_0) \p_{x_1 r}^2 \psi  + \p_{r}^2 \psi + \f 1r \p_r \psi \\
  \quad\quad\quad+ k_1 (\na \hat{\psi }+\epsilon \na \psi_0) \p_{x_1} \psi + k_2 (\na \hat{\psi }+\epsilon \na \psi_0) \p_r \psi = F_0 (\na \hat{\psi }), \\
  \p_r \psi (L_0, r) = \eps h_1 (r), \q \forall r\in [0, 1], \q  \psi (L_0, 0) =0, \\
  \p_r \psi ({x_1}, 1) = \p_r \psi ({x_1}, 0) = 0, \q \forall {x_1} \in [L_0, L_1].
\end{cases}
\end{equation}

We first present the following important properties for the coefficients in \eqref{linearized1}, which play a crucial role in the subsequent stability analysis.

\begin{lemma}\label{coe-estimate}
  Suppose that $\hat{\psi} \in \Sigma_{\de_0}$, for sufficiently small constants $\epsilon$ and $\delta_0$, there exists a constant $C_0$ depending only on the background flows and the functions $h_0$ such that $k_{1i} (\na \hat{\psi}+\epsilon \na \psi_0)$, $k_i (\na \hat{\psi}+\epsilon \na \psi_0)$ for $i= 1,2 $ and $F (\na \hat{\psi}+\epsilon \na \psi_0)$, $\mc{F}(\na \hat{\psi})$ satisfy the following estimates
  \be\label{coe11}
  &&
  \| k_{11}   - \bar{k}_{11}  \|_{H^3_r (D )} + \| k_1   - \bar{k}_1 \|_{H^3_r (D )} \leq C_0(\epsilon + \|\psi\|_{H_r^4(D)})\leq C_0 (\eps + \de_0), \\\label{coe19}
  && \| \p_r k_{11} \|_{L^{\oo}_r (D)} + \| \p_r k_1 \|_{L^{\oo}_r (D)} \leq C_0(\epsilon + \|\psi\|_{H_r^4(D)})\leq C_0 (\eps + \de_0),
  \\\label{coe12}
  &&
  \| \f {k_{12}}r   \|_{H^2_r (D)} + \| k_{12}   \|_{L^{\oo}_r (D)} + \| \p_r k_{12}   \|_{H^2_r (D)}
  + \| \p_{x_1} k_{12} \|_{L^{\oo}_r (D)} \leq C_0(\epsilon + \|\psi\|_{H_r^4(D)}) \le C_0 ( \eps + \de_0) , \\\label{coe13}
  &&
  \| \f {k_2}r  \|_{H^2_r (D)} + \| k_2   \|_{L^{\oo}_r (D)}  + \| \p_r k_2 \|_{H^2_r (D)} + \| \p_{x_1} k_2\|_{L^{\oo}_r (D)}\leq C_0(\epsilon + \|\psi\|_{H_r^4(D)})^2  \le C_0 ( \eps + \de_0)^2 ,  \\\label{coe10}
   &&
   \| \p_{x_1}^2 k_{12} \|_{L^4_r (D)} + \| \na \p_{x_1}^2 k_{12} \|_{L^2_r (D)}\leq C_0(\epsilon + \|\psi\|_{H_r^4(D)}) \le C_0 ( \eps + \de_0), \\\label{coe9}
  &&
  \| \p_{x_1}^2 k_2 \|_{L^2_r (D)} + \| \na \p_{x_1}^2 k_2 \|_{L^2_r (D)}\leq C_0(\epsilon + \|\psi\|_{H_r^4(D)})^2 \le  C_0 ( \eps + \de_0)^2,
  \\\label{coe14}
  &&
  \| F(\na \hat{\psi }+\epsilon \na \psi_0)\|_{H^3_r (D )} + \|\mc{F}(\na \hat{\psi}) \|_{H^2_r (D)} \le C_0(\eps + ( \eps + \de_0)^2),
  \ee
   and the compatibility conditions
  \be\label{coe15}
  &&
  k_{12} (\na \hat{\psi} + \eps \na \psi_0)(x_1, 0/1)  = \p_{r}^2\{ k_{12} (\na \hat{\psi} + \eps \na \psi_0)\}(x_1, 0/1) =0,\ \forall x_1\in [L_0,L_1],  \\\label{coe16}
  &&
  \p_r \{ k_{11} (\na \hat{\psi} + \eps \na \psi_0)\}(x_1, 0/1)  = \p_r \{k_1 (\na \hat{\psi} + \eps \na \psi_0)\}(x_1, 0/1) = 0,\ \forall x_1\in [L_0,L_1],\\\label{coe17}
  &&
  k_{12} (\na \hat{\psi} + \eps \na \psi_0 )(L_0, r) = 0, \ \forall r\in [1 - \beta_0, 1] ,\\\label{coe18}
  &&
  k_2 (\na \hat{\psi} + \eps \na \psi_0)(x_1, 0/1) = 0,\q  \p_r\{ F_0 (\na \hat{\psi})\}(x_1, 0/1)  = 0, \ \forall x_1\in [L_0,L_1].
\ee
\end{lemma}
The proof of Lemma \ref{coe-estimate} is nontrivial and quite long, and will be given in the Appendix \S\ref{appendix}.
\begin{remark}
{\it The weighted norm in $H^m_r(D) (m=2,3,4)$ involves the singular operator $\frac{1}{r}\p_r$, thus the coefficients $k_{12}(\na \hat{\psi }), k_2(\na \hat{\psi })$ behave quite different from $k_{11}(\na \hat{\psi }), k_1(\na \hat{\psi })$, these tedious issues essentially come from the artificial singularity near the axis.}
\end{remark}

The function $\hat{\psi}+\epsilon \psi_0$ can be approximated by a sequence of $C^4(\overline{D})$ smooth functions $\{\psi_n\}_{n\geq 1}$ such that $k_{11}(\nabla \psi_n), k_1(\nabla \psi_n), F_0(\nabla \psi_n)$ converge to $k_{11}, k_1, F_0(\nabla \hat{\psi}) \in H^3_r (D)$ in $H^3_r (D)$ which also satisfy the compatibility conditions in \eqref{coe16} and \eqref{coe18}. And $\p_r \{k_{12}(\nabla \psi_n)\}, \frac{k_{12}(\nabla \psi_n)}r, \p_r \{k_{2}(\nabla \psi_n)\} , \frac{k_{2}(\nabla \psi_n)}r \in C^2 (\overline{D}) $ converge to $\p_r k_{12}, \frac {k_{12}}r, \p_r k_2, \frac{k_2} r$ in $H^2_r (D)$ satisfying the compatibility conditions in \eqref{coe15},\eqref{coe17} and \eqref{coe18}. Therefore, in the following we assume that $k_{11}$, $k_1 \in C^3 (\overline{D})$, $\p_r k_{12}$, $\f {k_{12}}r$, $\p_r k_2$, $\f {k_2} r \in C^2 (\overline{D})$ satisfy the estimates \eqref{coe11}-\eqref{coe14} and the compatibility conditions \eqref{coe15}-\eqref{coe18}.



The following Lemma gives the $H^1_r (D)$ energy estimate for \eqref{linearized2}. The proof is based on an old idea by Friedrichs \cite{Friedrichs1958} to find a multiplier for \eqref{linearized2} and show that the boundary conditions posed in \eqref{linearized2} are admissible. The properties of the background flow proved in Lemma \ref{bkg-coe} play a crucial role.

\begin{lemma}\label{H1estimate}
There exist $\epsilon_*>0$ and $\de_* >0$ depending only on the background flow and the boundary data, such that if $0<\epsilon<\epsilon_*$ and $0 < \de_0 \le \de_*$ in Lemma \ref{coe-estimate}, the classical solution to \eqref{linearized2} satisfies 
\begin{equation}\label{H1}
\iint_D ( |\psi (x_1,r)|^2 + |\na \psi (x_1, r)|^2 ) r dr dx_1  \le C_* \iint_D F_0^2 r dr  dx_1,
\end{equation}
where $C_*$ depends only on the $H^3_r (D )$ norms of $k_{11}$, $k_1$, the $H^2_r (D)$ norms of $\p_r k_{12}$, $\f {k_{12}}r$ and the $L^{\oo}$ norm of $k_2$.
\end{lemma}
\begin{proof}
Let $d (x_1) = 6 (x_1 - d_0) < 0$ for $x_1 \in [L_0, L_1]$. Multiplying $\eqref{linearized2}_1$ by $d (x_1) \p_{x_1} \psi $ and integration by parts in $ D $ lead to
\be\no
&& \iint_D d(x_1) \p_{x_1} \psi F_0  r dr dx_1  = \iint_D d(x_1) \p_{x_1} \psi \mc{L} \psi r dr dx_1 \\\no
&& = \iint_D \b[(d k_1 - \f 12  \p_{x_1} (d k_{11}) - d \p_r k_{12} - d \f { k_{12}}r) (\p_{x_1} \psi)^2 + \f 12 d^{\prime} (\p_r \psi)^2  +  d k_2 \p_{x_1} \psi \p_r \psi \b] r dr d{x_1}\\\label{4}
&& \q  + \f 12 \int_0^1[d(x_1) k_{11} (\p_{x_1} \psi )^2 - d (\p_r \psi )^2 ] \b|_{x_1 = L_0}^{L_1} r dr.
\ee

According to \eqref{6}-\eqref{7}, there exist $\epsilon_*>0,\de_* >0$ such that if $0<\epsilon<\epsilon_*, 0<\de_0 \le \de_*$ in Lemma \ref{coe-estimate}, there holds
\be\no
&& k_1 d - \f 12 \p_{x_1} (d k_{11}) - d \p_r k_{12} - \f {d k_{12}}r \\\no
&& = \bar{k}_1 d - \f 12 (d \bar{k}_{11})^{\prime} + (k_1 - \bar{k}_1) d - \f 12 \p_{x_1} ((k_{11} - \bar{k}_{11})d) - d \p_r k_{12} - \f {d k_{12}}r \\\no
&& \ge 4 - \| (k_1 - \bar{k}_1) d \|_{L^{\oo}_r (D) } - \f 12 \| \p_{x_1} ((k_{11} - \bar{k}_{11})d) \|_{L^{\oo}_r (D) } - \| d\p_r k_{12} \|_{L^{\oo}_r (D) } - \b\| \f {d k_{12}}r \b\|_{L^{\oo}_r (D) } \\\no
&& \ge 3 , \q \forall (x_1 ,r) \in D ,\\\no
 &&  \| d k_2 \|_{L^{\oo}_r (D ) } \le C_0 ( \eps + \de_0),
\ee
due to the Sobolev embedding $H^3_r (D ) \subset C^{1}(\ol{D})$ and $H^2_r (D ) \subset L^{\oo }_r (D ) $. Note also $d (L_0) < 0$ and $d(L_1) < 0$, while $k_{11} (L_0, r) > 0$ and $k_{11} (L_1, r) < 0$, it can be inferred from \eqref{4} that
\begin{equation}\no
\iint_D d (x) \p_{x_1} \psi F_0 r  dr d x_1 \ge 2 \iint_D |\na \psi |^2 r dr dx_1 .
\end{equation}
Since $\psi(L_0, r) = 0$, the estimate \eqref{H1} is obtained.
\end{proof}


In order to show the existence and uniqueness of strong solutions to \eqref{linearized2}, we investigate the following singular perturbation system of \eqref{linearized2}, which includes an additional third order dissipation term and another two boundary conditions.
\begin{equation}\label{ap1}
\begin{cases}
  \mc{L}_{\si} \psi^{\si} = \si \p_{x_1}^3 \psi^{\si } + k_{11} \p_{x_1 }^2 \psi^{\si } + 2 k_{12} \p_{x_1 r}^2 \psi^{\si } + \p_{r}^2 \psi^{\si } + \f 1r \p_r \psi^{\si } + k_1 \p_{x_1} \psi^{\si } + k_2 \p_r \psi^{\si }  = F_0 (x_1,r),  \\
  \p_{x_1}^2 \psi^{\si } (L_0, r) = \p_{x_1}^2 \psi^{\si } (L_1, r) =0, \q \forall r \in [0, 1], \\
  \p_r \psi^{\si } (x_1, 0) = \p_r \psi^{\si } (x_1,1) =0, \q \forall x_1 \in [L_0, L_1],\\
   \psi^{\si } (L_0, r)  = 0, \q \forall r \in [0, 1].
\end{cases}
\end{equation}
This idea follows essentially from \cite{Kuzmin2002}. It should be noted that the supplementary boundary conditions $\p_{x_1 }^2 \psi^{\si } (L_0, r) = \p_{x_1}^2 \psi^{\si } (L_1, r) =0$ we employed are different from $\p_{x_1 } \psi^{\si } (L_0, r) = \p_{x_1} \psi^{\si } (L_1, r) =0$ in \cite{Kuzmin2002}, so that one can obtain a uniform $H^2_r $ estimate to \eqref{ap1} with respect to $\si$.  Thus the existence of a unique $H^2_r $ strong solution $\psi$ to \eqref{linearized2} can be derived directly by taking a limit $\si\to 0$.

The following lemma gives the $H^2_r$ estimate for the classical solutions of \eqref{ap1} uniformly in $\si$.
\begin{lemma}\label{H1estimate-ap}
Under the same assumptions in Lemma \ref{H1estimate}, the classical solution to \eqref{ap1} satisfies
  \begin{equation}\label{H1-ap1}
 \si \iint_D |\p_{x_1}^2 \psi^{\si }|^2 r dr dx_1 + \iint_D ( |\psi^{\si }|^2 + |\na \psi^{\si } |^2 ) r dr dx_1 \le C_* \iint_D F_0^2 r dr dx_1,\end{equation}
 \begin{equation}\label{H2-ap}
 \iint_D \b( |\na^2 \psi^{\si }|^2 + |\f 1r \p_r \psi^{\si }|^2 \b)r dr dx_1 \le C_* \iint_D (F_0^2 + |\na F_0|^2 ) r dr dx_1,
  \end{equation}
where $C_*$ depends only on the $H^3_r (D )$ norms of $k_{11}$, $k_1$, the $H^2_r (D)$ norms of $\p_r k_{12}$, $\f {k_{12}}r$, $\p_r k_2$ and the $L^{\oo}_r (D)$ norm of $k_{12}$, $\p_{x_1} k_{12}$ and $k_2$.

\end{lemma}
\begin{proof}
  For the sake of simplicity, we will omit the superscript $\si $ from the following argument. By choosing the same multiplier as in Lemma \ref{H1estimate}, one derives
  \be\no
  && \iint_D d(x_1) \p_{x_1} \psi F_0 r dr d{x_1}
  = \iint_D d (x_1) \p_{x_1} \psi \mc{L} \psi r dr d{x_1}\\\no
&& = \iint_D \b[(d k_1 - \f 12  \p_{x_1} (d k_{11}) - d \p_r k_{12} - d \f {k_{12}}r ) (\p_{x_1} \psi)^2 + \f 12 d^{\prime} (\p_r \psi)^2
+ d k_2 \p_{x_1} \psi \p_r \psi \b] r dr d{x_1}\\\no
&&
- \si \iint_D [ 6 \p_{x_1} \psi \p_{x_1}^2 \psi + d (\p_{x_1 }^2 \psi )^2  ] r dr d{x_1}
+ \int_0^1 \f 12 [d({x_1}) k_{11} (\p_{x_1} \psi )^2 - d (\p_r \psi )^2 ] \b|_{{x_1} = L_0}^{L_1} r dr.
\ee

Since $d (x_1 ) = 6 (x_1 - d_0) < 0$ for all $x_1 \in [L_0, L_1]$, then
$
 - \si \iint_D d (\p_{ x_1}^2 \psi )^2 r dr d{x_1} > 0 $,
and
\begin{equation}\no
 6 \si \iint_D \p_{x_1} \psi \p_{x_1}^2 \psi r dr d{x_1}
 \le   36 \si \iint_D (\p_{x_1} \psi)^2 r dr d{x_1} - \f {\si }4 \iint_D d({x_1}) (\p_{x_1 }^2 \psi )^2 r dr d{x_1} .
\end{equation}
As discussed in Lemma \ref{H1estimate}, for sufficiently small $\si > 0$ and $\epsilon,\delta$, there holds
\begin{equation}\label{H2ap0}
 \si \iint_D |\p_{x_1}^2 \psi|^2 r dr dx_1 + \iint_D |\na \psi |^2 r dr dx_1 \le C_* \iint_D F_0^2 r dr d x_1.
\end{equation}
Since $\psi (L_0, r) = 0$, \eqref{H1-ap1} holds.

Choose a monotonic decreasing cut-off function $\eta_1 \in C^{\oo} ([L_0, L_1])$ satisfying
\begin{equation*}
\eta_1 (x_1) =
\begin{cases}
  1, & \mbox{if } L_0 \le x_1 \le \f {L_0}2, \\
  0, & \mbox{if } \f {L_0}4 \le x_1 \le L_1.
\end{cases}
\end{equation*}

Multiplying $\eqref{ap1}_1$ by $\eta_1^2 \p_{x_1}^2 \psi$, integrations by parts give
\be\no
&&  \iint_D [(- \si \eta_1 \eta_1^{\prime} + \eta_1^2 k_{11}) (\p_{x_1}^2 \psi)^2 + \eta_1^2 (\p_{x_1 r}^2 \psi )^2 ]r dr dx_1
=
- 2 \iint_D \eta_1 \eta_1^{\prime} \p_r \psi \p_{x_1 r}^2 \psi r dr dx_1\\\no
&& \q
+ \iint_D \eta^2_1 (F_0 -k_1 \p_{x_1} \psi - k_2 \p_r \psi ) \p_{x_1}^2 \psi r dr dx_1
- 2 \iint_D \eta_1^2  k_{12} \p_{ x_1}^2 \psi \p_{x_1 r}^2 \psi r dr dx_1.
\ee
Since $\eta_1 $ is monotonically decreasing, then $- \si \iint_D \eta_1 \eta_1^{\prime} (\p_{x_1}^2 \psi)^2 r dr dx_1 > 0$.
Note that $\bar{k}_{11} \geq 2\ka_* >0$ for all $x_1 \in [L_0, \f {L_0}8]$ with some positive constant $\ka_*$, then
\begin{equation}\no
k_{11} (x_1, r) \ge \ka_* > 0, \q \forall (x_1, r) \in [L_0, \frac {L_0}8] \times [0, 1].
\end{equation}
Therefore,
\be\no
&\quad&\ka_* \int_{L_0}^{\f {L_0}2} \int_0^1 (|\p_{x_1}^2 \psi |^2 + |\p_{x_1 r}^2 \psi|^2) r dr dx_1\leq \iint_D \eta_1^2(|\p_{x_1}^2 \psi |^2 + |\p_{x_1 r}^2 \psi|^2) r dr dx_1\\\no
& \leq & \f {\ka_*}2 \iint_D \eta_1^2(|\p_{x_1}^2 \psi |^2 + |\p_{x_1 r}^2 \psi|^2) r dr dx_1 + \frac {C}{\ka_*} \iint_{D}(|F_0|^2 +(k_1^2+k_2^2+|\eta_1'|^2) |\na \psi|^2 ) r dr dx_1\\\label{H2ap1}
& \leq & C_* \iint_D |F_0|^2 r dr dx_1.
\ee

Denote $w_1 = \p_{x_1} \psi$. Then $w_1$ solves
\begin{equation}\label{ap-w1}
\begin{cases}
  \si \p_{x_1}^3 w_1 + k_{11} \p_{x_1}^2 w_1 + 2 k_{12} \p_{x_1 r}^2 w_1 + \p_{r}^2 w_1 + \f 1r \p_r w_1 + k_3 \p_{x_1} w_1
+ k_4 \p_r w_1 = F_1 (x_1, r), \\
  \p_{x_1} w_1 (L_0, r) = \p_{x_1} w_1 (L_1 , r) = 0  , \q \forall r \in [0,1],\\
  \p_r w_1 (x_1, 0) = \p_r w_1 (x_1, 1) = 0, \q \forall x_1 \in [L_0, L_1],
\end{cases}
\end{equation}
where
\begin{equation}\no
k_3 = \p_{x_1} k_{11} + k_1 , \q k_4 = 2 \p_{x_1} k_{12} + k_2, \q
F_1 (x_1, r) =  \p_{x_1} F_0 - \p_{x_1} k_1 \p_{x_1} \psi - \p_{x_1} k_2 \p_r \psi .
\end{equation}

Select another monotonic increasing cut-off function $\eta_2 \in C^{\oo} ([L_0, L_1])$ such that
\begin{equation*}
\eta_2 (x_1) =
\begin{cases}
  0, & \mbox{if } L_0 \le x_1 \le \f {3}4 L_0, \\
  1, & \mbox{if } \f {L_0}2 \le x_1 \le L_1.
\end{cases}
\end{equation*}
Multiplying $\eqref{ap-w1}_1$ by $\eta_2^2 d(x_1) \p_{x_1} w_1$ and integrating by parts yield
\be\no
&& \iint_D \eta_2^2 d(x_1) \p_{x_1} w_1 F_1 ({x_1}, r) r dr d{x_1} = - \si \iint_D (2 \eta_2 \eta_2^{\prime} d + \eta_2^2 d^{\prime}  ) \p_{x_1} w_1  \p_{x_1 }^2 w_1 r dr dx_1\\\no
&& \q
- \si \iint_D \eta_2^2 d(x_1) (\p_{x_1 }^2 w_1)^2 r dr dx_1 + \f 12 \iint_D \p_{x_1} (\eta_2^2 d ) (\p_r w_1)^2 r dr dx_1
\\\no
&& \q
+ \iint_D [ \eta_2^2 d k_3 - \eta_2^2 d (\p_r k_{12} + \f 1r k_{12}) - \f 12 \p_{x_1} (\eta^2_2 d k_{11}) ] (\p_{x_1} w_1)^2 r dr dx_1
 \\\no
&& \q
+ \iint_D \eta_2^2 d  (  2 \p_{x_1} k_{12} + k_2)  \p_{x_1} w_1 \p_r w_1 r dr dx_1
 - \f 12 d(L_1) \int_0^1 (\p_r w_1(L_1,r))^2 r dr,
\ee
where we use $\p_{x_1}^2 \psi (L_0, r) = \p_{x_1}^2 \psi (L_1, r) = 0$ for any $r \in [0, 1]$. By \eqref{6}-\eqref{7} and some estimates in Lemma \ref{coe-estimate}, there holds
 \be\no
 && \int_{\f {1}2 L_0}^{L_1} \int_0^1 (|\nabla w_1|^2+ \sigma (\p_{x_1}^2 w_1)^2) r dr dx_1+\int_0^1(\p_r w_1(L_1,r))^2 r dr\\\no
 && \leq \iint_D \eta_2^2 |\na w_1|^2 r dr dx_1  + \si \iint_D\eta_2^2(\p_{x_1 }^2 w_1)^2 r dr dx_1+ \int_0^1(\p_r w_1(L_1,r))^2 r dr\\\no
 &&\leq  C_*\iint_D (\eta_2')^2 |\nabla w_1|^2 rdr dx_1 + C_* \iint_D \eta_2^2 F_1^2 rdr dx_1\\\no
 && \le C_* \int_{\f {3}4 L_0}^{\frac12 L_0} \int_0^1 |\nabla w_1|^2 r dr dx_1 + C_* \iint_D F_1^2 r dr dx_1.
 \ee
Combining with \eqref{H2ap0} and \eqref{H2ap1} yield
\begin{equation}\label{H2ap6}
\iint_D [(\p_{x_1}^2 \psi)^2 +(\p_{x_1r}^2 \psi)^2 ]  r dr dx_1\leq C_*\iint_D (F_0^2 +(\p_{x_1} F_0)^2) rdr dx_1.
\end{equation}

Define $\nu_1 = \p_r \psi $, then
\begin{equation*}
\begin{cases}
 \si \p_{x_1}^3 \psi + k_{11} \p_{x_1 }^2 \psi + 2 k_{12} \p_{x_1 r}^2 \psi + \p_r \nu_1 + \f 1r \nu_1 + k_1 \p_{x_1} \psi + k_2 \p_r \psi  = F_0 (x_1,r), & \\
  \nu_1 (x_1, 0) = \nu_1 (x_1, 1) = 0, & \forall x_1 \in [L_0, L_1], \\
  \nu_1 (L_0, r) = \p_{x_1 }^2 \nu_1 (L_0, r) = \p_{x_1 }^2 \nu_1 (L_1, r) = 0, & \forall r \in [0, 1].
\end{cases}
\end{equation*}
Multiplying the above equation by $- d (x_1) \p_{x_1 } (\p_r \nu_1 + \f {\nu_1}r)$ and integrating by parts give
\be\no
&& \iint_D d (x_1) \p_r F_0 \p_{x_1} \nu_1 r dr dx_1
= - \si \iint_D 6 \p_{x_1} \nu_1 \p_{x_1}^2 \nu_1 + d (x_1) |\p_{x_1}^2 \nu_1|^2 r dr dx_1 \\\no
& & \q + \iint_D \b( d (x_1) k_1  - \f 12 \p_{x_1} (d k_{11}) + d (x_1) (\p_r k_{12} - \f {k_{12}}r ) \b) |\p_{x_1} \nu_1|^2 r dr dx_1
 \\\no
&& \q +  3 \iint_D   \b( |\p_r \nu_1|^2  + \b|\f {\nu_1}r\b|^2 \b)r dr dx_1
 + \iint_D d (x_1) \p_{x_1} \nu_1 \p_r k_{11} \p_{x_1}^2 \psi  r dr dx_1
  \\\no
&& \q  + \iint_D d (x_1) \p_r k_1 \p_{x_1} \nu_1 \p_{x_1} \psi r dr dx_1
+ \iint_D d (x_1) ( \p_r k_2 \nu_1 + k_2 \p_r \nu_1 ) \p_{x_1} \nu_1 r dr dx_1
 \\\label{ap-v1}
&& \q  + \f 12 \int_0^1 d (x_1) k_{11} (\p_{x_1} \nu_1)^2 \b|_{x_1=L_0}^{L_1} r dr
- \f 12 d(L_1) \int_0^1 ( |\p_r \nu_1|^2 + |\f 1r \nu_1|^2 ) (L_1, r) r dr.
\ee

Given that $d (x_1) < 0$ for all $x_1 \in [L_0, L_1]$, and that $k_{11} (L_0, r) >0$ and $k_{11} (L_1, r) < 0$ for any $r \in [0, 1]$, we conclude that
\be\no
&& \f 12 \int_0^1 d (x_1) k_{11} (\p_{x_1} \nu_1)^2 \b|_{x_1=L_0}^{L_1} r dr - \f 12 d(L_1) \int_0^1 ( |\p_r \nu_1|^2 + |\f 1r \nu_1|^2 ) (L_1, r) r dr > 0,\\\no
&&  6 \si \iint_D  \p_{x_1} \nu_1 \p_{x_1}^2 \nu_1 r dr dx_1
 \le  36 \si \iint_D |\p_{x_1} \nu_1|^2 r dr dx_1 - \f {\si}4 \iint_D d(x_1) |\p_{x_1}^2 \nu_1|^2 r dr dx_1.
\ee
Using \eqref{6}-\eqref{7} and some estimates in Lemma \ref{coe-estimate}, one has
\begin{equation}\no
 \iint_D \b(|\p_{x_1} \nu_1|^2 +  |\p_r \nu_1|^2 + \b| \f {\nu_1}{r} \b|^2 \b) r dr dx_1 + \iint_D d \p_r k_{11} \p_{x_1} \nu_1 \p_{x_1}^2 \psi r dr dx_1 \le  C_* \iint_D (|F_0|^2 + |\na F_0|^2) r dr dx_1.
\end{equation}
This, combined with \eqref{H2ap1} and \eqref{H2ap6}, yields that
\begin{equation}\no
 \iint_D \b(|\p_{x_1} \nu_1|^2 +  |\p_r \nu_1|^2 + \b| \f {\nu_1}{r} \b|^2 \b) r dr dx_1 + \iint_{D} (|\p_{x_1}^2 \psi |^2 + |\p_{x_1 r}^2 \psi |^2 ) r dr dx_1
\le
C_* \iint_D ( |F_0|^2 + |\na F_0|^2 ) r dr dx_1 .
\end{equation}
Summing up, one gets \eqref{H2-ap}.
\end{proof}

Now we show the existence of an orthonormal basis of $L^2_r([0,1])$ which are eigenfunctions of the operator $\p_r^2+\frac{1}{r}\p_r$. The following lemma can be proved by using Theorem 6.3.4 in \cite{BB92}.

\begin{lemma}\label{eigen}
There exists a monotone increasing nonnegative real number sequence $\{ \la_j \}_{j = 1}^{\oo}$ of eigenvalues and the associated eigenfunctions $\{ b_j (r) \}_{j = 1}^{\oo}$ to the following eigenvalue problem
  \begin{equation*}
\begin{cases}
  - b_j''(r) - \f 1r b_j'(r) = \la_j b_j(r), \q r \in [0,1], \\
  b_j' (0) = b_j' (1) = 0.
\end{cases}
\end{equation*}
Furthermore, $\{ b_j \}_{j = 1}^{\oo}$ is an orthonormal basis of $L^2_r ([0,1])$ with the inner product
\begin{equation}\label{l2inner}
(b (r), q (r) )_0 = \int_0^1 b (r) q (r) r dr.
\end{equation}
\end{lemma}

Define the approximated solution as
\begin{equation}\label{app1}
\psi^{N, \si} (x_1, r) = \sum_{j =1}^N A_j^{N, \si } (x_1) b_j (r),
\end{equation}
 which satisfies the following $N$ linear equations on $[L_0, L_1]$,
\begin{equation*}
\begin{cases}
\int_0^1 \mc{L}_{\si} \psi^{N, \si } (x_1, r) b_m (r) r dr = \int_0^1 F_0 (x_1, r) b_m (r) r dr , \q m = 1, \cdots, N, \\
\p_{x_1}^2 \psi^{N, \si } (L_0, r) = \p_{x_1}^2 \psi^{N, \si } (L_1, r) = 0,\\
\psi^{N, \si} (L_0, r) =0.
\end{cases}\end{equation*}
Thus $\{A^{N, \si }_j\}_{j=1}^N$ solves
\begin{equation}\label{app2}
\begin{cases}
 \si \f {d^3}{ d x_1^3} A_m^{N, \si }
   + \sum_{j= 1}^N a_{jm}\f {d^2}{d x_1^2} A_j^{N, \si }
   + \sum_{j= 1}^N b_{jm} \f d {d x_1} A_j^{N, \si }
   + \sum_{j= 1}^N c_{jm} A_j^{N, \si }
   = F_{0m} (x_1), \forall x_1\in [L_0,L_1],
   \\
   A_m^{N, \si } (L_0) = \f {d^2}{d x_1^2} A_m^{N, \si } (L_0) = 0,\\
   \f {d^2}{d x_1^2} A_m^{N, \si } (L_1) = 0,
\end{cases}
\end{equation}
 where
\be\no
&& a_{jm}=\int_0^1 k_{11} (x_1, r) b_j (r) b_m (r) r dr,\ \ \ b_{jm}= \int_0^1 (k_1 (x_1, r) b_j (r) + 2 k_{12} b_j^{\prime} )b_m (r) r dr,\\\no
&& c_{jm}=- \la_j \de_{jm} + \int_0^1 k_2 (x_1,r ) b_j^{\prime} (r) b_m (r) r dr, \ \ F_{0m}(x_1)= \int_0^1 F_0 (x_1, r) b_m (r)  r dr.
\ee

 \begin{lemma}\label{exist1}
   There exists a unique smooth solution $\{A_j^{N, \si }\}_{j=1}^N$ to \eqref{app2} such that the approximate solution $\psi^{N, \si } (x_1, r)$ defined in \eqref{app1} satisfies
   \begin{equation}\label{H2ap5}
   \iint_D ( |\psi^{N, \si }|^2 + |\na \psi^{N, \si}|^2 + |\na^2 \psi^{N, \si}|^2 + \b|\f {\p_r \psi^{N, \si }}r \b|^2 ) r dr dx_1 \le C_* \iint_D (F_0^2 + |\na F_0|^2 ) r dr dx_1 ,
   \end{equation}
 where $C_*$ depends only on the $H^3_r (D )$ norms of $k_{11}$, $k_1$, the $H^2_r (D)$ norms of $\p_r k_{12}$, $\f {k_{12}}r$, $\p_r k_2$ and the $L^{\oo}_r (D)$ norm of  $k_{12}$, $\p_{x_1} k_{12}$ and $k_2$.

 \end{lemma}
 \begin{proof}
   Multiplying the $m^{th}$ equation in \eqref{app2} by $d(x_1) \f d{dx_1 } A_m^{N, \si } $, summing from 1 to $N$, and integrating over $[L_0, L_1]$ yield that
   \begin{equation*}
   \iint_D \mc{L}_{\si} \psi^{N, \si } d (x_1) \p_{x_1} \psi^{N, \si } r dr dx_1 = \iint_D F_0 d (x_1) \p_{x_1} \psi^{N, \si } r dr dx_1.
   \end{equation*}
   Using the integration by parts as in Lemma \ref{H1estimate-ap}, we obtain
   \begin{equation*}
  \si \iint_D |\p_{x_1}^2 \psi^{N, \si }|^2 r dr dx_1
   + \iint_D (|\psi^{N, \si }|^2 + |\na \psi^{N, \si }|^2 ) r dr dx_1 
    \le C \iint_D F_0^2 r dr dx_1.
   \end{equation*}
   This implies the solution to problem \eqref{app2} is unique. For the system of $N$ third-order equations endowed with $3N$ boundary conditions, the existence of the solution to \eqref{app2} can be derived from the uniqueness and the proof is same as \cite[Lemma 2.7]{WX23}, so is omitted.

   To get the $H^2_r(D)$ estimate of $\psi^{N,\si}$, selecting the cutoff functions $\eta_1$, $\eta_2$ as in Lemma \ref{H1estimate-ap}, multiplying the $m^{th }$ equation in \eqref{app2} by $\eta_1^2 \f {d^2}{d x_1^2} A_m^{N, \si } $, summing from 1 to $N$, integrating over $[L_0, L_1]$ and running the same argument as for \eqref{H2ap1} will yield
   \begin{equation}\label{H2ap2}
   \int_{L_0}^{\f {L_0}2} \int_0^1 (|\p_{x_1}^2 \psi^{N, \si }|^2 + |\p_{x_1 r}^2 \psi^{N, \si }|^2 )r dr dx_1 \le \iint_D F_0^2 r dr dx_1.
   \end{equation}
  Denote $w_1^{N, \si } = \p_{x_1} \psi^{N, \si } = \sum_{j =1}^N w_{1, j}^{N, \si } (x_1) b_j (r)$, where $w_{1, j}^{N, \si } (x_1) = \f d{dx_1} A_j^{N, \si }$. Taking $\f d {dx_1}$ on each equation in \eqref{app2}, then multiplying it by $\eta_2^2 d(x_1) \f {d}{dx_1} w_{1, j}^{N, \si } $, summing from 1 to $N$ and integrating over $[L_0, L_1]$. After some computations, combined with \eqref{H2ap2}, one gets
   \begin{equation}\no
   \iint_D (|\p_{x_1}^2 \psi^{N, \si}|^2 + |\p_{x_1 r}^2 \psi^{N, \si}|^2 ) r dr dx_1 \le C_* \iint_D (F_0^2 + |\p_{x_1} F_0|^2) r dr dx_1,
   \end{equation}
   with a uniform constant $C_*$ for $N$, $\si $. Finally, we will get the estimate of $\na \p_r \psi^{N, \si }$ and $\f 1r \p_r \psi^{N, \si }$. Note that $-b_m''(r)-\frac{1}{r}b_m'(r)=\lambda_m b_m(r)$, then
   \begin{equation}\no
   \sum_{m = 1}^N \la_m b_m (r) \f {d}{d x_1 } A_m^{N, \si } (x_1) = \sum_{m = 1}^N\f {d}{d x_1 } A_m^{N, \si } (x_1) (- b_m'' (r) - \f 1r b_m' (r)) = - \p_{x_1}\left(\p_r^2+\f 1r \p_r\right)\psi^{N, \si}.
   \end{equation}
   Thus, one may multiply the $m^{th }$ equation in \eqref{app2} by $\la_m d(x_1) \f {d}{d x_1 } A_m^{N, \si } $, sum from 1 to $N$, and integrate over $[L_0, L_1]$, where $\la_m $ is the eigenvalue associated with $b_m (r)$. Integrations by parts as in \eqref{ap-v1} yield
   \begin{equation}\label{H2ap7}
   \iint_D (|\p_{x_1 r}^2 \psi^{N,\si}|^2 + |\p_r^2 \psi^{N, \si}|^2 + |\f 1r \p_r \psi^{N, \si}|^2 ) r dr dx_1\le C_* \iint_D (F_0^2 + |\na F_0|^2) r dr dx_1,
   \end{equation}
    with a uniform constant $C_*$ for $N$, $\si $. Then \eqref{H2ap5} follows immediately.
 \end{proof}

 \begin{lemma}\label{exist2}
   The problem \eqref{linearized2} has a unique $H^2_r(D)$ strong solution $\psi (x_1, r)$ satisfying
   \begin{equation}\label{H2}
   \iint_D \b( |\psi |^2 + |\na \psi|^2 + |\na^2 \psi|^2 + \b|\f {\p_r \psi}r\b|^2 \b) r dr dx_1 \le C_* \iint_D (F_0^2 + |\na F_0|^2 ) r dr dx_1,
   \end{equation}
  where $C_*$ depends only on the $H^3_r (D )$ norms of $k_{11}$, $k_1$, the $H^2_r (D)$ norms of $\p_r k_{12}$, $\f {k_{12}}r$, $\p_r k_2$ and the $L^{\oo}_r (D)$ norm of $k_{12}$, $\p_{x_1} k_{12}$ and $k_2$.
 \end{lemma}
 \begin{proof}
    It follows from \eqref{H2ap5} that $\| \psi^{N, \si } \|_{H^2_r (D )} $ are uniformly bounded with respect to $N,\sigma$. For a fixed $\sigma$, by the weak compactness of a bounded set in a Hilbert space, there exists a subsequence, still denoted by $\psi^{N, \si}$ for simplicity, which converges strongly in $H^1_r (D )$ and converges weakly in $H^2_r (D )$ to a limit $\psi^{\si} \in H^2_r (D )$ with a uniform $H^2_r(D)$ estimate with respect to $\si $,
   \begin{equation}\no
   \| \psi^{\si} \|_{H^2_r (D )} \le C_* \| F_0 \|_{H^1_r (D )}.
   \end{equation}

   Due to the strong convergence $\psi^{N, \si } \to \psi^{\si }$ as $N \to \oo$ in $H^1_r (D)$, $\psi^{\si }$ retains the boundary conditions
   \begin{equation*}
   \begin{cases}
     \psi^{\si } (L_0, r) =0,& \forall r \in [0, 1],   \\
     \p_r \psi^{\si }(x_1, 0 ) = \p_r \psi^{\si } (x_1, 1) = 0, & \forall x_1 \in [L_0, L_1].  \\
   \end{cases}
   \end{equation*}

   Next we show that $\psi^{\si }$ is a weak solution to the system \eqref{ap1}. Given any test function $\chi (x_1, r) = \sum_{m =1}^{N_0} \chi_m (x_1) b_m (r)$, where $\chi_m (x_1) \in C^{\oo} ([L_0, L_1])$ and $\chi_m(L_0)=\chi_m(L_1)=0$ for any $1\leq m\leq N$. Let $N \ge N_0$. Multiplying each equation in \eqref{app2} by $\chi_m $ ($\chi_m \equiv 0$ for any $N_0 + 1 \le j \le N$), summing from $m =1$ to $m =N$, and integrating with respect to $x_1$ from $L_0$ to $L_1$  give
   \be\no
   \iint_D \b(\si \p_{x_1}^3 \psi^{N, \si } + k_{11} \p_{x_1}^2 \psi^{N, \si } + 2 k_{12} \p_{x_1 r}^2 \psi^{N, \si } + \p_{r}^2 \psi^{N, \si } + \f 1r \p_r \psi^{N, \si}  \\\no
   + k_1 \p_{x_1} \psi^{N, \si } + k_2 \p_r \psi^{N, \si }\b) \chi r dr dx_1 = \iint_D F_0 \chi r dr dx_1 .
   \ee
   After integrating by parts and taking the limit of the above weak convergent subsequence of $\psi^{N, \si }$, one gets
   \be\label{wf2}
   && \iint_D ( - \si \p_{x_1 }^2 \psi^{\si } \p_{x_1} \chi
   - \p_{x_1} (k_{11} \chi ) \p_{x_1} \psi^{\si}
   - 2 \p_{x_1} (k_{12} \chi ) \p_r \psi^{\si } \\\no
   && \q\q - \p_r \psi^{\si } \p_r \chi + k_1 \p_{x_1} \psi^{\si } \chi + k_2 \p_r \psi^{\si } \chi ) r dr dx_1 = \iint_D F_0 \chi r dr dx_1 .
   \ee

   Using a density argument, the weak formulation \eqref{wf2} holds for any text function $\chi \in H^1_r (D )$ vanishing at $x_1 = L_0$ and $x_1 = L_1$. We now consider a sequence of approximate solutions $\psi^{\si}$ as $\si \to 0$. Thanks to \eqref{H2ap5}, the norm $\| \psi^{\si } \|_{H^2_r (D )}  $ is uniformly bounded with respect to $\si $. This further implies the existence of a weakly convergent subsequence labeled as $\{ \psi^{\si_j} \}_{j =1}^{\oo}$ with $\si_j \to 0$ as $j \to \oo$, which converges weakly to a limit $\psi \in H^2_r (D)$. Moreover, $\psi $ retains the boundary condition
 \begin{equation*}
 \begin{cases}
   \psi    (L_0, r) =0, & \forall r \in [0, 1],   \\
   \p_r \psi (x_1, 0 ) = \p_r \psi  (x_1, 1) = 0, & \forall x_1 \in [L_0, L_1].
 \end{cases}
 \end{equation*}

 From
 \eqref{wf2}, it is obvious that
 \begin{equation*}
 \iint_D (- \p_{x_1} \psi \p_{x_1} (k_{11} \chi ) - 2 \p_r \psi \p_{x_1} (k_{12} \chi ) - \p_r \psi \p_r \chi
 + k_1 \p_{x_1} \psi \chi + k_2 \p_r \psi \chi) r dr dx_1 = \iint_D F_0 \chi r dr dx_1,
 \end{equation*}
 holds for any $\chi \in H^1_r (D )$ vanishing at $x_1 = L_0$ and $x_1 = L_1$. Since $\psi \in H^2_r (D ) $, then $\psi $ is actually a strong solution to \eqref{linearized2} and the equation in \eqref{linearized2} holds almost everywhere.
 \end{proof}

\section{The $H^4_r$ energy estimates and the proof of Theorem \ref{irro}.}\label{h4estimate}\noindent

In this section, we establish higher order estimates for the solution of \eqref{linearized2} and complete the proof of Theorem \ref{irro}. Since $k_{11} (x_1, r)$ changes sign as the fluid moves across the sonic front, the equation \eqref{linearized2} is elliptic in subsonic region and changes types thereafter, we have to improve the regularity of $\psi$ in the subsonic region and the transonic region separately.

Since the equation \eqref{linearized2} is elliptic in $D_{\f 18} \co \{ (x_1, r): L_0< x_1 < \f {L_0}8, 0 < r < 1\} $, the $H^4_r$ estimate of $\psi $ on $D_{\f 18}$ can be obtained by elliptic theories. However, due to the term $\frac{1}{r}\p_r\psi$ in \eqref{linearized2}, one derives only $\partial^3_r \psi (x_1, 1) + \partial^2_r \psi (x_1, 1) = 0$, rather than $\partial^3_r \psi (x_1, 1) =0$. The symmetric extension technique used in \cite{WX23} can not be applied in this case. Furthermore, we need to deal with the possible singularity near the axis $r=0$. Thus we have to separate the region $D_{\f18}$ into two subregions: one is near the nozzle wall $r=0$, another one contains the axis $r = 0$. Different techniques are employed to derive the $H_r^4$ norm estimates in these two subregions. Near the nozzle wall, we use the special structure of \eqref{linearized2} so that the elliptic estimates derived in \cite{Grisvard11} for convex domains can be applied. Near the axis, we transform the potential function back to the Cartesian coordinates so that the singularities near the axis disappear and one can use the standard estimates of second order elliptic equation in \cite{gt}.

\begin{lemma}\label{H3}
Under the assumptions of Lemma \ref{H1estimate}, the $H_r^2$ strong solution to \eqref{linearized2} satisfies
  \begin{equation}\label{H31}
  \int_{L_0}^{\f {3 L_0}8} \int_0^1 (|\na^3 \psi |^2 + |\na^4 \psi |^2 + |\na (\f 1r {\p_r \psi})|^2 + |\na^2 (\f 1r {\p_r \psi})|^2 + |\f 1r \p_r (\f 1r \p_r \psi )|^2 ) r dr dx_1 \le C_* \| F_0 \|_{H^2_r (D )}^2,
  \end{equation}
 where $C_*$ depends only on the $H^3_r (D )$ norms of $k_{11}$, $k_1$, the $H^2_r (D)$ norms of $\p_r k_{12}$, $\f {k_{12}}r$, $\p_r k_2$, $\f {k_2}r$, the $L^{\oo}_r (D)$ norm of $k_{12}$, $\p_{x_1} k_{12}$ and $k_2$, the $L^2_r (D)$ norm of $\p_{x_1}^2 k_2$ and the $L^4_r (D)$ norm of $\p_{x_1}^2 k_{12}$.
\end{lemma}
\begin{proof}

Define $D_{\f 38} \co [L_0, \f 38 L_0] \times [0,1]$. The proof will be divided into several steps.

{\bf Step 1.} The $H^4_r $ estimate near the nozzle wall $r = 1$.
Set $v_1 = \p_r \psi $, then
\begin{equation*} \begin{cases}
   k_{11} \p_{x_1 }^2 v_1 + 2 k_{12} \p_{x_1 r}^2 v_1 + \p_{r}^2 v_1 + \f 1r \p_r v_1 - \f 1{r^2} v_1 + k_5 \p_{x_1} v_1  + k_2 \p_r v_1 + k_6 v_1
= F_2,  \\
  v_1 (x_1, 0) = v_1 (x_1, 1) = 0, \q \forall x_1 \in [L_0, L_1], \\
   v_1 (L_0, r) 
   = 0, \q \forall r \in [0, 1],
\end{cases}
\end{equation*}
where
\begin{equation*}
k_5 =  k_1 + 2 \p_r k_{12}, \q
k_6 = \p_r k_2, \q
F_2 = \p_r F_0 - \p_r k_{11} \p_{x_1}^2 \psi
- \p_r k_1 \p_{x_1} \psi.
\end{equation*}

{\bf Step 1.1.} The $H^3_r $ estimate of $\psi $ near the nozzle wall $r = 1$.

Choose a cut-off function $\xi_1 (x_1, r) \in C^{\oo} ([L_0, L_1] \times [0, 1])$ such that $0 \le \xi_1  (x_1, r) \le 1$ for all $(x_1 , r) \in [L_0, L_1] \times [0, 1]$ and
\begin{equation*}
\xi_1  (x_1, r) = \begin{cases}
              1, & \mbox{on } (x_1, r) \in (L_0, \f 3{16} L_0 ) \times (\f 14, 1), \\
              0, & \mbox{on } (x_1, r ) \notin (L_0, \f 18 L_0) \times (\f 18, 1).
            \end{cases}
\end{equation*}
Let $\tilde{v}_1 = \xi_1  v_1 $, then $\tilde{v}_1$ satisfies
\begin{equation*}
\begin{cases}
  \p_{x_1} (k_{11} \p_{x_1} \tilde{v}_1 + k_{12} \p_r \tilde{v}_1) + \p_r (k_{12} \p_{x_1} \tilde{v}_1 + \p_r \tilde{v}_1) = g_1, & \\
  \tilde{v}_1 (L_0, r) = \tilde{v}_1 (\f 18 L_0, r) = 0, & \forall r \in [0, 1], \\
  \tilde{v}_1 (x_1, \f 18) = \tilde{v}_1 (x_1, 1) = 0, & \forall x_1 \in [L_0, \f 18 L_0],
\end{cases}
\end{equation*}
where
\be\no
&& g_1 =  \xi_1  (\p_r F_0 - \p_r k_{11} \p_{x_1}^2 \psi - \p_r k_1 \p_{x_1} \psi - \f 1r \p_r v_1 + \f 1{r^2} v_1 - k_1 \p_{x_1} v_1 - \p_r k_{12} \p_{x_1} v_1  - k_2 \p_r v_1 \\\no
&&\quad \quad   - \p_r k_2 v_1 + \p_{x_1} k_{11} \p_{x_1} v_1 + \p_{x_1} k_{12} \p_r v_1) + v_1 ( \p_{x_1} (k_{11} \p_{x_1} \xi_1 + k_{12} \p_r \xi_1) + \p_r (k_{12} \p_{x_1} \xi_1 + \p_r \xi_1) ) \\\no
&&\quad\quad
+ 2  \p_{x_1} \xi_1 ( k_{11} \p_{x_1} v_1 +  k_{12}  \p_r v_1) + 2 \p_r \xi_1 ( k_{12}   \p_{x_1} v_1 + \p_r v_1 )  .
\ee
According to \cite[Theorem 3.1.3.1 and 3.2.1.2]{Grisvard11}, there holds that
\begin{equation*}
 \| {v}_1 \|_{H^2_r (D_{3,4})} \le \| \tilde{v}_1 \|_{H^2_r (D_{1,1})} \le c_* \| g_1 \|_{L^2_r (D_{1,1})},
\end{equation*}
where $D_{3,4} \co (L_0, \f 3{16} L_0) \times (\f 14, 1)$, $D_{1, 1} \co (L_0, \f 18 L_0) \times (\f 18, 1)$ and $c_*$ depends only on the diameter of $D_{1,1}$ and the $C^{0, 1}$ norm of $k_{ij}$ on $D_{1,1}$, thus $c_*$ depends only on $L_0$ and the $H^3_r (D_{1,1})$ norms of $k_{11}$, the $H^2_r (D_{1,1})$ norms of $\p_r k_{12}$ and the $L^{\oo}_r (D_{1,1})$ norm of $\p_{x_1} k_{12}$ .

It follows from Lemma \ref{coe-estimate} and \eqref{H2} that
\be\no
&& \| g_1 \|_{L^2_r (D_{1,1})}  \le C_* \b(
\| \p_r F_0 \|_{L^2_r (D_{1,1})}
+ \|  \p_r k_{11} \|_{L^{\oo}_r (D_{1,1})} \| \p_{x_1}^2 \psi \|_{L^2_r (D_{1,1})} + \| \p_r  k_1  \|_{L^{\oo}_r (D_{1,1})}  \| \p_{x_1} \psi \|_{L^2_r (D_{1,1})} \\\no
&& \q
+ \| \p_{x_1 r}^2 \psi \|_{L^2_r (D_{1,1})}  (
 \| \p_{x_1} k_{11}  \|_{L^{\oo}_r (D_{1,1})}
 + \| k_1  \|_{L^{\oo}_r (D_{1,1})}
+ \| k_{12} \|_{L^{\oo}_r (D_{1,1})}
+ \| \p_r k_{12} \|_{L^{\oo}_r (D_{1,1})})
 \\\no
&& \q
 + \| \p_{r}^2 \psi \|_{L^2_r (D_{1,1})} ( \| k_2 \|_{L^{\oo}_r (D_{1,1})} + \| k_{12} \|_{L^{\oo}_r (D_{1,1})} + \| \p_{x_1} k_{12} \|_{L^{\oo}_r (D_{1,1})} + \| \p_r k_2 \|_{L^{\oo}_r (D_{1,1})} ) \\\no
 && \q
 + \| \p_r \psi  \|_{L^2_r (D_{1,1})} (  \|  \p_{x_1} k_{11} \|_{L^{\oo}_r (D_{1,1})}
 +  \| k_{11}  \|_{L^{\oo}_r (D_{1,1})}
+ \| \p_r k_2 \|_{L^{\oo}_r (D_{1,1})} + \| k_{12} \|_{L^{\oo}_r (D_{1,1})}
 \\\no
 && \q
  + \| \p_{x_1} k_{12} \|_{L^{\oo}_r (D_{1,1})}   + \| \p_r k_{12} \|_{L^{\oo}_r (D_{1,1})}   )
+ \| \psi \|_{H^2_r (D_{1,1})} \b)
  \le  C_*  \| F_0 \|_{H^1_r (D)} .
\ee
Therefore,
\begin{equation*}
 \| \p_r \psi  \|_{H^2_r (D_{3,4})} \le C_*  \| F_0 \|_{H^1_r (D)},
\end{equation*}
and $\na (\f {\p_r \psi } r) \in L^2_r (D_{3,4})$. Note that
\begin{equation*}
\p_{x_1}^2 \psi =  \f 1{k_{11}} (F_0 - 2 k_{12} \p_{x_1 r}^2 \psi - \p_{r}^2 \psi - \f 1r \p_r \psi - k_1 \p_{x_1} \psi - k_2 \p_r \psi ),
\end{equation*}
one has
\begin{equation*}
\|\psi\|_{H^3_r (D_{3,4})} \le C_*  \| F_0 \|_{H^1_r (D)}.
\end{equation*}


{\bf Step 1.2.} The $H^4_r $ estimate of $\psi $ near the nozzle wall. A key observation here is the following. Set $w_2 = \p_{x_1}^2 \psi $. Since $k_{12} (L_0, r) = 0$ for any $r \in [1 - \beta_0, 1]$, $\p_r \psi (L_0, r) = \p_{r}^2 \psi (L_0, r) = 0$ for any $r\in [0,1]$, then
\begin{equation*}
k_{11} (L_0, r) \p_{x_1}^2 \psi (L_0, r) + k_1 (L_0, r) \p_{x_1} \psi (L_0, r) = F_0 (L_0, r), \ \forall r \in[1 - \beta_0, 1],
\end{equation*}
and
\begin{equation*}
\begin{cases}
  w_2 (L_0, r) = \f 1{k_{11} (L_0, r)} (F_0 (L_0, r) - k_1 (L_0, r) \p_{x_1} \psi (L_0, r)), \q  \forall r \in[1 - \beta_0, 1], \\
   \p_r w_2 (x_1, 1) = 0, \q x_1 \in [L_0, \f 3{16} L_0].
\end{cases}
\end{equation*}

Set $\hat{w}_2 (x_1, r) = w_2 (x_1, r) - \f 1 {k_{11} (x_1, r)} (F_0 (x_1, r) - k_1 (x_1, r) \p_{x_1} \psi (x_1, r))$, then $\hat{w}_2$ satisfies a homogeneous mixed boundary conditions on the entrance and the nozzle wall:
\begin{equation}\label{w21}
\begin{cases}
\p_{x_1} (k_{11} \p_{x_1} \hat{w}_2 + k_{12} \p_r \hat{w}_2) + \p_r (k_{12} \p_{x_1} \hat{w}_2 + \p_r \hat{w}_2) = \hat{g}_2,  & (x_1, r) \in  (L_0,\frac{3}{16}L_0)\times (1 - \beta_0,1),  \\
  \hat{w}_2 (L_0, r) = 0, & \forall r \in [1 - \beta_0, 1], \\
  \p_r \hat{w}_2 (x_1, 1) = 0, & \forall x_1 \in [L_0, \f 3{16} L_0],
\end{cases}
\end{equation}
where
\be\no
&& \hat{g}_2 = \p_{x_1}^2 F_0 - \p_{x_1}^2 k_{11} w_2 - \p_{x_1} k_{11} \p_{x_1} w_2 - 3 \p_{x_1} k_{12} \p_r w_2 - \f1r \p_r w_2 - 2 \p_{x_1} k_1 w_2\\\no
&& \q\q
 - k_1 \p_{x_1} w_2
- k_2 \p_r w_2 - 2 \p_{x_1}^2 k_{12} \p_{x_1 r}^2 \psi - \p_{x_1}^2 k_1 \p_{x_1} \psi - \p_{x_1}^2 k_2 \p_r \psi - 2 \p_{x_1} k_2 \p_{x_1 r}^2 \psi \\\no
&& \q\q
- \p_{x_1} \b[
k_{11} \p_{x_1} (\f 1{k_{11}} (F_0 - k_1 \p_{x_1} \psi )) + k_{12} \p_r (\f 1{k_{11}} (F_0 - k_1 \p_{x_1} \psi )) \b] \\\no
&& \q\q
- \p_r \b[
k_{12} \p_{x_1} (\f 1{k_{11}} (F_0 - k_1 \p_{x_1} \psi )) + \p_r (\f 1{k_{11}} (F_0 - k_1 \p_{x_1} \psi ))\b],
\ee
and $\hat{g}_2$ satisfies
\be\no
&& \| \hat{g}_2 \|_{ L^2_r (D_{3, 1})} \le \| \p_{x_1}^2 F_0 \|_{L^2_r (D_{3, 1})} + \| \p_{x_1}^2 k_{11}\|_{L^4_r (D)} \| w_2 \|_{L^4_r (D_{3, 1})} + \| \p_{x_1} k_{11}  \|_{L^{\oo}_r (D)} \| \p_{x_1} w_2\|_{L^2_r (D_{3, 1})}
 \\\no
&& \q
+ \| \p_{x_1}^2 k_{12} \|_{L^4_r (D)} \| \p_{x_1 r}^2 \psi \|_{L^4_r (D_{3,1})} + \| \p_{x_1} k_{12} \|_{L^{\oo}_r (D)} \| \p_r w_2 \|_{L^2_r (D_{3,1})} + \| \f 1r \p_r w_2 \|_{L^2_r (D_{3,1})}  \\\no
&& \q
+ \| \p_{x_1}^2 k_1\|_{L^2_r (D)} \| \p_r \psi \|_{L^{\oo}_r (D_{3,1})} + \| \p_{x_1} k_1  \|_{L^{\oo}_r (D)} \|w_2 \|_{L^2_r (D_{3,1})} + \| k_1 \|_{L^{\oo}_r (D)} \| \p_{x_1} w_2 \|_{L^2_r (D_{3,1})}\\\no
&& \q
  + \| \p_{x_1}^2 k_2 \|_{L^2_r (D)} \| \p_r \psi \|_{L^{\oo}_r (D_{3,1})}
+ \| \p_{x_1} k_2 \|_{L^{\oo}_r (D)} \| \p_{x_1r}^2 \psi \|_{L^2_r (D_{3,1})} + \| k_2 \|_{L^{\oo}_r (D)} \| \p_r w_2 \|_{L^2_r (D_{3,1})} \\\no
&& \q
 + ( \| k_{11} \|_{L^{\oo}_r (D)} + \| k_{12} \|_{L^{\oo}_r (D)}  )
\| \na^2 (\f 1{k_{11}} (F_0 - k_1 \p_{x_1} \psi)) \|_{L^2_r (D_{3,1})} + (\| \p_{x_1} k_{11} \|_{L^{\oo}_r (D)} \\\no
&& \q  + \| \p_{x_1} k_{12} \|_{L^{\oo}_r (D)} + \| \p_r k_{12} \|_{L^{\oo}_r (D)})
\| \na (\f 1{k_{11}} (F_0 - k_1 \p_{x_1} \psi)) \|_{L^2_r (D_{3,1})}
  \le C_* \| F_0 \|_{H^2_r (D)},
\ee
with $D_{3, 1} = (L_0, \f 3{16}L_0) \times (1 - \beta_0, 1)$.

Extend $\hat{w}_2$, $\hat{g}_2$, $k_{11}$ and $k_{12}$ from $(L_0, \f 3{16}L_0) \times (1 - \beta_0, 1)$ to $(L_0, \f 3{16}L_0) \times (1 - \beta_0, 1 + \beta_0)$ as follows
\begin{equation*}
\tilde{w} (x_1, r) = \begin{cases}
                              \hat{w}_2 (x_1, r), & 1 - \beta_0 < r < 1 \\
                              \hat{w}_2 (x_1, 2 - r), & 1 < r < 1 + \beta_0
                            \end{cases} , \q
(\tilde{k}_{11}, \tilde{g}_2) (x_1, r) = \begin{cases}
                                            (k_{11}, \hat{g}_2) (x_1, r), & 1 - \beta_0 < r < 1 \\
                                            (k_{11}, \hat{g}_2) (x_1, 2 - r), & 1 < r < 1 + \beta_0
                                          \end{cases},
\end{equation*}
and
\begin{equation*}
\tilde{k}_{12} (x_1, r) = \begin{cases}
                                 k_{12} (x_1, r), & 1 - \beta_0 < r < 1 \\
                                 - k_{12} (x_1, 2 - r), & 1 < r < 1 + \beta_0
                               \end{cases}.
\end{equation*}
Choose a cut-off function $\xi_2 (x_1, r) \in C^{\oo} ([L_0, \f 3{16} L_0] \times [1 - \beta_0, 1 + \beta_0])$ such that $0 \le \xi_2(x_1, r) \le 1$ for all $(x_1, r) \in [L_0, \f 3{16} L_0] \times [1 - \beta_0, 1 + \beta_0]$ and
\begin{equation*}
\xi_2 (x_1, r) = \begin{cases}
                   1, & \mbox{on } (x_1, r) \in (L_0, \f 14 L_0) \times (1 - \f 12 \beta_0, 1 + \f 12 \beta_0) \\
                   0, & \mbox{on } (x_1, r) \notin (L_0, \f 7{32} L_0) \times (1 - \f 34 \beta_0, 1 + \f 34 \beta_0)
                 \end{cases}.
\end{equation*}
Let $w  = \xi_2 \tilde{w}$, then $w$ satisfies
\begin{equation}\label{w22}
\begin{cases}
  \p_{x_1} (\tilde{k}_{11} \p_{x_1} w + \tilde{k}_{12} \p_r w) + \p_r (\tilde{k}_{12} \p_{x_1} w + \p_r w) = g_2, & \\
  w (L_0, r) = w (\f 3{16} L_0, r) = 0, &   r \in [1 - \beta_0, 1+ \beta_0 ], \\
 w (x_1, 1 - \beta_0) =  w (x_1, 1 + \beta_0) = 0, & x_1 \in [L_0, \f 3{16} L_0],
\end{cases}
\end{equation}
where
\be\no
g_2 = \xi_2 \tilde{g}_2 + \tilde{w} [\p_{x_1} (\tilde{k}_{11} \p_{x_1} \xi_2 + \tilde{k}_{12} \p_r \xi_2) + \p_r (\tilde{k}_{12} \p_{x_1} \xi_2 + \p_r \xi_2) ] \\\no
+ 2 \p_{x_1} \xi_2 (\tilde{k}_{11} \p_{x_1} \tilde{w} + \tilde{k}_{12} \p_r \tilde{w} )
+ 2 \p_r \xi_2 (\tilde{k}_{12} \p_{x_1} \tilde{w} + \p_r \tilde{w}).
\ee
Again, it follows from \cite[Theorem 3.1.3.1 and 3.2.1.2]{Grisvard11} that
\begin{equation}\label{w2h21}
\| \tilde{w}  \|_{H^2_r (D_{4,2})} \le \| w \|_{H^2_r (D_{3,2})} \le c_* \| g_2 \|_{L^2_r (D_{3,2})},
\end{equation}
where $D_{4,2} \co (L_0, \f 14 L_0) \times (1 - \f 12 \beta_0, 1 + \f 12 \beta_0)$, $D_{3,2} \co (L_0, \f 3{16} L_0) \times (1 - \beta_0, 1 + \beta_0)$, and $c_*$ only depends on the $C^{0, 1}$ norms of $k_{ij}$, so only depends on the $H^3_r (D)$ norm of $k_{11}$, the $H^2_r (D)$ norm of $\p_r k_{12}$ and the $L^{\oo}_r (D)$ norm of $\p_{x_1} k_{12}$. Moreover,
\be\no
&&  \| g_2 \|_{L^2_r (D_{3,2})} \le C_* \{ \| \tilde{g}_2 \|_{L^2_r (D_{3,2})}
+ \| \tilde{w} \|_{L^2_r (D_{3,2})} (\| \p_{x_1} \tilde{k}_{11} \|_{L^{\oo}_r (D)} + \| \tilde{k}_{11} \|_{L^{\oo}_r (D)} + \| \na \tilde{k}_{12} \|_{L^{\oo}_r (D)} + \| \tilde{k}_{12} \|_{L^{\oo}_r (D)} )
  \\\no
&& \q
 +   ( \| \tilde{k}_{11} \|_{L^{\oo}_r (D)} + \| \tilde{k}_{12} \|_{L^{\oo}_r (D)}) \| \p_{x_1} \tilde{w} \|_{L^2_r (D_{3,2})} +   ( \| \tilde{k}_{12} \|_{L^{\oo}_r (D)} + 1) \| \p_r \tilde{w} \|_{L^2_r (D_{3,2})} \}
 \le C_* \|  F_0 \|_{H^2_r (D)} .
\ee
Then we can obtain that
\begin{equation*}
\| w_2 \|_{H^2_r ([L_0, \f 14 L_0] \times [1 - \f 12 \beta_0, 1])} \leq C_* \|  F_0 \|_{H^2_r (D)}.
\end{equation*}
Due to
\begin{equation*}
\p_{r}^2 \psi  = F_0 - k_{11} \p_{x_1}^2 \psi - 2 k_{12} \p_{x_1 r}^2 \psi - k_1 \p_{x_1} \psi - k_2 \p_r \psi - \f 1r \p_r \psi,
\end{equation*}
it can be demonstrated that $\na^4 \psi \in L^2_r ([L_0, \f 14 L_0] \times [1 - \f 12 \beta_0, 1])$, $\na^2 ( \f {\p_r \psi }r )$, $\f 1r \p_r (\f 1r \p_r \psi ) \in L^2_r ([L_0, \f 14 L_0] \times [1 - \f 12 \beta_0, 1])$ and
\begin{equation}\label{w2h23}
\|\psi\|_{H^4_r([L_0, \f 14 L_0] \times [1 - \f 12 \beta_0, 1])}\leq C_* \|  F_0 \|_{H^2_r (D)}.
\end{equation}

{\bf Step 2.} The $H^4_r $ estimate near the axis $r = 0$. To eliminate the singularity of $\f 1r \p_r \psi$ at $r = 0$, we transform the function $\psi (x_1, r)$ back to the Cartesian coordinates and define $\check{\psi} (x_1, x') = \psi (x_1, |x'|)$. It follows from \eqref{linearized2} that
\begin{equation*}
\begin{cases}
b_{11} \p_{x_1}^2 \check{\psi }  + 2 b_{12} \sum_{i =2}^3 x_i \p_{x_1 x_i}^2 \check{\psi } + \sum_{i =2}^3  \p_{x_i }^2 \check{\psi }  + b_1 \p_{x_1} \check{\psi } + b_2 \sum_{i =2}^3 x_i \p_{x_i} \check{\psi } = f_0 ,  \\
  \check{\psi } (L_0, x') = 0,
\end{cases}
\end{equation*}
where
\be\no
& & b_{11} (x_1, x') = k_{11} (x_1, |x'|),  \q  b_{12} (x_1, x') = \f  { k_{12} (x_1, |x'|)}{|x'|}, \\\no
& & b_1 (x_1, x') = k_1 (x_1, |x'|), \q  b_2 (x_1, x') = \f{k_2 (x_1, |x'|)}{|x'|},  \q f_0 (x_1, x') = F_0 (x_1, |x'|).
\ee
It follows from Lemma \ref{coe-estimate} that $b_{11}, b_1  \in H^3 (\Om ) \hookrightarrow C^{1, \f 12} (\ol{\Om })$, $b_{12}, b_2 \in H^2 (\Om ) \hookrightarrow C^{0, \f 12} (\ol{\Om})$ and
\be\no
&&\|b_{11}-\bar{k}_{11}\|_{H^3(\Omega)}\leq \|k_{11}-\bar{k}_{11}\|_{H^3_r(D)},\ \ \|b_{1}-\bar{k}_1\|_{H^3(\Omega)}\leq\|k_{1}-\bar{k}_{1}\|_{H^3_r(D)},\\\no
&&\|b_{12}\|_{H^2(\Omega)}\leq \|\frac{k_{12}}{r}\|_{H^2_r(D)},\ \ \|b_{2}\|_{H^2(\Omega)}\leq \|\frac{k_{2}}{r}\|_{H^2_r(D)}, \\\no
&&\| \p_{x_2} (b_{12} x_3) \|_{L^{\infty} (\Om )}=\|\frac{x_2x_3}{r^2}(\p_r k_{12}-\frac{k_{12}}{r})\|_{L^{\infty}_r(D)}\le \|\frac{k_{12}}{r}\|_{H^2_r(D)} + \| \p_r k_{12} \|_{H^2_r (D)}, \  \\\no
&&\| \p_{x_2} (b_2 x_3) \|_{L^{\infty}(\Om)}=\|\frac{x_2x_3}{r^2}(\p_r k_{2}-\frac{k_{2}}{r})\|_{L^{\infty}_r(D)} \le \|\frac{k_2}{r}\|_{H^2_r(D)} + \| \p_r k_2 \|_{H^2_r (D)},\\\no
&&\|x_3\p_{x_2}^2 b_{12}\|_{L^4(\Omega)}=\|\frac{x_2^2x_3}{r^3}(\p_r^2 k_{12}-2\p_r(\frac{k_{12}}{r}))+\p_r(\frac{k_{12}}{r})(1-\frac{x_2^2}{r^2})\frac{x_3}{r}\|_{L^4(\Omega)}\\\no
&&\leq C_*(\|\p_r k_{12}\|_{H^2_r(D)}+ \|\frac{k_{12}}{r}\|_{H_r^2(D)}.
\ee

Define $\check{v}_2 = \p_{x_2} \check{\psi}$, then $\check{v }_2 (x_1, x')$ solves
\begin{equation*}
\begin{cases}
b_{11} \p_{x_1}^2 \check{v}_2 + 2 b_{12} \sum_{i = 2}^3 x_i \p_{x_1 x_i}^2 \check{v}_2  + \sum_{j =2}^3 \p_{x_j}^2 \check{v}_2  + b_3 \p_{x_1} \check{v}_2 + b_2 \sum_{i = 2}^3 x_i \p_{x_i} \check{v}_2 = f_1, \ \ \ \text{in }\Omega_1,\\
\check{v}_2 (L_0, x') = 0,\ \ \ \ \ \forall x'\in \{x'\in \mathbb{R}^2: |x'|<1\},
\end{cases}
\end{equation*}
where
\be\no
&&\Omega_1= \{(x_1,x'): L_0<x_1<\frac{1}{8}L_0, |x'|<1\},  \q b_3 =  b_1 +  2 \p_{x_2} ( b_{12} x_2) , \\\no
&& f_1 = \p_{x_2} f_0 - \p_{x_2 } b_{11} \p_{x_1}^2 \check{\psi } - 2 \p_{x_2} b_{12} x_3 \p_{x_1 x_3}^2 \check{\psi } - \p_{x_2} b_1 \p_{x_1} \check{\psi } - \p_{x_2} b_2  x_3 \p_{x_3} \check{\psi } - \p_{x_2} (b_2 x_2) \check{v}_2.
\ee
It follows from \cite[Theorem 9.11 and 9.13]{gt} and \eqref{H2} that for $\Om_1\supset\supset\Omega_2:=\{(x_1,x'): L_0<x_1<\frac{3}{16}L_0, |x'|\leq 1-\frac{1}{4}\beta_0\}$
\be\no
&& \iiint_{\Om_2} |\na^2 \p_{x_2} \check{\psi }|^2 dx' dx_1  \le  c_* \b(
\iiint_{\Om_1 } |\p_{x_2} \check{\psi} (x_1, x')|^2 dx' dx_1
+ \iiint_{\Om_1 } |f_1 (x_1, x')|^2 dx' dx_1
\b) \\\no
&& \le  c_* ( \| \check{\psi} \|_{H^1 (\Om_1)}^2 +  \| F_0 \|_{H^1 (\Om )}^2 +  \| \check{\psi} \|_{H^2 (\Om_1)}^2 (\|\p_{x_2}b_{11}\|_{L^{\infty} (\Om_1 )}^2 + \| x_2\p_{x_2} b_{12}  \|_{L^{\infty} (\Om_1 )}^2 \\\label{check2}
&& \q + \|\p_{x_2}b_1\|_{L^{\infty}(\Om_1)}^2 + \|b_2\|_{H^2 (\Om )}^2+\|x_i\p_{x_2}b_2\|_{L^{\infty}(\Om_1)}))
 \le  C_* \| F_0 \|_{H^1_r (D)}^2,
\ee
 where $c_*$ depends on $\Om_2$, $\Om_1$, the $L^{\infty}$ norm of $b_{11}$, $b_{12} x_i$, $b_2 x_i$ ($i = 2,3$), $b_3$ and the moduli of continuity of $b_{11}$, $b_{12} x_i$ ($i = 2,3$), and thus $C_*$ depends on the $H^3_r (D)$ norms of  $k_{11}$, $k_1$, the $H^2_r (D)$ norms of $\p_r k_{12}$, $\f {k_{12}}r$, $\p_r k_2$, $\f {k_2}r$ and the $L^{\oo}_r (D)$ norm of $k_{12}$, $\p_{x_1} k_{12}$ and $k_2$. Similarly, one derive the estimate of $\na^2 \p_{x_3} \check{\psi}$.

Since the following equality holds in $\Omega_1$,
\begin{equation}\label{check3}
\p_{ x_1}^2 \check{\psi }  =  \f 1{b_{11}} \b(f_0 - 2 b_{12} \sum_{j =2}^3 x_j \p_{x_1 x_j}^2 \check{\psi } - \sum_{i =2}^3  \p_{ x_i}^2 \check{\psi } - b_1 \p_{x_1} \check{\psi } - b_2 \sum_{j =2}^3 x_j \p_{x_j} \check{\psi }  \b),
\end{equation}
thus one has
\begin{equation*}
\|\check{\psi}\|_{H^3 (\Om_2)} \le C_* \| F_0 \|_{H^1_r (D)}.
\end{equation*}

Define $\check{w}_2 = \p_{x_2} \check{v}_2$, then $\check{w}_2 (x_1, x')$ solves
\begin{equation*}
\begin{cases}
  b_{11} \p_{x_1}^2 \check{w}_2 + 2 b_{12} \sum_{i = 2}^3 x_i \p_{x_1 x_i}^2 \check{w}_2 + \sum_{i = 2}^3 \p_{x_i}^2 \check{w}_2 + b_4 \p_{x_1} \check{w}_2 + b_2 \sum_{i = 2}^3 x_i \p_{x_i} \check{w}_2 = f_2, &  \\
  \check{w}_2 (L_0, x') = 0,
\end{cases}
\end{equation*}
where
\be\no
&& b_4= 2 \p_{x_2} (b_{12} x_2) + b_3 = b_1 + 4 \p_{x_2} (b_{12} x_2), \\\no
&& f_2 =  \p_{x_2} f_1 - \p_{x_2} b_{11} \p_{x_1}^2 \check{v}_2 - 2 \p_{x_2} b_{12} x_3 \p_{x_1 x_3}^2 \check{v}_2 - \p_{x_2} b_3 \p_{x_1} \check{v}_2 - \p_{x_2} b_2 x_3 \p_{x_3} \check{v}_2 - \p_{x_2} (b_2 x_2) w_2 
.
\ee
Using the similar approach as in \eqref{check2}, for any $\Om_2\supset\supset \Om_3:=\{(x_1,x'): L_0<x_1<\frac{1}{4}L_0, |x'|\leq 1-\frac{1}{2}\beta_0\}$, one has
\be\no
&& \iiint_{\Om_3} |\na^2 \p_{x_2}^2 \check{\psi}|^2 dx' dx_1 \le c_* \b( \iiint_{\Om_2} |\p_{x_2}^2 \check{\psi} (x_1, x')|^2 dx' dx_1 + \iiint_{\Om_2} |f_2 (x_1, x')|^2 dx' dx_1 \b) \\\no
&& \le
c_* \| \check{\psi} \|_{H^2 (\Om )}^2 + c_* \| F_0 \|_{H^2_r (D)}^2
+ c_* \| \na^2 \p_{x_2} \check{\psi} \|_{L^2 (\Om_2)}^2 ( \| \p_{x_2} b_{11} \|^2_{L^{\oo} (\Om)} + \| \p_{x_2} (b_{12} x_3) \|_{L^{\oo} (\Om)}^2)\\\no
&& \q
+ c_* \| \na^2 \check{\psi} \|^2_{L^4 (\Om_2)} (\| \p_{x_2}^2 b_{11} \|^2_{L^4 (\Om)}
+ \| \p_{x_2}^2 (b_{12} x_3) \|_{L^4 (\Om)}^2 + \| \p_{x_2} b_1 \|_{L^4 (\Om )}^2 + \| \p_{x_2}^2 (b_{12} x_2) \|^2_{L^4 (\Om)}  \\\no
&& \q
+ \| \p_{x_2} (b_2 x_2) \|^2_{L^4 (\Om )} + \| \p_{x_2} (b_2 x_3) \|_{L^4 (\Om)}^2  )
+ c_* \| \na \check{\psi} \|_{L^{\oo} (\Om_2)}^2 (\| \p_{x_2}^2 b_1 \|_{L^2 (\Om )}^2 + \| \p_{x_2}^2 b_2 x_2 \|^2_{L^2 (\Om)} \\\no
&& \q
+ \| \p_{x_2}^2 (b_2 x_2) \|_{L^2 (\Om )}^2)
\le C_* \| F_0 \|_{H^2_r (D)}^2. 
\ee
The estimates of $\na \p_{x_2 x_3}^2 \psi $ and $\na^2 \p_{x_3 }^2 \psi $ can be derived similarly. Again utilizing \eqref{check3}, one gets the estimate of $\p_{x_1}^4 \check{\psi}$. Finally,
\begin{equation*}
\iiint_{\Om_3} |\na^3 \check{\psi }|^2 + |\na^4 \check{\psi }|^2 dx_1 dx' \le C_* \iint_D (|F_0|^2 + |\na F_0|^2 + |\na^2 F_0|^2) r dr dx_1.
\end{equation*}
According to the norms equivalence \eqref{equiv} in Appendix \S\ref{appendix}, together with \eqref{w2h23}, the estimate \eqref{H31} holds.

\textbf{Step 3.} Verify the boundary condition. Differentiating \eqref{linearized2} with respect to $r$ in $D_{\f 18}$ and evaluating at $(x_1, 0)$, one derives that
\begin{equation*}
\p_r^3 \psi (x_1, 0) + \p_r (\f 1r \p_r \psi ) (x_1, 0) = 0.
\end{equation*}
Since
\begin{equation*}
\p_r (\f 1r \p_r \psi ) (x_1, 0) = \lim_{r \to 0^+} \f {r \p_{rr}^2 \psi - \p_r \psi }{r^2} = \lim_{r \to 0^+} \f {r \p_r^3 \psi }{2r} = \f 1r \p_r^3 \psi (x_1, 0),
\end{equation*}
then
\begin{equation}\label{c1}
\p_r^3 \psi (x_1, 0) =0, \q \forall x_1 \in [L_0, \f {L_0}8].
\end{equation}
\end{proof}

 To improve the regularity of $\psi $ in $[\f {3 L_0}8, L_1] \times [0, 1]$, we adopt an approach introduced by Kuzmin \cite{Kuzmin2002} and extend our problem to an auxiliary problem in a longer cylinder where the equation in \eqref{linearized2} becomes elliptic near the exit of the new cylinder. Firstly, our background solution is extended to $[L_0, L_2]$ where $L_2 = 2 L_1$ by simply extending the function $\bar{f}$ to $[L_0, L_2]$ so that $\bar{f} (x_1) $ is a $C^4$ differentiable function on $[L_0, L_2]$ and $\bar{f} (x_1)$ is positive on $(0, L_2]$. In the meanwhile, one can extend $(\bar{u}, \bar{\rho})$ to $[L_0, L_2]$ using the theory of ordinary differential equation so that the functions $\bar{k}_1$, $\bar{k}_{11}$ defined in \eqref{5} also satisfy the properties in \eqref{6}-\eqref{7} on $[L_0, L_2]$ if $d_0$ is chosen to be large enough.

 Let $l = \f {L_1}{20}$. We define two non-increasing cut-off functions on $[L_0, L_2]$ as follows
 \begin{equation*}
 \zeta_1 (x_1) = \begin{cases}
               1, & \mbox{if } L_0 \le x_1 \le L_1 + 2 l  \\
               0, & \mbox{if } L_1 + 4 l \le x_1 \le L_2
             \end{cases},
             \q
 \zeta_2 (x_1) = \begin{cases}
               1, & \mbox{if } L_0 \le x_1 \le L_1 + l \\
               0, & \mbox{if } L_1 + 2 l \le x_1 \le L_2
             \end{cases}.
 \end{equation*}
 Set
 \be\no
 \bar{a}_{11} (x_1) = \bar{k}_{11} (x_1) \zeta_1 (x_1) + (1 - \zeta_1 (x_1)), \\\no
 \bar{a}_1 (x_1) = \bar{k}_1 (x_1) \zeta_2 (x_1) - k_0 (1 - \zeta_2 (x_1)),
 \ee
 where $k_0$ is a positive constant to be specified later. Then
 \begin{equation*}
 \bar{a}_{11} (x_1 ) = \begin{cases}
                         \bar{k}_{11} (x_1), & \mbox{if } L_0 \le x_1 \le L_1 + 2 l \\
                         1 , & \mbox{if } L_1 + 4 l \le x_1 \le L_2
                       \end{cases},
 \q
 \bar{a}_1 (x_1) = \begin{cases}
                     \bar{k}_1 (x_1), & \mbox{if } L_0 \le x_1 \le L_1 + l \\
                     - k_0 , & \mbox{if }L_1 + 2 l \le x_1 \le L_2
                   \end{cases},
 \end{equation*}
 and for $j = 0, 1 ,2 ,3$
 \be\no
 && 2 \bar{a}_1 + (2 j -1) \bar{a}_{11}' = 2 \bar{k}_1 \zeta_2 + (2 j -1) \bar{k}_{11}' \zeta_1 + (2j -1) (\bar{k}_{11} - 1) \zeta_1' - 2 k_0 (1 - \zeta_2)\\\no
 && = \begin{cases}
        2 \bar{k}_1 + (2 j -1) \bar{k}_{11}' \le \ka_* < 0, & \mbox{if } L_0 \le x_1 \le L_1 + l, \\
        2 \bar{k}_1 \zeta_2 + (2 j -1) \bar{k}_{11}' - 2 k_0 (1 - \zeta_2), & \mbox{if } L_1 + l \le x_1 \le L_1 + 2 l, \\
        (2 j -1) \bar{k}_{11}' \zeta_1 + (2 j -1) (\bar{k}_{11} - 1) \zeta_1' - 2 k_0, & \mbox{if } L_1 + 2 l \le x_1 \le L_1 + 4 l, \\
        - 2 k_0 , & \mbox{if } L_1 + 4 l \le x_1 \le L_2.
      \end{cases}
 \ee
 Therefore, for sufficiently large $k_0$, $d_0 > 0$, the following inequalities hold, for $\forall x_1 \in [L_0, L_2]$,
 \be\label{8}
 2 \bar{a}_1 + (2 j -1) \bar{a}_{11}' \le - \ka_* < 0, \q j = 0, 1, 2, 3, 4,\\\label{9}
 (\bar{a}_1 + j \bar{a}_{11}') d - \f 12 (\bar{a}_{11} d)' \ge 4, \q  j = 0, 1, 2, 3,
 \ee
 where $d (x_1) = 6 (x_1 - d_0) < 0$ for any $x_1 \in [L_0, L_2]$.

 Furthermore, we introduce the extension operator $\mc{E}$ that extends a function $f (x_1, r)$ from $ D $ to $D_2 \co  \{(x_1, r ): L_0 < x_1 < L_2, 0 < r < 1\} $ as
\begin{equation*}
\mc{E} (f) (x_1, r) = \begin{cases}
                   f (x_1, r), & \mbox{if } (x_1, r) \in D, \\
                   \sum_{j =1}^4 c_j f (L_1 + \f 1j (L_1 - x_1), r) , & \mbox{if } (x_1, r) \in (L_1, L_2) \times (0, 1),
                 \end{cases}
\end{equation*}
where the constants $c_j$ are uniquely determined by the following algebraic equations
\begin{equation*}
\sum_{j =1}^4 (- \f 1j)^k c_j = 1, \q k = 0, 1, 2,3.
\end{equation*}
The extension operator $\mc{E} $ is a bounded operator from $H^j_r (D )$ to $H^j_r (D_2)$ for any $j = 1, 2, 3, 4$. Hence, we can define the extension of the operator $\mc{L}$ in \eqref{linearized2} to the domain $D_2$ as follows
\begin{equation*}
 a_{11} = \bar{a}_{11} + \mc{E} (k_{11} - \bar{k}_{11}),
\q  a_{12} = a_{21} = \mc{E} (k_{12}), \q  a_1 = \bar{a}_1 + \mc{E} (k_1 - \bar{k}_1), \q
a_2 =  \mc{E} (k_2 ),
\q G_0 = \mc{E} F_0.
\end{equation*}
Then, similar to Lemma \ref{coe-estimate}, there hold
\begin{equation}\label{coe2}
\begin{cases}
  \| a_{11} - \bar{a}_{11} \|_{H^3_r (D_2)} + \| a_1 - \bar{a}_1 \|_{H^3_r (D_2)} +\|\p_r a_{11}\|_{L^{\infty}_r(D_2)}+\|\p_r a_{1}\|_{L^{\infty}_r(D_2)} \le C_* ( \eps + \de_0),   \\
  \| \f {a_{12}}r  \|_{H^2_r (D_2)} +  \| a_{12} \|_{L^{\oo}_r (D_2)} + \| \p_r a_{12} \|_{H^2_r (D_2)} + \| \p_{x_1} a_{12} \|_{L^{\oo}_r (D_2)} \le C_* ( \eps + \de_0), \\
   \| \f {a_2}r  \|_{H^2_r (D_2)} + \| a_2 \|_{L^{\oo}_r (D_2)} + \| \p_r a_2 \|_{H^2_r (D_2)} +   \| \p_{x_1} a_2 \|_{L^{\oo}_r (D_2)}
   \le C_* ( \eps + \de_0)^2, \\
   \| \p_{x_1}^2 a_{12} \|_{L^4_r (D_2)} + \| \na \p_{x_1}^2 a_{12} \|_{L^2_r (D_2)} \le C_0 (\eps + \de_0), \\
   \| \p_{x_1}^2 a_2 \|_{L^2_r (D_2)} + \| \na \p_{x_1}^2 a_2 \|_{L^2_r (D_2)} \le C_0 (\eps + \de_0)^2,\\
   \| G_0  \|_{H^3_r (D)} \le C_* (  \| F \|_{H^3_r (D)} + \| \mc{F} \|_{H^2_r (D )} )
  \le C_* (\epsilon + ( \eps + \de_0)^2),\\
  a_{12} (x_1, 0 ) = a_{12} (x_1, 1) = \p_{r}^2 a_{12} (x_1, 0) = \p_{r}^2 a_{12} (x_1, 1) = 0, \\
  \p_r a_{11} (x_1, 0) = \p_r a_{11} (x_1, 1) = \p_r a_1 (x_1, 0) = \p_r a_1 (x_1, 1) = 0, \q x_1 \in [L_0, L_2], \\
  a_2 (x_1, 0) = a_2 (x_1, 1) = 0, \q \p_r G_0 (x_1, 0 ) =\p_r G_0 (x_1, 1 ) = 0,\\
  a_{12}(L_0, r) = 0, \ \forall r\in [1 - \beta_0,1].
\end{cases}
\end{equation}

Let us consider the following auxiliary problem in domain $D_2$,
\begin{equation}\label{au1}
\begin{cases}
  \mc{M} \Psi = a_{11} \p_{ x_1}^2 \Psi + 2 a_{12} \p_{x_1 r}^2 \Psi + \p_{r}^2 \Psi + \f 1r \p_r \Psi  + a_1 \p_{x_1} \Psi + a_2 \p_r \Psi = G_0 , &  (x_1, r) \in D_2, \\
  \Psi (L_0, r) = 0 , & \forall r \in [0, 1],\\
  \p_r \Psi (x_1, 0) = \p_r \Psi (x_1, 1) = 0, & \forall x_1 \in [L_0, L_2],\\
  \p_{x_1} \Psi (L_2, r) = 0, & \forall r \in [0, 1].
\end{cases}
\end{equation}

We will next prove that there exists a unique $H^2_r $ strong solution $\Psi $ to \eqref{au1} and derive the higher order estimates for $\na \p_{x_1 }^2 \Psi $ and $\na \p_{x_1}^3 \Psi$ in the subregion $(\f 38 L_0, L_1 + 12 l) \times (0,1)$. Further, it is verified that $\psi = \Psi $ on $D$, thus one obtains the estimates for $\na \p_{x_1}^2 \psi$ and $\na \p_{x_1}^3 \psi$ on $D$.

To find a solution to \eqref{au1}, we still resort to the singular perturbation problem
\begin{equation}\label{au2}
\begin{cases}
  \mc{M}^{\si } \Psi^{\si }  = \si \p_{x_1}^3 \Psi^{\si }  + a_{11} \p_{x_1}^2 \Psi^{\si }  + 2 a_{12} \p_{x_1 r}^2 \Psi^{\si }  + \p_{r}^2 \Psi^{\si } \\
  \q \q + \f 1r \p_r \Psi^{\si } + a_1 \p_{x_1} \Psi^{\si }  + a_2 \p_r \Psi^{\si }  = G_0 , & \forall (x_1, r) \in D_2, \\
  \Psi^{\si }  (L_0, r) = \p_{x_1}^2 \Psi^{\si } (L_0, r) = 0 , & \forall r \in [0, 1],\\
  \p_r \Psi^{\si }  (x_1, 0) = \p_r \Psi (x_1, 1) = 0, & \forall x_1 \in [L_0, L_2],\\
  \p_{x_1}  \Psi^{\si }  (L_2, r) = 0, & \forall r \in [0, 1].
\end{cases}
\end{equation}

One could prove the following $H^2_r (D_2)$ estimate for the solution $\Psi^{\si}$ to \eqref{au2}.

\begin{lemma}\label{aH2}
  There exist $\epsilon_*>0,\de_* > 0$ depending only on the background flow and the boundary data, such that if $0<\epsilon<\epsilon_*, 0 < \de_0 \le \de_*$ in \eqref{coe2}, the classical solution $\Psi^{\si }$ to \eqref{au2} satisfies 
  \begin{equation}\label{aH1}
   \si \iint_{D_2} |\p_{x_1}^2 \Psi^{\si }|^2 r dr dx_1 + \iint_{D_2} (|\Psi^{\si }|^2 + |\na \Psi^{\si}|^2 ) r dr dx_1 \le C_* \iint_{D_2} G_0^2 r dr dx_1, \end{equation}
  \begin{equation}\label{aH21}
   \si \int_{\f 78 L_0}^{L_1 + 16 l} \int_0^1 |\p_{x_1}^3 \Psi^{\si}|^2 r dr dx_1 +
   \iint_{D_2} ( |\na^2 \Psi^{\si }|^2 + |\f 1r \p_r \Psi^{\si}|^2) r dr dx_1 
  \le C_* \iint_{D_2} (|G_0|^2 + |\na G_0|^2 ) r dr dx_1,
  \end{equation}
 where $C_*$ depends only on the $H^3_r (D_2)$ norms of $a_{11}$, $a_1$, the $H^2_r (D_2)$ norms of $\p_r a_{12}$, $\f {a_{12}}r$, $\p_r a_2$ and the $L^{\oo}_r (D_2)$ norms of $a_{12}$, $\p_{x_1} a_{12}$ and $a_2$.
\end{lemma}
\begin{proof}
  The proof is quite similar to that of Lemma \ref{H1estimate-ap}. We omit the superscript $\si$ for simplicity of notations. According to the boundary conditions in \eqref{au2}, the boundary integral term $\int_0^1 \si d \p_{x_1} \Psi \p_{x_1 }^2 \Psi r \b|_{L_0}^{L_2} dr $ vanishes. Since \eqref{8}-\eqref{9} hold, one can derive \eqref{aH1} as in Lemma \ref{H1estimate-ap}. Same argument as in Lemma \ref{H1estimate-ap} yields
  \begin{equation}\label{aH22}
  \int_{L_0}^{\f { L_0}2} \int_0^1 |\na \p_{x_1} \Psi |^2 r dr dx_1 \le C_* \| G_0 \|_{L^2_r (D_2)}^2.
  \end{equation}

%

  Choose a monotonic increasing cut-off function $\eta_3 \in C^{\oo} ([L_0, L_2])$ such that
  \begin{equation*}
  \eta_3 (x_1) = \begin{cases}
                 0, & L_0 \le x_1 \le L_1 + 2 l, \\
                 1, & L_1 + 4 l \le x_1 \le L_2.
               \end{cases}
  \end{equation*}
 Multiplying \eqref{au2} by $\eta_3^2 \p_{x_1}^2 \Psi$, integrations by parts yield that
  \be\no
&& \iint_{D_2} \eta_3^2 (a_{11} (\p_{x_1}^2 \Psi)^2 + (\p_{x_1 r}^2 \Psi)^2) r dr dx_1
 + \f {\si }2 \int_0^1 (\p_{x_1 }^2 \Psi (L_2, r))^2 r dr
\\\no
&  = &
\si \iint_{D_2} \eta_3 \eta_3' (\p_{x_1}^2 \Psi)^2 r dr dx_1
+ \iint_{D_2} \eta_3^2 \p_{x_1 }^2 \Psi ( G_0 - a_1 \p_{x_1} \Psi - a_2 \p_r \Psi ) r dr dx_1
 \\\no
&& - 2 \iint_{D_2} \eta_3^2 a_{12} \p_{x_1}^2 \Psi \p_{x_1r}^2 \Psi r dr dx_1
- 2 \iint_{D_2}  \eta_3 \eta_3' \p_{x_1 r}^2 \Psi \p_r \Psi r dr dx_1.
\ee
Using \eqref{aH1} to control the term involving $\si$, one gains
\begin{equation}\label{aH23}
\int_{L_1 + 4 l }^{L_2} \int_0^1 |\na \p_{x_1} \Psi|^2 r dr dx_1 \le C_* \iint_{D_2} G_0^2 r dr dx_1 .
\end{equation}

Set $W_1 = \p_{x_1} \Psi$. Then $W_1$ satisfies
\begin{equation}\label{aw1}
\begin{cases}
  \si \p_{x_1}^3 W_1 + a_{11} \p_{x_1}^2 W_1 + 2 a_{12} \p_{x_1 r}^2 W_1 + \p_{r}^2 W_1 + \f 1r \p_r W_1 + a_3 \p_{x_1} W_1 + a_4 \p_r W_1 = G_1,  \\
  \p_{x_1} W_1 (L_0, r)  = 0, \q \forall r \in [0, 1], \\
  \p_r W_1 (x_1, 0) = \p_r W_1 (x_1, 1) = 0, \q \forall x_1 \in [L_0, L_2], \\
   W_1 (L_2, r) = 0, \q \forall r \in [0, 1],
\end{cases}
\end{equation}
where
\begin{equation*}
a_3 = \p_{x_1} a_{11} + a_1 , \quad
a_4 = 2 \p_{x_1} a_{12} + a_2 , \quad
G_1 = \p_{x_1} G_0 - \p_{x_1} a_2 \p_r \Psi - \p_{x_1} a_1 \p_{x_1} \Psi.
\end{equation*}
Set $0 \le \eta_4 (x_1) \le 1$ be a smooth cut-off function on $[L_0, L_2]$ as
\begin{equation*}
\eta_4 (x_1) = \begin{cases}
               0, & L_0 \le x_1 \le \f {15}{16} L_0, \\
               1, & \f 78 L_0 \le x_1 \le L_1 + 16 l, \\
               0, & L_1 + 18 l \le x_1 \le L_2.
             \end{cases}
\end{equation*}
Multiplying \eqref{aw1} by $\eta_4^2 d \p_{x_1} W_1$ and integrating over $D_2$, some integrations by parts allow us to obtain
\be\no
&& - \si \iint_{D_2} \eta_4^2 d (\p_{x_1}^2 W_1)^2 r dr dx_1
+ \iint_{D_2} (\eta_4^2 d a_3 - \f 12 \p_{x_1} (\eta_4^2 d a_{11}) - \eta_4^2 d \p_r a_{12} - \eta_4^2 d \f {a_{12}}r ) (\p_{x_1} W_1)^2 r dr dx_1 \\\no
&& \q + \iint_{D_2} ( \eta_4 \eta_4' d + 3 \eta_4^2) (\p_r W_1)^2 r dr dx_1
+ \iint_{D_2} \eta_4^2 d ( 2 \p_{x_1} a_{12} + a_2 )  \p_{x_1} W_1 \p_r W_1 r dr dx_1
\\\no
&&
= \iint_{D_2 } \eta_4^2 d \p_{x_1} W_1 G_1 r dr dx_1
+ \si \iint_{D_2} \p_{x_1} (\eta_4^2 d) \p_{x_1} W_1 \p_{x_1 }^2 W_1 r dr dx_1.
\ee
Note that
\be\no
&& \si \iint_{D_2} \p_{x_1} (\eta_4^2 d) \p_{x_1} W_1 \p_{x_1 }^2 W_1 r dr dx_1  \\\no
&& \le - \f {\si}4 \iint_{D_2} \eta_4^2 d |\p_{x_1}^2 W_1|^2 r dr dx_1
+ C_* \si \iint_{D_2} ( \eta_4^2 + (\eta_4')^2 ) |\p_{x_1} W_1|^2 r dr dx_1,
\ee
and the support of $\eta_4'$ is contained in $(\f {15}{16} L_0, \f 78 L_0) \cup (L_1 + 16 l, L_1 + 18 l)$ such that $\iint_{D_2} \eta_4 \eta_4' d  (\p_r W_1)^2  r dr dx_1$ can be controlled by \eqref{aH22} and \eqref{aH23}. By \eqref{9}, \eqref{coe2} and \eqref{aH1}, there holds
\begin{equation}\label{aH24}
 \iint_{D_2} ( \si \eta_4^2 (\p_{x_1}^2 W_1)^2 + \eta_4^2 |\na W_1|^2 ) r dr d x_1
 \le C_* \iint_{D_2} ( |\na W_1|^2 +  G_1^2 ) r dr d x_1.
\end{equation}
Since
\be\no
&& \| G_1 \|_{L^2_r (D_2)}   \le  \| \p_{x_1} G_0 \|_{L^2_r (D_2)} + \| \p_{x_1} a_2 \|_{L^{\oo}_r (D_2)} \| \p_r \Psi \|_{L^2_r (D_2)} + \| \p_{x_1} a_1 \|_{L^{\oo}_r (D_2)} \| \p_{x_1} \Psi \|_{L^2_r (D_2)} \\\no
&& \le  C_* ( \| G_0 \|_{L^2_r (D_2)} +  \| \na G_0 \|_{L^2_r (D_2)}).
\ee
Combining with \eqref{aH22}, \eqref{aH23} and \eqref{aH24} give
\begin{equation}\label{aH25}
\si \int_{\f 78 L_0}^{L_1 + 16 l} \int_0^1 |\p_{x_1}^3 \Psi^{\si}|^2 r dr dx_1 + \iint_{D_2} |\na \p_{x_1} \Psi^{\si}|^2 r dr dx_1 \le C_* (\| G_0 \|_{L^2_r (D_2)}^2 + \| \na G_0 \|_{L^2_r (D_2)}^2).
\end{equation}

It remains to estimate $\p_{r}^2 \Psi + \f 1r \p_r \Psi $. Define $V_1 = \p_r \Psi$, then
\begin{equation}\label{av1}
\begin{cases}
  \si \p_{x_1}^3 \Psi + a_{11} \p_{x_1}^2 \Psi + 2 a_{12} \p_{x_1 r}^2 \Psi + \p_r V_1 + \f 1r V_1 + a_1 \p_{x_1} \Psi + a_2 \p_r \Psi = G_0,  \\
  V_1 (x_1, 0) = V_1 (x_1, 1) = 0, \q \forall x_1 \in [L_0, L_2], \\
  V_1 (L_0, r) = \p_{x_1}^2 V_1 (L_0, r) = 0, \q \forall r \in [0, 1], \\
  \p_{x_1} V_1 (L_2, r) = 0, \q \forall r \in [0, 1].
\end{cases}
\end{equation}
Since $a_{11} (L_0, r) > 0 $ for any $r \in [0, 1]$ and $d(x_1) < 0$ for any $x_1 \in [L_0, L_2]$, multiplying the equation in $\eqref{av1}$ by $- d(x_1) \p_{x_1} (\p_r V_1 + \f {V_1}r )$ and integrating over $D_2$, one gets
\begin{equation}\label{aH26}
 \iint_{D_2} ( |\na V_1|^2 + |\f {V_1}{r}|^2 ) r dr dx
 + \iint_{D_2} d \p_r a_{11} \p_{x_1} V_1 \p_{x_1}^2 \Psi r dr dx_1
\le C_* \iint_{D_2} (|G_0|^2 + |\na G_0|^2 ) r dr dx .
\end{equation}
This, together with \eqref{aH25}, derives that
\begin{equation*}
\si \int_{\f 78 L_0}^{L_1 + 16 l} \int_0^1 |\p_{x_1}^3 \Psi^{\si}|^2 r dr dx_1 +
\iint_{D_2} \b(  |\na^2 \Psi^{\si}|^2   + |\f {\p_r \Psi^{\si}}r|^2 \b) r dr dx_1
\le C_* \iint_{D_2} (|G_0|^2 + |\na G_0|^2) r dr dx_1.
\end{equation*}
\end{proof}

Then one can easily prove that
\begin{lemma}\label{equality}
  There exists a unique $H^2_r (D_2)$ strong solution $\Psi $ to \eqref{au1} with the estimate
  \begin{equation}\label{au100}
  \| \Psi \|_{H^2_r (D_2 )} \le C_* \| G_0 \|_{H^1_r (D_2)},
  \end{equation}
    where $C_*$ depends only on the $H^3_r (D_2)$ norms of $a_{11}$, $a_1$, the $H^2_r (D_2) $ norms of $\p_r a_{12}$, $\f {a_{12}}r$, $\p_r a_2$ and the $L^{\oo}_r (D_2)$ norms of $ k_{12}$, $\p_{x_1} k_{12}$ and $k_2$. Moreover, the solution $\Psi $ coincided with the $H^2_r $ strong solution $\psi$ to \eqref{linearized2} on the domain $D$.
\end{lemma}
\begin{proof}
  Thanks to \eqref{aH1}-\eqref{aH21}, the existence and uniqueness of strong $H^2_r$ solution $\Psi^{\si}$ to \eqref{au2} can be proved by employing the Galerkin method as in Lemma \ref{exist2}. Since the estimates \eqref{aH1}-\eqref{aH21} are uniformly in $\si $, one can further extract a subsequence $\{\Psi^{\si_j} \}_{j = 1}^{\oo}$ which converges weakly to $\Psi$ in $H^2_r (D_2)$ as $\si_j \to 0$. This function $\Psi $ satisfies \eqref{au100} and solves the problem \eqref{au1}.

  Let $v =\Psi - \psi $, then $v \in H^2_r (D )$ solves
 \begin{equation*}
  \begin{cases}
    k_{11} \p_{x_1 }^2 v + 2 k_{12} \p_{x_1 r}^2 v + \p_{r}^2 v + \f 1r \p_r v
    + k_1 \p_{x_1} v + k_2 \p_r v = 0, & (x_1, r) \in D , \\
    v (L_0, r) = 0, & r \in [0, 1], \\
    \p_r v (x_1, 0 ) = \p_r v (x_1, 1) = 0 , & x_1 \in [L_0, L_1].
  \end{cases}
  \end{equation*}
  An energy estimate as in Lemma \ref{H1estimate} implies that $\iint_{D} |\na v|^2 r dr dx_1 = 0$ and thus $\na v \equiv 0$. Since $v (L_0, r) = 0$, it follows that $v (x_1, r) \equiv 0$ on $\Om $.
\end{proof}

The following lemma gives the estimate for $\na \p_{x_1}^2 \Psi $ on the subregion $(\f 58 L_0, L_1 + 14l) \times (0, 1)$.
\begin{lemma}\label{aH3}
Under the assumptions in Lemma \ref{aH2}, the classical solution to \eqref{au2} satisfies
  \begin{equation}\label{aH31}
   \si \int_{\f 58 L_0}^{L_1 + 14l } \int_0^1 |\p_{x_1}^4 \Psi^{\si}|^2 r dr dx_1 + \int_{\f 58 L_0}^{L_1 + 14l } \int_0^1 |\na \p_{x_1 }^2 \Psi^{\si}|^2 r dr dx_1 \le C_{\sharp} \| G_0 \|_{H^2_r (D_2)}^2,
\end{equation}
  where $C_\sharp$ depends only on the $C^3 (\overline{D_2})$ norms of $a_{11}$, $a_1$, and the $C^2 (\overline{D_2})$ norms of $\p_r a_{12}$, $\f {a_{12}}r$ and $\p_r a_2$ and the $C^0 (\overline{D_2})$ norms of $a_{12} $, $\p_{x_1} a_{12}$, $\p_{x_1}^2 a_{12}$, $ a_2$ and $\p_{x_1}^2 a_2$.
\end{lemma}
\begin{proof}
  Let smooth cut-off functions $0 \le \eta_j (x) \le 1$ on $[L_0, L_2]$ for $j = 5, 6$ satisfy
  \begin{equation*}
  \eta_5 (x_1) = \begin{cases}
                 0, & L_0 \le x_1 \le \f 78 L_0 \\
                 1, & \f 34 L_0 \le x_1 \le \f 58 L_0 \\
                 0, & \f 12 L_0 \le x_1 \le L_2
               \end{cases},
  \quad
  \eta_6 (x_1) = \begin{cases}
                 0, & L_0 \le x_1 \le L_1 + 13 l \\
                 1, & L_1 + 14 l \le x_1 \le L_1 + 15 l \\
                 0, & L_1 + 16 l \le x_1 \le L_2
               \end{cases}.
  \end{equation*}

Multiplying \eqref{aw1} by $\eta_j^2 \p_{x_1}^2 W_1 $ for $j = 5, 6$ respectively, integrations by parts yields
  \be\no
&& \iint_{D_2}  \eta_j^2 ( a_{11} (\p_{x_1}^2 W_1)^2 + (\p_{x_1 r}^2 W_1)^2 )r dr dx_1 \\\no
&& =
\si \iint_{D_2} \eta_j \eta_j' (\p_{x_1}^2 W_1)^2 r dr dx_1
- 2 \iint_{D_2} \eta_j^2  a_{12}  \p_{x_1 }^2 W_1 \p_{x_1 r}^2 W_1 r dr dx_1\\\no
&& \q
- \iint_{D_2} 2 \eta_j \eta_j' \p_r W_1 \p_{x_1 r}^2 W_1 r dr dx_1
+ \iint_{D_2 } \eta_j^2 \p_{x_1 }^2 W_1 (G_1 - a_3 \p_{x_1} W_1 -  a_4  \p_r W_1) r dr dx_1.
\ee
 Since the supports of $\eta_j' (x_1)$ are contained in $[\f 78 L_0, L_1 + 16l]$ for all $j = 5, 6$, the first term on the right hand side can be controlled using \eqref{aH21}. Note that on the support of $\eta_j$ for $j = 5, 6$, $a_{11} \ge \f 12 \ka_* >0$. Then there holds
 \begin{equation}\label{aH32}
   \int_{\f 34 L_0}^{\f 58 L_0} \int_0^1 |\na \p_{x_1} W_1|^2 r dr dx_1
+ \int_{L_1 + 14 l}^{L_1 + 15 l } \int_0^1 |\na \p_{x_1} W_1|^2 r dr dx_1 \le C_{\sharp} \iint_{D_2} (G_0^2 + |\na G_0|^2 ) r dr dx_1,
 \end{equation}
where $C_{\sharp}$ depends only on the $C^1 (\overline{D_2})$ norms of $a_{11}, a_{12}$ and $a_1, a_2$.

Set $W_2 = \p_{x_1} W_1$. Then $W_2$ satisfies
\begin{equation}\label{aw2}
\begin{cases}
  \si \p_{x_1}^3 W_2 + a_{11} \p_{x_1}^2 W_2 + 2 a_{12} \p_{x_1 r}^2 W_2  + \p_{r}^2 W_2  + \f 1r \p_r W_2 + a_6 \p_{x_1} W_2 + a_7 \p_r W_2 = G_2, &  (x_1, r) \in D_2, \\
  W_2 (L_0, r) = 0, & \forall r \in [0, 1], \\
  \p_r W_2 (x_1, 0) = \p_r W_2 (x_1, 1) = 0, & \forall x_1 \in [L_0, L_2],
\end{cases}
\end{equation}
where
\be\no
&& a_6 =  \p_{x_1} a_{11} + a_3 = 2 \p_{x_1} a_{11} + a_1, \q
a_7 = 2 \p_{x_1} a_{12} + a_4 = 4 \p_{x_1} a_{12} +  a_2 , \\\no
&& G_2 = \p_{x_1} G_1 - \p_{x_1} a_3 \p_{x_1} W_1 - \p_{x_1} a_4 \p_r W_1  \\\no
&&  = \p_{x_1 }^2 G_0 - \p_{x_1}^2 a_2 \p_r \Psi  - \p_{x_1}^2 a_1 W_1  - 2 ( \p_{x_1}^2 a_{12} + \p_{x_1} a_2) \p_r W_1 - (\p_{x_1 }^2 a_{11} + 2 \p_{x_1} a_1) W_2
.
\ee

Define a smooth cut-off function $0 \le \eta_7 (x_1) \le 1$ on $[L_0, L_2]$ satisfying
\begin{equation*}
\eta_7 (x_1) = \begin{cases}
               0, & L_0 \le x_1 \le \f 34 L_0, \\
               1, & \f 58 L_0 \le x_1 \le L_1 + 14 l, \\
               0, & L_1 + 15 l \le x_1 \le L_2.
             \end{cases}
\end{equation*}

Multiplying \eqref{aw2} by $\eta_7^2 d \p_{x_1} W_2$ and integrating over $D_2$, integrations by parts yield that
 \be\no
 && - \si \iint_{D_2} \eta_7^2 d (\p_{x_1}^2 W_2)^2 r dr dx_1
 + \iint_{D_2} (\eta_7^2 d a_6 - \f 12  \p_{x_1} (\eta_7^2 d a_{11}) - \eta_7^2 d \p_r a_{12} - \eta_7^2 d \f {a_{12}}r ) (\p_{x_1} W_2)^2 r dr dx_1 \\\label{aw200}
 && \q
 + \iint_{D_2} ( \eta_7 \eta_7' d + 3 \eta_7^2 ) (\p_r W_2)^2 r dr dx_1
 + \iint_{D_2} \eta_7^2 d a_7  \p_{x_1} W_2 \p_r W_2 r dr dx_1 \\\no
 &&  =
  \iint_{D_2} \eta_7^2 d \p_{x_1} W_2 G_2 r dr dx_1
  + \si \iint_{D_2}  \p_{x_1} (\eta_7^2 d) \p_{x_1} W_2 \p_{x_1}^2 W_2 r dr dx_1.
 \ee
Then there holds
\begin{equation}\label{aH33}
\si \iint_{D_2} \eta_7^2 |\p_{x_1}^2 W_2|^2 r dr dx_1
+ \iint_{D_2} \eta_7^2 |\na W_2|^2 r dr dx_1
 \le C_* \iint_{D_2} ( \eta_7^2 G_2^2 + |\eta_7'|^2 |\na W_2|^2 ) r dr dx_1,
\end{equation}
where we used \eqref{8}-\eqref{9} for $j = 2$, and \eqref{coe2}.
Since the support of $\eta_7' (x_1)$ is contained in $[\f 34 L_0, \f 58 L_0] \cup [L_1 + 14 l, L_1 + 15 l]$, the term $\iint_{D_2} |\eta_7' (x_1)|^2 |\na W_2|^2 r dr dx_1 $ can be controlled by \eqref{aH32}. Also one has
\be\no
&& \| \eta_7 G_2 \|_{L^2_r (D_2) } 
\le
C_* \| G_0 \|_{H^2_r (D_2)}
+ \| \p_{x_1}^2 a_2 \|_{L^2_r (D_2)} \| \f 1r \p_r \Psi \|_{L^{\oo}_r (D_2)}
+ \| \p_{x_1}^2 a_1 \|_{L^4_r (D_2)} \| W_1 \|_{L^4_r (D_2)}
 \\\no
&& \quad
 + \| \p_{x_1}^2 a_{12} + \p_{x_1} a_2 \|_{L^{\oo}_r (D_2)}  \| \p_r W_1 \|_{L^2_r (D_2)}
 + \| \p_{x_1}^2 a_{11} + \p_{x_1} a_1 \|_{L^{\oo}_r (D_2)} \| \eta_7 W_2 \|_{L^2_r (D_2)}  \\\label{aH35}
&& \le
     C_* \| G_0 \|_{H^2_r (D_2) } + C_{\sharp} \| G_0 \|_{H^1_r (D_2)} + \| a_1\|_{H^3_r (D_2)} \| \p_{x_1} \Psi \|_{H^1_r (D_2)} \le C_{\sharp} \| G_0 \|_{H^2_r (D_2) } .
\ee
Then \eqref{aH31} can be derived from \eqref{aH33}-\eqref{aH35}.
\end{proof}

The higher order estimate for $\na \p_{x_1}^3 \Psi $ on the subregion $(\f 38 L_0, L_1 + 12l) \times (0, 1)$ is as follows.
\begin{lemma}\label{aH4}
Under the assumptions in Lemma \ref{aH2}, the classical solution to \eqref{au2} satisfies
\begin{equation}\label{aH41}
\si \int_{\f 38 L_0}^{L_1 + 12 l} \int_0^1 |\p_{x_1}^5 \Psi^{\si}|^2 r dr dx_1
+ \int_{\f 38 L_0}^{L_1 + 12 l} \int_0^1 |\na \p_{x_1}^3 \Psi^{\si}|^2 r dr dx_1
\le C_{\sharp}  \| G_0 \|_{H^3_r (D_2)}^2 ,
\end{equation}
 where $C_\sharp$ depends only on the $C^3 (\overline{D_2})$ norms of $a_{11}$, $a_1$, and the $C^2 (\overline{D_2})$ norms of $\p_r a_{12}$, $\f {a_{12}}r$, $\p_r a_2$, $\f {a_2}r$ and the $C^0 (\overline{D_2})$ norms of $ \p_{x_1}^{|\al|} a_{12} $ and $\p_{x_1}^{|\al|} a_2$ with $|\al| = 0, \cdots, 3$.

\end{lemma}
\begin{proof}
  Define smooth cut-off functions $0 \le \eta_j (x_1) \le 1$ on $[L_0, L_2]$ for $j = 8, 9$ such that
 \begin{equation*}
  \eta_8 (x_1) = \begin{cases}
                 0, & L_0 \le x_1 \le \f 58 L_0 \\
                 1, & \f 12 L_0 \le x_1 \le \f 38 L_0 \\
                 0, & \f 14 L_0 \le x_1 \le L_2
               \end{cases},
  \q
  \eta_9 (x_1) = \begin{cases}
                 0, & L_0 \le x_1 \le L_1 + 11l  \\
                 1, & L_1 + 12 l \le x_1 \le L_1 + 13 l \\
                 0, & L_1 + 14 l \le x_1 \le L_2
               \end{cases}.
  \end{equation*}

  Multiplying \eqref{aw2} by $\eta_j^2 \p_{x_1}^2 W_2$ for $j = 8, 9$ respectively, integrations by parts yield that
 \be\no
&&
 \iint_{D_2} \eta_j^2 ( a_{11} (\p_{x_1}^2 W_2)^2 + (\p_{x_1 r}^2 W_2)^2) r dr dx_1 = \si \iint_{D_2} \eta_j \eta_j' (\p_{x_1}^2 W_2)^2 r dr dx_1\\\no
&  & \q
 - 2 \iint \eta_j^2 a_{12} \p_{x_1}^2 W_2 \p_{x_1 r}^2 W_2 r dr dx_1
 - \iint_{D_2} \eta_j^2  a_7  \p_{x_1}^2 W_2 \p_r W_2 r dr dx_1\\\no
&& \q
 - \iint_{D_2} 2 \eta_j \eta_j' \p_r W_2 \p_{x_1 r}^2 W_2 r dr dx_1
+ \iint_{D_2} \eta_j^2 ( G_2 - a_6 \p_{x_1} W_2 ) \p_{x_1 }^2 W_2 r dr dx_1 .
\ee
Since the supports of $\eta_j'$ are contained in $[\f 58 L_0, L_1 + 14 l]$, the first term on the right hand side can be controlled using \eqref{aH31}. Then one obtains
\begin{equation}\label{aH42}
 \int^{\f 38 L_0}_{\f 12 L_0} \int_0^1 |\na \p_{x_1} W_2|^2 r dr dx_1 + \int_{L_1 + 12 l}^{L_1 + 13 l} \int_0^1 |\na \p_{x_1} W_2|^2 r dr dx_1
\le C_{\sharp} \| G_0 \|^2_{H^2_r (D_2)}.
\end{equation}

Set $W_3 = \p_{x_1} W_2$. Then $W_3$ satisfies
\begin{equation}\label{aw3}
\begin{cases}
  \si \p_{x_1}^3 W_3 + a_{11} \p_{x_1}^2 W_3 + 2 a_{12} \p_{x_1 r}^2 W_3 + \p_{r}^2 W_3 + \f 1r \p_r W_3 + a_8 \p_{x_1} W_3 + a_9 \p_r W_3 = G_3,  \\
  \p_r W_3 (x_1, 0) = \p_r W_3 (x_1, 1) = 0, \q \forall x_1 \in [L_0, L_2],
\end{cases}\end{equation}
where
\be\no
&& a_8 = \p_{x_1} a_{11} + a_6 = 3 \p_{x_1} a_{11} + a_1 , \q
a_9 = 2 \p_{x_1} a_{12} + a_7 = 6 \p_{x_1} a_{12} + a_2, \\\no
&& G_3 = \p_{x_1}^3 G_0 - \p_{x_1}^3 a_2 \p_r \Psi - \p_{x_1}^3 a_1 W_1
- 3(2\p_{x_1}^2 a_{12} + \p_{x_1} a_2) \p_r W_2
- (\p_{x_1}^3 a_{11} + 3 \p_{x_1 }^2 a_1) W_2\\\no
&& \q - 3 (\p_{x_1}^2 a_{11} + \p_{x_1} a_1) W_3
- (2 \p_{x_1}^3 a_{12} + 3 \p_{x_1 }^2 a_2) \p_r W_1.
\ee
Define a smooth cut-off function $0 \le \eta_{10} (x_1) \le 1$ on $[L_0, L_2]$ such that
\begin{equation*}
\eta_{10} (x_1) = \begin{cases}
                  0, & L_0 \le x_1  \le \f 12 L_0, \\
                  1, & \f 38 L_0 \le x_1 \le L_1 + 12 l, \\
                  0, & L_1 + 13 l \le x_1 \le L_2.
                \end{cases}
\end{equation*}

Multiplying \eqref{aw3} by $\eta_{10}^2 d \p_{x_1} W_3$, integrating over $\Om_2$ 
and using \eqref{8}-\eqref{9} for $j =3$, one gets
\begin{equation*}
\si \iint_{D_2} \eta_{10}^2 (\p_{x_1}^2 W_3)^2 r dr dx_1 + \iint_{D_2} \eta_{10}^2 |\na W_3|^2 r dr dx_1
\le C_* \iint_{D_2} ( \eta_{10}^2 G_3^2 + |\eta_{10}'| |\na W_3|^2 ) r dr dx_1.
\end{equation*}
The term $\iint_{D_2} |\eta_{10}'| |\na W_3|^2 r dr dx_1 $ can be controlled utilizing \eqref{aH42}, since the support of $\eta_{10}'$ is contained in $[\f 12 L_0, \f 38 L_0] \cup [L_1 + 12 l, L_1 + 13 l]$. The estimate of $\| \eta_{10} G_3 \|_{L^2}$ is as follows
\be\no
&& \| \eta_{10} G_3 \|_{L^2_r (D_2)}
\le  \| G_0  \|_{H^3_r (D_2)}
+ \| \p_{x_1}^3 a_2 \|_{L^{\oo}_r (D_2)} \| \eta_{10} \p_r \Psi\|_{L^2_r (D_2) }
+ \| \p_{x_1}^3 a_1 \|_{L^{\oo}_r (D_2)} \|\eta_{10} W_1 \|_{L^2_r (D_2)}
\\\no
&& \q
+ \| 3(2\p_{x_1 }^2 a_{12} + \p_{x_1} a_2) \|_{L^{\oo}_r (D_2)} \|\eta_{10} \p_r W_2 \|_{L^2_r (D_2) }
+ \| \p_{x_1}^3 a_{11} + 3 \p_{x_1 }^2 a_1 \|_{L^{\oo}_r (D_2)} \| \eta_{10} W_2 \|_{L^2_r (D_2)}
 \\\no
&& \q
+ \| 3 (\p_{x_1 }^2 a_{11} + \p_{x_1} a_1)\|_{L^{\oo}_r (D_2)}  \| \eta_{10 } W_3 \|_{L^2_r (D_2)}
+ \| 2 \p_{x_1}^3 a_{12} + 3 \p_{x_1}^2 a_2\|_{L^{\oo}_r (D_2)} \|\eta_{10} \p_r W_1  \|_{L^2_r (D_2)} \\\no
&& \le   C_{\sharp} \| G_0  \|_{H^3_r (D_2)}.
\ee
Here we use \eqref{aH31} to control $\| \eta_{10} \na W_2 \|_{L^2_r (D_2)}$ since the support of $\eta_{10}$ is contained in $(\f 58 L_0, L_1 + 14 l)$. Therefore, we have demonstrated the estimate \eqref{aH41}.
\end{proof}

Now we could improve the regularity of the solution $\psi$ to \eqref{linearized2} to be $H^4_r (D)$.
\begin{lemma}\label{H4first}
  The $H^2_r (D) $ strong solution $\psi$ to \eqref{linearized2} indeed belongs to $H^4_r (D)$ with
  \begin{equation}\label{aH400}
  \| \psi \|_{H^4_r (D)}  \le C_{\sharp} ( \| F \|_{H^3_r (D )} + \| \mc{ F}  \|_{H^2_r (D)}),
  \end{equation}
 where $C_\sharp$ depends only on the $C^3 (\overline{D_2})$ norms of $a_{11}$, $a_1$, and the $C^2 (\overline{D_2})$ norms of $\p_r a_{12}$, $\f {a_{12}}r$, $\p_r a_2$, $\f {a_2}r$ and the $C^0 (\overline{D_2})$ norms of $ \p_{x_1}^{|\al|} a_{12} $ and $\p_{x_1}^{|\al|} a_2$ with $|\al| = 0, \cdots, 3$.

\end{lemma}
\begin{proof}
  For the finite approximation $\Psi^{N, \si} (x_1, r)= \sum_{j =1}^N A_j^{N, \si } (x_1) b_j (r)$, one can use the same arguments as in Lemma \ref{exist1} to obtain the estimates in Lemma \ref{aH2}, \ref{aH3} and \ref{aH4}. These estimates are uniformly in $N$. Thus one can extract a weakly convergent subsequence whose weak limit coincides with the $H^2_r $ strong solution $\Psi^{\si}$ to \eqref{au2} due to the uniqueness, and the following estimate holds
  \begin{equation}\label{aH401}
  \| \Psi^{\si} \|_{H^2_r (D_2)}^2 +  \int_{\f 38 L_0}^{L_1 + 12 l}  \int_0^1 ( |\na \p_{x_1}^2 \Psi^{\si } |^2 + |\na \p_{x_1}^3 \Psi^{\si}|^2 ) r dr dx_1
  \le C_{\sharp}  \| G_0  \|_{H^3_r (D_2)}^2  \le C_{\sharp} \|F_0\|_{H^3_r (D )}^2.
  \end{equation}
  Since the estimate \eqref{aH401} is uniformly in $\si$, so there exists a subsequence $\{ \Psi^{\si_j} \}_{j =1}^{\oo}$ which weakly converges to a function $\widetilde{\Psi}$ with the estimate
  \begin{equation}\label{aH402}
  \| \widetilde{\Psi} \|_{H^2_r (D_2)}^2
  + \int_{\f 38 L_0}^{L_1 + 12 l} \int_0^1 ( |\na \p_{x_1 }^2 \widetilde{\Psi}  |^2 + |\na \p_{x_1}^3 \widetilde{\Psi} |^2 ) r dr dx_1 \le C_{\sharp} \|F_0\|_{H^3_r (D )}^2.
  \end{equation}
  Due to the uniqueness of the solution to \eqref{au1}, the function $\widetilde{\Psi}$ coincides with the solution $\Psi$ constructed in Lemma \ref{equality}. Furthermore, Lemma \ref{equality} implies that $\Psi = \psi$ in $D $ and thus
  \begin{equation*}
  \int_{\f 38 L_0}^{L_1} \int_0^1 (|\na \p_{x_1}^2 \psi|^2 + |\na \p_{x_1}^3 \psi|^2 ) r dr dx_1 \le C_{\sharp} \|F_0\|_{H^3_r (D )}^2.
  \end{equation*}
  This, together with \eqref{H31}, yields that
  \begin{equation}\label{aH404}
  \| \na \p_{x_1}^2 \psi \|^2_{L^2_r (D)} + \| \na \p_{x_1}^3 \psi \|^2_{L^2_r (D)} \le C_{\sharp} \|F_0\|_{H^3_r (D )}^2.
  \end{equation}
  Since the following equality holds almost everywhere
  \begin{equation}\label{aH45}
   \begin{cases}
   \p_{r}^2 \psi + \f 1r \p_r \psi = F_0 - k_{11} \p_{x_1 }^2 \psi - 2 k_{12} \p_{x_1 r}^2 \psi - k_1 \p_{x_1} \psi - k_2 \p_r \psi \eqq H, \\
   \p_r \psi (x_1, 1) = 0,
   \end{cases}
  \end{equation}
  then
  \begin{equation}\label{aH405}
  \iint_D ( |\p_{x_1} \p_r^2 \psi |^2 + |\p_{x_1} (\f 1r \p_r \psi)|^2 ) r dr dx_1 \le \iint_D |\p_{x_1} H|^2 r dr dx_1 
   \le C_{\sharp} \|F_0\|_{H^3_r (D )}^2,
  \end{equation}
  \begin{equation}\label{aH406}
   \iint_D (|\p_{x_1}^2 \p_r^2 \psi |^2 + |\p_{x_1}^2 (\f 1r \p_r \psi)|^2) r dr dx_1 \le \iint_D |\p_{x_1}^2 H |^2 r dr dx_1 
   \le
   C_{\sharp} \|F_0\|_{H^3_r (D )}^2.
  \end{equation}

  Note that $\check{\psi} (x_1, x') = \psi (x_1, r)$ solves
  \begin{equation}\label{aH46}
  \begin{cases}
  \sum_{i = 2}^3 \p_{x_i}^2 \check{\psi } = \check{H} (x_1, x'),  \q \forall x' \in \{ x' \in \mbR^2, |x'| < 1 \}, \\
  \sum_{i = 2}^3 x_i \p_{x_i} \check{\psi } (x_1, x') = 0, \q  \text{on } |x'| = 1,
  \end{cases}
  \end{equation}
  where $\check{H} (x_1, x') = H (x_1, r)$. There holds
  \begin{equation*}
  \iint_{|x'| < 1} |\na^2_{x'} \check{\psi}|^2 + |\na^3_{x'} \check{\psi}|^2 d x' \le c_* \iint_{|x'| < 1} |\check{H}|^2 + |\na_{x'} \check{H}|^2 dx'
  = c_* \int_0^1 ( |H|^2 + |\p_r H|^2 ) r dr .
  \end{equation*}
  Integrating the above inequality with respect to $x_1$ on $[L_0,L_1]$, one gets
  \begin{equation}\label{aH34}
   \iint_D (|\p_r^3 \psi|^2 + |\p_r (\f 1r \p_r \psi)|^2) r dr dx_1 \le c_* \iint_D (|H|^2 + |\p_r H|^2) r dr dx_1 
   \le C_{\sharp} \|F_0\|_{H^3_r (D )}^2.
  \end{equation}

  Denote $\check{w}_1 (x_1, x') = \p_{x_1} \check{\psi} (x_1, x')$, then $\check{w}_1$ solves
  \begin{equation*}
  \begin{cases}
  \sum_{i = 2}^3 \p_{x_i}^2 \check{w}_1 = \p_{x_1} \check{H} (x_1, x'),  \q \forall x' \in \{ x' \in \mbR^2, |x'| < 1 \}, \\
  \sum_{i = 2}^3 x_i \p_{x_i} \check{w}_1 (x_1, x') = 0, \q  \text{on } |x'| = 1.
  \end{cases}
  \end{equation*}
  Therefore,
  \begin{equation*}
   \iint_{|x'| < 1} (|\na_{x'}^2 \check{w}_1|^2 + |\na^3_{x'} \check{w}_1|^2) d x' \le c_* \iint_{|x'|< 1} (|\p_{x_1} \check{H}|^2 + |\na_{x'} \p_{x_1} \check{H}|^2) dx' .
  \end{equation*}
  Same as \eqref{aH34}, there holds
  \begin{equation}\label{aH407}
   \iint_D (|\p_r^2 \p_{x_1} \psi|^2 + |\p_r^3 \p_{x_1} \psi|^2 + |\p_{x_1 r}^2 (\f 1r \p_r \psi )|^2) r dr dx_1 \le  c_* \iint_D (|\p_{x_1} H|^2 + |\p_{x_1 r}^2 H|^2) r dr dx_1
  \le  C_{\sharp} \|F_0\|_{H^3_r (D )}^2.
  \end{equation}

  It follows from \eqref{aH46} that
  \be\no
  && \iint_D (|\p_r^4 \psi|^2+|\p_r(\frac1r \p_r)\psi|^2 + |\frac1r \p_r(\frac1r \p_r\psi)|^2)rdr dx_1\\\no
  && \leq\int_{L_0}^{L_1} \iint_{|x'| < 1} ( |\na_{x'}^2 \check{\psi }|^2 +  |\na_{x'}^3 \check{\psi }|^2 + |\na_{x'}^4 \check{\psi }|^2 ) d x' d x_1 \\\no
  && \le c_*  \int_{L_0}^{L_1} \iint_{|x'| < 1} ( |\check{H}|^2 + |\na_{x'} \check{H }|^2 + |\na_{x'}^2 \check{H }|^2 ) d x' d x_1 \\\no
  && \le c_* \iint_D (|H (x_1, r)|^2 + |\p_r H (x_1, r)|^2 + |\p_r^2 H (x_1, r)|^2 + |\f 1r \p_r H (x_1, r)|^2) r dr dx_1 \\\label{aH47}
  && \le C_{\sharp} (\| F \|_{H^3_r (D)}^2 + \| \mc{F} \|_{H^2_r (D)}^2).
  \ee
  The estimate \eqref{aH400} follows from \eqref{aH404}, \eqref{aH405}, \eqref{aH406}, \eqref{aH34}, \eqref{aH407} and \eqref{aH47}.
\end{proof}

Finally, we prove that the constant $C_{\sharp}$ in \eqref{aH400} can be replaced by a constant $C_*$ which depends only on the $H^3_r (D )$ norms of $k_{11}$, $k_1$ and the $H^2_r (D)$ norms of $\p_r k_{12}$, $\f {k_{12}}r$, $\p_r k_2$, $\f {k_2}r$, the $L^{\oo}_r (D)$ norms of $k_{12}$, $\p_{x_1} k_{12}$, $k_2$ and $\p_{x_1} k_2$, the $L^4_r (D)$ norm of $\p_{x_1}^2 k_{12}$ and the $L^2_r (D)$ norms of $\na \p_{x_1}^2 k_{12}$, $\p_{x_1}^2 k_2$ and $\na \p_{x_1}^2 k_2$.

\begin{lemma}\label{H4}
  There exists a constant $\de_* > 0$ depending only on the background flow, such that if $0 < \de_0 \le \de_*$ in Lemma \ref{coe-estimate}, the solution to \eqref{linearized2} satisfies the compatibility condition
  \begin{equation}\label{cp3}
  \p_r^3 \psi (x_1, 0)  = 0 \text{ in the sense of $H^1_r (D )$ trace, }
  \end{equation}
  and the estimate
  \begin{equation}\label{f3}
  \| \psi \|_{H^4_r (D )}  \le C_* ( \| F\|_{H^3_r (D )} + \|  \mc{F}  \|_{H^2_r (D)}),
  \end{equation}
  where $C_*$ depends only on the $H^3_r (D )$ norms of $k_{11}$, $k_1$ and the $H^2_r (D)$ norms of $\p_r k_{12}$, $\f {k_{12}}r$, $\p_r k_2$, $\f {k_2}r$, the $L^{\oo}_r (D)$ norms of $k_{12}$, $\p_{x_1} k_{12}$, $k_2$ and $\p_{x_1} k_2$, the $L^4_r (D)$ norm of $\p_{x_1}^2 k_{12}$ and the $L^2_r (D)$ norms of $\na \p_{x_1}^2 k_{12}$, $\p_{x_1}^2 k_2$ and $\na \p_{x_1}^2 k_2$.

\end{lemma}
\begin{proof}
  Given that $\psi \in H^4_r (D)$, then $w_1 = \p_{x_1} \psi $ satisfies the following equation almost everywhere
  \begin{equation}\label{f1}
  \begin{cases}
    \mc{L}_1 w_1 \co k_{11} \p_{x_1}^2 w_1  + 2 k_{12} \p_{x_1 r}^2 w_1  + \p_{r}^2 w_1 + \f 1r \p_r w_1 + k_3 \p_{x_1} w_1 + k_4 \p_r w_1 = F_1,  \\
    \p_r w_1 (x_1, 0) = \p_r w_1 (x_1, 1) = 0, \q \forall x_1 \in [L_0, L_1],
  \end{cases}
  \end{equation}
where
\begin{equation*}
k_3 =  \p_{x_1} k_{11}  + k_1 , \q
k_4 = 2 \p_{x_1} k_{12}  + k_2 ,\q F_1 =  \p_{x_1} F_0 -  \p_{x_1} k_1 w_1 - \p_{x_1} k_2 \p_r \psi.
\end{equation*}
Define a monotone increasing smooth cut-off function $\eta_{11}$ on $[L_0, L_1]$ such that $0 \le \eta_{11} (x_1) \le 1$ and
\begin{equation*}
\eta_{11} (x_1) = \begin{cases}
             0, & L_0 \le x_1 \le \f 34 L_0, \\
             1, & \f 12 L_0 \le x_1 \le L_1.
           \end{cases}
\end{equation*}
Then $\tilde{w}_1 = \eta_{11} w_1$ would satisfy
\begin{equation}\label{tw1}
\begin{cases}
  \mc{L}_1 \tilde{w}_1 \co k_{11} \p_{x_1}^2 \tilde{w}_1 + 2 k_{12} \p_{x_1 r}^2 \tilde{w}_1 + \p_{r}^2 \tilde{w}_1 + \f 1r \p_r \tilde{w}_1 + k_3 \p_{x_1} \tilde{w}_1 + k_4 \p_r \tilde{w}_1 = \tilde{F}_1, \\
  \p_r \tilde{w}_1 (x_1, 0) =  \p_r \tilde{w}_1 (x_1, 1) = 0, \q \forall x_1 \in [L_0, L_1], \\
  \tilde{w}_1 (L_0, r) = 0, \q \forall r \in [0, 1],
\end{cases}
\end{equation}
where
\begin{equation*}
\tilde{F}_1 = \eta_{11} \p_{x_1} F_0 - \eta_{11} \p_{x_1} k_1 w_1 - \eta_{11} \p_{x_1} k_2 \p_r \psi + k_{11} (\eta_{11}'' w_1 + 2 \eta_{11}' \p_{x_1} w_1) + 2 \eta_{11}' k_{12} \p_r w_1 + \eta_{11}' (\p_{x_1} k_{11} + k_1) w_1.
\end{equation*}

Note that if $0 < \de_0 \le \de_*$ in Lemma \ref{coe-estimate}, then we have for any $ (x_1, r) \in D$,
  \be\no
  && 2 k_3 - \p_{x_1} k_{11} = \p_{x_1} k_{11} + 2 k_1 \le 2 \bar{k}_1 + \bar{k}_{11}' + 2 \| k_1 - \bar{k}_1 \|_{L^{\oo}} + \| \p_{x_1} k_{11} - \bar{k}_{11}' \|_{L^{\oo}} \le - \ka_* < 0 ,\\\no
  && 2 k_3 + \p_{x_1} k_{11} = 3 \p_{x_1} k_{11} + 2 k_1 \le 2 \bar{k}_1 + 3 \bar{k}_{11}' + 2 \| k_1 - \bar{k}_1 \|_{L^{\oo}} + 3 \| \p_{x_1} k_{11} - \bar{k}_{11}' \|_{L^{\oo}} \le - \ka_* < 0.
  \ee
As in Lemma \ref{exist2}, there exists a unique strong solution $v_1 \in H^2_r (D)$ to \eqref{tw1} with the estimate
  \begin{equation*}
  \| v_1 \|_{H^2_r (D )} \le C_* \| \tilde{F}_1 \|_{H^1_r (D)}.
  \end{equation*}
  The uniqueness implies that $v_1 = \tilde{w}_1$ holds a.e. in $\Om$. Thus
  \begin{equation*}
   \b( \int_{\f 12 L_0}^{L_1} \int_0^1 |\na^2 w_1|^2 r dr dx_1  \b)^{\f 12} + \b( \int_{\f 12 L_0}^{L_1} \int_0^1 |\f {\p_r w_1}r|^2 r dr dx_1 \b)^{\f 12}
   \le   \| \tilde{w}_1 \|_{H^2_r (D) }
   \le  C_* \| \tilde{F}_1 \|_{H^1_r (D) },
  \end{equation*}
  and
  \be\no
   && \| \tilde{F}_1 \|_{H^1_r (D) }  \le  C_* (\| \p_{x_1} F_0 \|_{H^1_r (D)} + \| \p_{x_1} k_1 \|_{H^2_r (D)} \| \eta_{11} w_1 \|_{H^1_r (D)} + \| k_{11}\|_{H^2_r (D)} (\| \eta_{11}'' w_1 +\eta_{11}' \p_{x_1} w_1 \|_{H^1_r (D)})  \\\no
    && \q
    + \| \p_{x_1} (\f {k_2}r) \|_{H^1_r (D)} \| \eta_{11} r \p_r \psi \|_{H^2_r (D)}+ \| \f {k_{12}}r \|_{H^2_r (D)} \| \eta_{11}' r \p_r w_1 \|_{H^1_r (D)}
    + \| \p_{x_1} k_{11} +k_{1}\|_{H^2_r (D)} \| \eta_{11}' w_1 \|_{H^1_r (D)} ) \\\no
  && \le   C_* ( \| F_0 \|_{H^2_r (D)} + \| \psi \|_{H^3_r ((\f 34 L_0, \f 12 L_0) \times (0, 1))} + ( \eps + \de_0) \| \psi \|_{H^3_r (D)} )
   \le C_* \| F_0 \|_{H^2_r (D)} + C_* (\eps + \de_0) \| \psi \|_{H^3_r (D)},
  \ee
  where we use the fact that $\textrm{supp } \eta_{11}' \subseteq (\f 34 L_0, \f 12 L_0)$.
  Together with \eqref{H31}, one can infer that
  \begin{equation}\label{f15}
  \| \na^2 \p_{x_1} \psi \|_{L^2_r (D )} + \| \p_{x_1} (\f 1r \p_r \psi) \|_{L^2_r (D)} \le C_* \| F_0 \|_{H^2_r (D)} + C_* (\eps + \de_0) \| \psi \|_{H^3_r (D)}.
  \end{equation}
  Using \eqref{aH45}, we conclude that
  \be\no
  && \|\p_r^3 \psi\|_{L^2_r (D)} + \|\p_r (\f 1r \p_r \psi)\|_{L^2_r (D)} \le C_* \|\p_r H\|_{L^2_r (D)} \\\no
  &&
  \le  C_* \| \p_r F_0 \|_{L^2_r (D)} + C_* ( \| \p_r k_{11} \|_{L^{\oo}_r (D)} + \| \p_r k_{12} \|_{L^{\oo}_r (D)} + \| k_1 \|_{L^{\oo}_r (D)}   + \| k_2 \|_{L^{\oo}_r (D)}  ) \| \na^2 \psi \|_{L^2_r (D)} \\\no
  && \q
   + C_* ( \| \p_r k_2 \|_{L^{\oo}_r (D)} + \| \p_r k_1 \|_{L^{\oo}_r (D)}  ) \| \na \psi \|_{L^2_r (D)}
  + C_* ( \| k_{11} \|_{L^{\oo}_r (D)}  +  \| k_{12} \|_{L^{\oo}_r (D)} ) \| \na^2 \p_{x_1} \psi \|_{L^2_r (D)}\\\no
  && \le C_* \| F_0\|_{H^2_r (D)} + C_* (\eps + \de_0) \| \psi \|_{H^3_r (D)}.
  \ee
  This, together with \eqref{f15}, gives
  \begin{equation*}
  \| \psi \|_{H^3_r (D)}  \le C_* \| F_0 \|_{H^2_r (D)} + C_* (\eps + \de_0) \| \psi \|_{H^3_r (D)}.
  \end{equation*}
  Choosing $\eps+\delta_0$ small enough so that $C_* (\eps + \de_0)\leq \frac{1}{2}$, then
  \begin{equation}\label{f17}
  \| \psi \|_{H^3_r (D)}  \le C_* \| F_0 \|_{H^2_r (D)}.
  \end{equation}

  Set $w_2 = \p_{x_1} w_1$, it satisfies the following equation almost everywhere
  \begin{equation*}
  \begin{cases}
    \mc{L}_2 w_2 \co k_{11} \p_{x_1}^2 w_2 + 2 k_{12} \p_{x_1 r}^2 w_2 + \p_{r}^2 w_2 + \f 1r \p_r w_2 + k_5 \p_{x_1} w_2 + k_6 \p_r w_2 = F_2,  \\
    \p_r w_2 (x_1, 0) = \p_r w_2 (x_1, 1) = 0,  \q \forall x_1 \in [L_0, L_1],
  \end{cases}
  \end{equation*}
  where
  \be\no
&& k_5 = \p_{x_1} k_{11} + k_3 = 2 \p_{x_1} k_{11} + k_1 , \q
k_6 = 2 \p_{x_1} k_{12} + k_4 = 4 \p_{x_1} k_{12} + k_2  , \\\no
&& F_2 = \p_{x_1 }^2 F_0 - (\p_{x_1 }^2 k_{11} + 2 \p_{x_1} k_1 ) w_2
- (2 \p_{x_1}^2 k_{12} + 2 \p_{x_1 } k_2) \p_r w_1
- \p_{x_1 }^2 k_1 w_1 - \p_{x_1}^2 k_2 \p_r \psi.
\ee

Set $\tilde{w}_2 = \eta_{11} w_2$, then $\tilde{w}_2$ satisfies
\begin{equation}\label{f21}
\begin{cases}
  \mc{L}_2 \tilde{w}_2 = k_{11} \p_{x_1}^2 \tilde{w}_2 + 2 k_{12} \p_{x_1 r}^2 \tilde{w}_2 + \p_{r}^2 \tilde{w}_2 + \f 1r \p_r \tilde{w}_2 + k_5 \p_{x_1} \tilde{w}_2 + k_6 \p_r \tilde{w}_2 = \tilde{F}_2 ,\\
  \p_r \tilde{w}_2 (x_1, 0) = \p_r \tilde{w}_2 (x_1 , 1) = 0, \q \forall x_1  \in [L_0, L_1], \\
  \tilde{w}_2 (L_0, r) = 0, \q \forall r \in [0, 1],
\end{cases}
\end{equation}
where $\tilde{F}_2 = \eta_{11} F_2 + (k_{11}  \eta_{11}'' + k_5 \eta_{11}') w_2 + 2 k_{11} \eta_{11}' \p_{x_1} w_2 + 2 k_{12} \eta_{11}' \p_r w_2$.

Note that if $0 < \de_0 \le \de_*$ in Lemma \ref{coe-estimate}, then for any $ (x_1, r) \in D $, one gains
\be\no
&& 2 k_5 - \p_{x_1} k_{11} = 2 k_1 + 3 \p_{x_1} k_{11} \le -\ka_* < 0, \\\no
&& 2 k_5 + \p_{x_1} k_{11} = k_1 + 5 \p_{x_1} k_{11} \le - \ka_* < 0.
\ee
Then as in Lemma \ref{exist2}, $\tilde{w}_2$ is the unique $H^2_r (D) $ strong solution to \eqref{f21} with the estimate
 \be\no
\b(\int_{\f 12 L_0}^{L_1} \int_0^1 (|w_2|^2 + |\na w_2|^2 + |\na^2 w_2|^2 ) r dr dx_1 \b)^{\f 12 }
+ \b(\int_{\f 12 L_0}^{L_1} \int_0^1 |\f {\p_r w_2}r|^2 r dr dx_1 \b)^{\f 12}
 \le  \| \tilde{w}_2 \|_{H^2_r (D )} \le  C_* \| \tilde{F}_2 \|_{H^1_r (D )} ,
 \ee
and
\be\no
&& \| \tilde{F}_2 \|_{H^1_r (D )}  \le
C_* \b(
\|F_0\|_{H^3_r (D)} + \| (\p_{x_1}^2 k_{11} - \bar{k}_{11}'') + 2 (\p_{x_1} k_1 - \bar{k}_1') \|_{H^1_r (D)} \| w_2 \|_{H^2_r (D)} + \|(\bar{k}_{11}'' + 2 \bar{k}_1') w_2 \|_{H^1_r (D)} \\\no
&& \q
+ \|  \p_{x_1} ( \f {k_2}r ) \|_{H^1_r (D)} \| r \p_r w_1 \|_{H^2_r (D)}
+ ( \| \p_{x_1}^2 k_{12} \|_{L^2_r (D)} + \| \na \p_{x_1}^2 k_{12} \|_{L^2_r (D)} ) \| \p_r w_1 \|_{L^{\oo}_r (D)}
\\\no
&& \q
+ \| \p_{x_1}^2 k_{12} \|_{L^4_r (D)} \| \na \p_r w_1 \|_{L^4_r (D)}
+ \| \p_{x_1}^2 k_1 - \bar{k}_1'' \|_{H^1_r (D)} \| w_1 \|_{H^2_r (D)}
+ \| \bar{k}_1'' w_1 \|_{H^1_r (D)} + \| \f {k_{12}}r  \|_{H^2_r (D)} \| \eta_{11}' r \p_r w_2 \|_{H^1_r (D)}
 \\\no
&& \q
+ \|  \p_{x_1}^2 k_2 \|_{L^2_r (D)} (\| \p_r \psi \|_{L^{\oo}_r (D)} + \| \na \p_r \psi \|_{L^{\oo}_r (D)}) + \| \na \p_{x_1}^2 k_2 \|_{L^2_r (D)} \| \p_r \psi \|_{L^{\oo}_r (D)}
+ ( \|  k_1 - \bar{k}_1 \|_{H^2_r (D)}  \\\no
&& \q
+  \| \p_{x_1} k_{11} - \bar{k}_{11}'\|_{H^2_r (D)} )  \| \eta_{11}' w_2 \|_{H^1_r (D)}
+ \| ( \bar{k}_1 +  \bar{k}_{11}' )\eta_{11}' w_2 \|_{H^1_r (D)} + \| k_{11} - \bar{k}_{11} \|_{H^2_r (D)} \| \eta_{11}'' w_2 \|_{H^1_r (D)}
\\\no
&& \q + \| \bar{k}_{11}  \eta_{11}'' w_2 \|_{H^1_r (D)} +  \| k_{11} - \bar{k}_{11} \|_{H^2_r (D)} \| \eta_{11}' \p_{x_1} w_2 \|_{H^1_r (D)} + \| \bar{k}_{11} \eta_{11}' \p_{x_1} w_2 \|_{H^1_r (D)}\b)\\\no
&&
 \le C_* (\|F_0\|_{H^3_r (D)} + \| \psi \|_{H^3_r (D)} + (\eps + \de_0) \| \psi \|_{H^4_r (D)}).
\ee
Combining with \eqref{H31} and  \eqref{f17} deduce that
\begin{equation}\label{f22}
\| \psi \|_{H^3_r (D)} + \| \na^2 \p_{x_1}^2 \psi \|_{L^2_r (D)}  + \| \p_{x_1}^2 (\f 1r \p_r \psi) \|_{L^2_r (D)}
\le C_* (\|F_0\|_{H^3_r (D)} + (\eps + \de_0) \| \psi \|_{H^4_r (D)}).
\end{equation}
Again, it follows from \eqref{aH45} that
\be\label{f23}
 && \| \p_{x_1} \p_r^3 \psi \|_{L^2_r (D)} + \| \p_{x_1 r}^2 (\f 1r \p_r \psi) \|_{L^2_r (D)} \le C_* \| \p_{x_1 r}^2 H \|_{L^2_r (D)} \\\no
 && \le C_* \{ \| \p_{x_1 r}^2 F_0 \|_{L^2_r (D)} + ( \| k_{11} \|_{L^{\oo}_r (D)} + \| k_{12} \|_{L^{\oo}_r (D)}) \| \na^2 \p_{x_1}^2 \psi \|_{L^2_r (D)} + (\| \p_{x_1} k_{11} \|_{L^{\oo}_r (D)} + \| \p_r k_{11} \|_{L^{\oo}_r (D)}
  \\\no
 && \q
 + \| \p_{x_1} k_{12} \|_{L^{\oo}_r (D)} + \| \p_r k_{12} \|_{L^{\oo}_r (D)} + \| k_1 \|_{L^{\oo}_r (D)} + \| k_2 \|_{L^{\oo}_r (D)} ) \| \na^3 \psi \|_{L^2_r (D)}  + \| \p_{x_1 r}^2 k_2 \|_{L^4_r (D)} \| \p_r \psi \|_{L^4_r (D)}
 \\\no
  && \q
+  (\| \p_{x_1} k_{11} \|_{L^{\oo}_r (D)} + \| \p_r k_{11} \|_{L^{\oo}_r (D)} + \| \p_{x_1} k_2 \|_{L^{\oo}_r (D)} + \| \p_r k_2 \|_{L^{\oo}_r (D)}) \| \na^2 \psi \|_{L^2_r (D)} \\\no
 && \q
 + \| \p_{x_1 r}^2 k_{11} \|_{L^4_r (D)} ( \| \p_{x_1}^2 \psi \|_{L^4_r (D)} + \| \p_{x_1} \psi \|_{L^4_r (D)} ) + \| \p_{x_1 r}^2 k_{12} \|_{L^4_r (D)} \| \p_{x_1 r}^2 \psi \|_{L^4_r (D)} \} \\\no
 &&  \le C_* \| F_0 \|_{H^2_r (D)} + C_* (\eps + \de_0) ( \|F_0\|_{H^3_r (D)}  + (\eps + \de_0) \| \psi \|_{H^4_r (D)}) + C_* (\eps + \de_0)^2 \| \psi \|_{H^2_r (D)}  \\\label{f23}
 &&
\le C_* ( \|F_0\|_{H^3_r (D)}  + (\eps + \de_0)^2 \| \psi \|_{H^4_r (D)}),
\ee
and
\be\no
&& \| \p_r^4 \psi \|_{L^2_r (D)} + \| \p_r^2 (\f 1r \p_r \psi ) \|_{L^2_r (D)} + \| \f 1r \p_r (\f 1r \p_r \psi) \|_{L^2_r (D)} \le C_* (\| \p_r^2 H \|_{L^2_r (D)} + \| \f 1r \p_r H \|_{L^2_r (D)})  \\\no
 && \le C_* \{ \| \p_r^2 F_0 \|_{L^2_r (D)} + \| \f 1r \p_r F_0 \|_{L^2_r (D)} +  \| k_{11} \|_{L^{\oo}_r (D)} \| \p_{x_1}^2 \p_r^2 \psi \|_{L^2_r (D)}  +  \| k_{12} \|_{L^{\oo}_r (D)} \| \p_{x_1} \p_r^3 \psi \|_{L^2_r (D)} \\\no
   && \q
  +  (\| \p_r k_{12} \|_{L^{\oo}_r (D)} + \| k_1 \|_{L^{\oo}_r (D)} + \| k_2 \|_{L^{\oo}_r (D)} + \| \f 1r k_{12} \|_{L^{\oo}_r (D)} +  \| \p_r k_{11} \|_{L^{\oo}_r (D)}) \| \na^3 \psi \|_{L^2_r (D)}  \\\no
   && \q
  +  (\| \p_r k_1 \|_{L^{\oo}_r (D)} + \| \p_r k_2 \|_{L^{\oo}_r (D)} + \| \f 1r k_2 \|_{L^{\oo}_r (D)}) \| \na^2 \psi \|_{L^2_r (D)}  +  \| \p_{x_1}^2 \psi \|_{L^4_r (D)} ( \| \p_r^2 k_{11} \|_{L^4_r (D)}
  \\\no
  && \q
   + \| \f 1r \p_r k_{11} \|_{L^4_r (D)} )   +  (\| \p_r^2 k_1 \|_{L^4_r (D)} + \| \f 1r \p_r k_1 \|_{L^4_r (D)}) \| \p_{x_1} \psi \|_{L^4_r (D)}  +  \| \p_r^2 k_{12} \|_{L^4_r (D)} \| \p_{x_1 r}^2 \psi \|_{L^4_r (D)}   \\\no
  && \q
  +  \| \p_r^2 k_2 \|_{L^4_r (D)} \| \p_r \psi \|_{L^4_r (D)}  + \| k_{11} \|_{L^{\oo}_r (D)} \| \p_{x_1}^2 (\f 1r \p_r \psi) \|_{L^2_r (D)}  +  \| \p_r k_2 \|_{L^{\oo}_r (D)} \| \f 1r \p_r \psi \|_{L^2_r (D)}  \\\no
  && \q
 +  (\| \p_r k_{12} \|_{L^{\oo}_r (D)} + \| k_1 \|_{L^{\oo}_r (D)}) \| \p_{x_1} (\f 1r \p_r \psi) \|_{L^2_r (D)} \}  \\\no
  && \le C_* \{ \| F_0 \|_{H^2_r (D)}  +  (\eps + \de_0) (\| \psi \|_{H^3_r (D)} + \| \na^2 \p_{x_1}^2 \psi \|_{L^2_r (D)} + \| \p_{x_1} \p_r^3 \psi \|_{L^2_r (D)} + \| \p_{x_1}^2 (\f 1r \p_r \psi) \|_{L^2_r (D)})\}
\\\label{f24}
  &&
\le C_* (\|F_0\|_{H^3_r (D)}  + (\eps + \de_0)^2 \| \psi \|_{H^4_r (D)}).
\ee
Combining with \eqref{f22} and \eqref{f23} deduce that
\begin{equation*}
\| \psi \|_{H^4_r (D)} \le C_* (\| F \|_{H^3_r (D)} + \| \mc{F} \|_{H^2_r (D)} + (\eps + \de_0) \| \psi \|_{H^4_r (D)} ).
\end{equation*}
 Let $0 < \de_0 \le \de_*$ so that $C_* ( \eps + \de_*) \le \f 12$, then \eqref{f3} follows.

It remains to prove the compatibility condition \eqref{cp3}. As a result of \eqref{c1}, it suffices to show that \eqref{cp3} holds on $(\f 18 L_0, L_1)$. Suppose $k_{11}, k_1 \in C^4 (\overline{D})$ and $\p_r k_{12}, \f {k_{12}}r, \p_r k_2, \f {k_2}r \in C^3 (\overline{D})$, then $a_{11}, a_1 \in C^4 (\overline{D_2})$ and $\p_r a_{12}, \f {a_{12}}r, \p_r a_2, \f {a_2}r \in C^3 (\overline{D_2})$. One may obtain the $L^2_r$ estimate of $\na \p_{x_1}^4 \Psi^{\si}$ on the domain $D_3 \co \{ (x_1, r) \in (\f 18 L_0, L_1 + 10 l) \times (0, 1) \} $ as demonstrated in Lemma \ref{aH4}. With that, we can derive the estimate of $\| \Psi \|_{H^5_r (D_3)}$ which implies that $\Psi \in C^{3, \f 12} (\overline{D_3})$ and $\psi \in C^{3, \f 12 } ([\f 18 L_0, L_1] \times [0, 1])$. Then $\p_r^3 \psi (x_1, 0) = 0$ for $\forall x_1 \in (\f 18 L_0, L_1)$ will follow by differentiating the equation \eqref{aH45} with respect to $r$ and evaluating at $r = 0$. The general case will follow by a density argument.
\end{proof}


We now prove Theorem \ref{irro}. For any $\hat{\psi} \in \Sigma_{\de_0}$, then Lemma \ref{coe-estimate} holds. By Lemmas \ref{exist2}, \ref{H4}, there exists a unique solution $\psi \in H^4_r (D)$ to \eqref{linearized2} with
\begin{equation*}
\| \psi \|_{H^4_r (D )} \le C_*\|F_0(\na \hat{\psi})\|_{H^3_r (D)} \le m_* \|F_0(\na \hat{\psi})\|_{H^3_r (D)}.
\end{equation*}
Here $C_*$ depends only on the $H^3_r (D )$ norms of $k_{11}$, $k_1$ and the $H^2_r (D)$ norms of $\p_r k_{12}$, $\f {k_{12}}r$, $k_2$, $\f {k_2}r$ , the $L^{\oo}_r (D)$ norms of $k_{12}$, $\p_{x_1} k_{12}$, $k_2$ and $\p_{x_1} k_2$, the $L^4_r (D)$ norm of $\p_{x_1}^2 k_{12}$ and the $L^2_r (D)$ norms of $\na \p_{x_1}^2 k_{12}$, $\p_{x_1}^2 k_2$ and $\na \p_{x_1}^2 k_2$, which can be bounded by a constant $m_*$ depends on the $C^3 ([L_0, L_1])$ norm of $\bar{k}_{11}$, $\bar{k}_1$ and the boundary data. In the following, the constant $m_*$ will always denote a constant depending only on the background solutions and the boundary data.

Recall the definition of $F_0 (\na \hat{\psi})$ in \eqref{f0}. Since the support of $\eta_0 (x_1)$ defined in \eqref{eta} is contained in $[L_0, \f 78 L_0]$, according to the $H^4_r $ estimate \eqref{H31} in Lemma \ref{H3} and the estimates obtained in Lemmas \ref{aH3}-\ref{H4}, there holds a better estimate
\be\no
&& \|  \psi \|_{H^4_r (D )}  \le  m_* ( \| F (\na \hat{\psi} + \eps \na \psi_0) \|_{H^3_r (D)} + \| \mc{F} (\na \hat{\psi}) \|_{H^2_r (D)} ) \\\no
&& \le
m_* ( \eps + \| \hat{\psi} \|_{H^4_r (D)}^2 + \eps (\| \f {h_1}r \|_{H^2_r ((0, 1)) } + \| h_1' \|_{H^2_r ((0, 1))})) \le m_* (\eps + \de_0^2).
\ee
Here only the norms $\| \f {h_1}r \|_{H^2_r ((0, 1)) }$ and $\| h_1' \|_{H^2_r ((0, 1))}$ are needed, this is the reason why introducing the cut-off function $\eta_0$ in \eqref{f0}.

Let $\de_0 = \sqrt{\epsilon}$, then if $0 < \epsilon \le \eps_0 = \min \{ \f 1{4m_*^2}, \de_*^2 \}$, we have
\begin{equation*}
\| \psi \|_{H^4_r (D )} \le m_* (\eps + \de_0^2) = 2 m_* \eps \le \de_0.
\end{equation*}
By \eqref{cp3}, $\psi \in \Sigma_{\de_0}$. Hence one can define the operator $\mc{T} \hat{\psi } = \psi $, which maps $\Sigma_{\de_0}$ to itself. Then it is easy to show that the mapping $\mc{T}$ is contractive in $H^1_r-$norm for a sufficiently small $\eps_0$, and there exists a unique $\psi \in \Sigma_{\de_0}$ such that $\mc{T} \psi = \psi$ which is the desired solution. The existence of the axisymmetric sonic front can be proved as in \cite{WX23}, so is omitted. The proof of Theorem \ref{irro} is completed. 

\section{Appendix}\label{appendix}\noindent

In the appendix, we will complete the proof of Lemma \ref{coe-estimate}. First, we give two lemmas that will be frequently used in the proof of Lemma \ref{coe-estimate}. In the following lemma, we prove that the $L^{\oo}_r $ estimates of some functions can be controlled by the $L^2_r$ norms of its derivatives up to second orders.

\begin{lemma}\label{g2infinity}
Suppose that $g$ satisfies $g$, $\na g$, $ \na^2 g \in L^2_r (D) $ with $g (x_1, 0) = 0$ for any $ x_1 \in [L_0, L_1]$, then $g \in L^{\oo}_r (D) $ with the following estimate
\begin{equation}\label{g2}
\| g  \|_{L^{\oo}_r (D)}^2 \le c_* (\| g \|_{L^2_r (D)}^2 + \| \na g \|_{L^2_r (D)}^2 + \| \na^2 g \|_{L^2_r (D)}^2 ).
\end{equation}
\end{lemma}

\begin{proof}
  {\bf Step 1.} We first prove that
  \begin{equation}\label{claim}
  \iint_D (|g |^2 + |\na g |^2) dr d x_1 \leq 4 \iint_D (|g |^2 + |\na g|^2 ) r dr dx_1 + \iint_D (|\p_r g |^2 + | \na \p_r g |^2 ) r dr dx_1.
  \end{equation}

  Indeed, it follows from the mean value theorems for definite integrals, there exists $\xi \in [\f 12 , 1]$, such that $\f 12 \xi  | g (x_1, \xi)|^2 = \int_{\f 12}^1 | g (x_1, r)|^2 r  dr $. Then
  \be\no
  \int_0^{\f 12 } | g (x_1, r)|^2 dr & \le & \int_0^{\xi } | g (x_1, r)|^2 dr = \xi g^2 (x_1, \xi) - \int_0^{\xi} 2 g(x_1, r) \p_r g(x_1, r) r dr \\\no
  & \le & 2 \int_{\f 12}^1 | g (x_1, r)|^2 r dr + \int_0^1 ( ( g (x_1, r))^2 + (\p_r g (x_1, r))^2 ) r dr,
  \ee
  and
  \begin{equation*}
  \int_0^1 | g (x_1, r)|^2 dr  = \int_0^{\f 12 } | g (x_1, r)|^2 dr + \int_{\f 12 }^1 | g (x_1, r)|^2 dr
  \le 4 \int_0^1 | g (x_1, r)|^2 r dr + \int_0^1 | \p_r g (x_1, r)|^2 r dr .
  \end{equation*}
  Integrating the above inequality over $[L_0, L_1]$ with respect to $x_1$ gives
  \begin{equation*}
   \iint_D | g (x_1, r) |^2 dr dx_1 \le 4 \iint_D |  g (x_1, r)|^2 r dr dx_1 + \iint_D | \p_r g (x_1, r)|^2 r dr dx_1.
  \end{equation*}

  Similarly, one can obtain
  \begin{equation*}
  \iint_D |\na g|^2 dr dx_1 \le 5 \iint_D |\na g (x_1, r)|^2 r dr dx_1 + \iint_D | \na \p_r g (x_1, r)|^2 r dr dx_1.
  \end{equation*}
  The equation \eqref{claim} has been proven.

  {\bf Step 2.} Since for any $b \in (0, 1)$, there exists $\xi \in [\f 12 , 1]$, such that $\f 12 \xi^{1 - b} |g (x_1, \xi)|^2 = \int_{\f 12}^1 |g (x_1, r)|^2 r^{1 - b} dr $. Therefore,
  \be\no
  \int_0^{\f 12} |g (x_1, r)|^2 r^{-b} dr & \le & \int_0^{\xi} |g (x_1, r)|^2 r^{-b} dr
  = \f 1{1 - b} \xi^{1 - b} |g (x_1, \xi)|^2 - \f 2{1 - b} \int_0^{\xi} g (x_1, r) \p_r g(x_1, r) r^{1 - b} dr \\\no
  & \le &
  \f 2{1 - b} \int_{\f 12}^1 |g (x_1, r)|^2 r^{1 - b} dr + \f 1{1 - b} \int_0^1 ( |g (x_1, r)|^2 + |\p_r g (x_1, r)|^2 ) dr ,
  \ee
  and
  \be\no
  \int_0^1 |g (x_1, r)|^2 r^{-b} dr & = & \int_0^{\f 12} |g (x_1, r)|^2 r^{-b} dr + \int_{\f 12 }^1 |g (x_1, r)|^2 r^{-b} dr \\\no
  & \le & ( \f 3 {1 - b} + 2^b ) \int_0^1 |g (x_1, r)|^2 dr + \f 1 {1 - b} \int_0^1 |\p_r g (x_1, r)|^2 dr .
  \ee
  Choosing $b = \f 14$ and $b = \f 12$ in the above inequality respectively, one gets
  \be \label{g14}
  && \int_0^1 |g (x_1, r)|^2 r^{-\f 14} dr \leq 6 \int_0^1 |g (x_1, r)|^2 dr + 2 \int_0^1 |\p_r g (x_1, r)|^2 dr,\\\label{g12}
  && \int_0^1 |g (x_1, r)|^2 r^{-\f 12} dr \leq 8 \int_0^1 |g (x_1, r)|^2 dr + 2 \int_0^1 |\p_r g (x_1, r)|^2 dr.
  \ee

  {\bf Step 3.} Since $g (x_1, 0 ) = 0$ for all $x_1\in [L_0,L_1]$, then
  \begin{equation}\label{g21}
  g^2 (x_1, r) = 2 \int_0^r g (x_1, t) \p_t g (x_1, t) dt
  \le \int_0^1 |g (x_1, r)|^2 r^{- \f 14 } dr + \int_0^1 |\p_r g (x_1, r)|^2 r^{\f 14} dr .
  \end{equation}
  Define $f_1 (x_1) = \int_0^1 |g (x_1, r)|^2 r^{- \f 14 } dr$ and $ f_2 (x_1) = \int_0^1 |\p_r g (x_1, r)|^2 r^{\f 14} dr $. Then
  \begin{equation*}
  g^2 (x_1, r)  \le  f_1 (x_1) + f_2 (x_1), \ \ \forall r\in [0,1], x_1\in [L_0,L_1],
  \end{equation*}
  and
  \be\no
  &&|f_1' (x_1) |= \b| 2 \int_0^1 g (x_1, r) \p_{x_1} g (x_1, r) r^{- \f 14} dr \b|
  \le \int_0^1 |g (x_1, r)|^2 r^{- \f 12} dr + \int_0^1 |\p_{x_1} g (x_1, r)|^2  dr , \\\no
  &&|f_2' (x_1)| = \b| 2 \int_0^1 \p_r g (x_1, r) \p_{x_1 r}^2 g (x_1, r) r^{\f 14} dr \b|
  \le \int_0^1 |\p_r g (x_1, r)|^2 r^{- \f 12} dr + \int_0^1 |\p_{x_1 r}^2 g (x_1, r)|^2 r dr.
  \ee
  It follows from \eqref{claim}, \eqref{g12} and \eqref{g14} that
  \be\no
  && \int_{L_0}^{L_1} |f_1 (x_1)| dx_1 = \int_{L_0}^{L_1} \int_0^1 |g (x_1, r)|^2 r^{- \f 14 } dr d x_1 \\\no
  && \le 8 \int_{L_0}^{L_1} \int_0^1 |g (x_1, r)|^2 dr dx_1 + 2 \int_{L_0}^{L_1} \int_0^1 |\p_r g (x_1, r)|^2 dr dx_1 \\\no
  && \le C_* \int_{L_0}^{L_1} \int_0^1 (|g |^2 + |\na g |^2 + |\na^2 g |^2 ) r dr dx_1,
  \ee
  and
  \be\no
  && \int_{L_0}^{L_1} |f_1' (x_1)| dx_1
  \le \int_{L_0}^{L_1} \int_0^1 |g (x_1, r)|^2 r^{- \f 12} dr dx_1 + \int_{L_0}^{L_1} \int_0^1 |\p_{x_1} g (x_1, r)|^2  dr dx_1
  \\\no
  && \le   8  \int_{L_0}^{L_1} \int_0^1 |g (x_1, r)|^2 dr dx_1 +  \int_{L_0}^{L_1} \int_0^1 ( 2 |\p_r g (x_1, r)|^2 + |\p_{x_1} g (x_1, r)|^2 ) dr dx_1\\\no
  && \le   C_* \int_{L_0}^{L_1} \int_0^1 (|g |^2 + |\na g |^2 + |\na^2 g |^2 ) r dr dx_1.
  \ee
  One can conclude that $\| f_1 \|_{W^{1, 1} ([L_0, L_1]) }
  \le C_* (\| g \|_{L^2_r (D)}^2 + \| \na g \|_{L^2_r (D)}^2 + \| \na^2 g \|_{L^2_r (D)}^2 )$.

  It remains to prove the estimate of $\| f_2 \|_{W^{1, 1} ([L_0, L_1]) }$. Firstly, we derive that
  \be\no
  && \int_{L_0}^{L_1} |f_2 (x_1)| dx_1 =  \int_{L_0}^{L_1} \int_0^1 |\p_r g (x_1, r)|^2 r^{\f 14} dr dx_1 \\\no
  && \le  \int_{L_0}^{L_1} \int_0^1 |\p_r g (x_1, r)|^2  dr dx_1
   \le C_* \int_{L_0}^{L_1} \int_0^1 (|g |^2 + |\na g |^2 + |\na^2 g |^2 ) r dr dx_1,
  \\\no
  && \int_{L_0}^{L_1} |f_2' (x_1)| d x_1 \le \int_{L_0}^{L_1}  \int_0^1 |\p_r g (x_1, r)|^2 r^{- \f 12} dr d x_1 + \int_{L_0}^{L_1}  \int_0^1 |\p_{x_1 r}^2 g (x_1, r)|^2 r dr d x_1.
  \ee

  There exists $\eta \in [\f 12, 1]$ such that $\f 12 \eta^{\f 12} |\p_r g (x_1, \eta)|^2 = \int_{\f 12}^1 |\p_r g (x_1, r)|^2 r^{\f 12} dr $, and
  \be\no
  &&\int_0^{\f 12 } |\p_r g (x_1, r)|^2 r^{-\f 12} dr\leq \int_0^{ \eta } |\p_r g (x_1, r)|^2 r^{-\f 12} dr
  = 2 \eta^{\f 12} |\p_r g (x_1, \eta)|^2 - 4 \int_0^{\eta} \p_r g (x_1, r) \p_{r}^2 g (x_1, r) r^{\f 12} dr \\\no
  &&\leq 4 \int_{\f 12}^1 |\p_r g (x_1, r)|^2 dr + 2 \int_0^1 |\p_r g (x_1, r)|^2 + |\p_{r}^2 g (x_1, r)|^2 dr ,
  \\\no
 && \int_0^1 |\p_r g (x_1, r)|^2 r^{-\f 12} dr  = \int_0^{\f 12 } |\p_r g (x_1, r)|^2 r^{-\f 12} dr + \int_{\f 12}^1 |\p_r g (x_1, r)|^2 r^{-\f 12} dr \\\no
  & & \le  8 \int_0^1 |\p_r g (x_1, r)|^2 dr + 2 \int_0^1  |\p_{r}^2 g (x_1, r)|^2 dr .
  \ee
  Therefore,
  \begin{equation*}
   \int_{L_0}^{L_1} |f_2' (x_1)| d x_1
    \le C_* \int_{L_0}^{L_1} \int_0^1 (|g |^2 + |\na g |^2 + |\na^2 g |^2 ) r dr dx_1.
  \end{equation*}

  Then, it follows from \eqref{g21} and the estimate of $f_i$, where $i = 1, 2$, that
  \be\no
  && \| g (x_1, r) \|_{L^{\oo} ([L_0, L_1] \times [0,1])}^2 \le \sum_{i = 1}^2 \| f_i \|_{L^{\oo} ([L_0, L_1])} \le \sum_{i = 1}^2 \| f_i \|_{W^{1, 1} ([L_0, L_1]) } \\\no
  && \le C_* (\| g \|_{L^2_r (D)}^2 + \| \na g \|_{L^2_r (D)}^2 + \| \na^2 g \|_{L^2_r (D)}^2 ).
  \ee
  The inequality \eqref{g2} has been obtained.
\end{proof}

The following lemma shows that the spaces $H^2_r (D)$ and $H^3_r (D)$ are Banach algebra.
 \begin{lemma}\label{alge-est}
 Let $f$ and $ g $ be functions belonging to the space $ H^3_r (D)$, we can derive the following inequalities,
 \be\label{fgh1}
 && \| fg \|_{H^1_r (D)} \le \| f \|_{H^1_r (D)} \| g \|_{H^2_r (D)},\\\label{fgh2}
 && \| fg \|_{H^2_r (D)} \le \| f \|_{H^2_r (D)} \| g \|_{H^2_r (D)},
 \\\label{fgh3}
 &&\| fg \|_{H^3_r (D)} \le \| f \|_{H^3_r (D)} \| g \|_{H^3_r (D)}.
 \ee

 Furthermore, if there exists a positive constant $c_*$ such that $ f \ge c_* $ and $f \in H^m_r (D) (m=2,3)$, then $\f 1f \in H^m_r (D) (m=2,3)$ with the estimate
 \begin{equation}\label{frach2}
    \| \f 1 f\|_{H^m_r (D)} \leq C_m \| f \|_{H^m_r (D)}, \ \ m=2,3.
   \end{equation}
 \end{lemma}
 \begin{proof}

Define $\check{f} (x_1, x') = f (x_1, |x'|)$ for any $(x_1, x') \in \Om $. After some tedious calculations, one can derive
\be\no
&& \| \check{f} \|_{H^1 (\Om )}^2 = 2 \pi \| f \|_{H^1_r (D)}^2, \\\no
&& \| \check{f} \|_{H^2 (\Om )}^2 = 2 \pi \b( \| f \|_{H^1_r (D )}^2 + \| \na^2 f \|_{L^2_r (D )}^2 + \b\| \f 1r \p_r f \b\|_{L^2_r (D)}^2 \b), \\\no
&& \| \check{f} \|_{H^3 (\Om )}^2 = 2 \pi \b(
\| f \|_{H^2_r (D)}^2 + \| \na^3 f \|_{L^2_r (D)}^2 + 2 \b\| \na \left(\f 1r \p_r f\right) \b\|_{L^2_r (D)}^2
\b), \\\no
&& \| \check{f} \|_{H^4 (\Om )}^2 = 2 \pi \b(
\| f \|_{H^3_r (D)}^2 + \| \na^4 f \|_{L^2_r (D)}^2 + \b\| \na^2 \left(\f 1r \p_r f\right) \b\|_{L^2_r (D)}^2 + 2 \b\| \p_{x_1 r}^2 \left(\f 1r \p_r f\right) \b\|_{L^2_r (D)}^2 \\\no
&& \q \q
+ 5 \b\| \p_r^2 \left(\f 1r \p_r f\right) \b\|_{L^2_r (D)}^2 + 9 \b\| \f 1r \p_r \left(\f 1r \p_r f\right) \b\|_{L^2_r (D)}^2
\b).
\ee
These equalities indicate that there exist positive constants $C_1$ and $C_2$ independent of $f$, such that
\begin{equation}\label{equiv}
C_1 \| f \|_{H^k_r (D)} \le \| \check{f} \|_{H^k (\Om )} \le C_2 \| f \|_{H^k_r (D)}, \q k = 0, \cdots, 4.
\end{equation}
Since $H^1(\Omega)\subset L^4(\Omega)$, $H^2(\Omega)\subset L^{\oo}(\Omega)$, one has the embedding $H^1_r (D) \subset L^4_r (D)$, $H^2_r (D) \subset L^{\oo} (D)$ and
 \be\no
 \| fg \|_{H^1_r (D)}^2 & = & \| fg \|_{L^2_r (D)}^2 + \| \na (fg) \|_{L^2_r (D)}^2
 \le \| f \|_{L^4_r (D)}^2 \| g \|_{L^4_r (D)}^2 + \| \na f \|_{L^2_r (D)}^2 \| g \|_{L^{\oo} (D)}^2 + \| f \|_{L^4_r (D)}^2 \| \na g \|_{L^4_r (D)}^2 \\\no
 & \le & C \| f \|_{H^1_r (D)}^2 \| g \|_{H^2_r (D)}^2 .
 \ee

Further, since $H^k (\Om )$ for $k = 2, 3$ are  Banach algebras, meaning that for any $\check{f}, \check{g} \in H^k (\Om )$, ($k = 2, 3$), one has
\begin{equation*}
\| \check{f} \check{g} \|_{H^k (\Om )} \le C_* \| \check{f} \|_{H^k (\Om )} \| \check{g} \|_{H^k (\Om )}, \q k = 2, 3,
\end{equation*}
combining with \eqref{equiv}, we obtain \eqref{fgh2} and \eqref{fgh3} directly. The inequality \eqref{frach2} follows similarly.

\end{proof}

\begin{proof}[Proof of Lemma \ref{coe-estimate} ]
  To simplify the notation, we denote $\hat{\psi }$ as $\psi $.

  {\bf Step 1.} We first prove the following claim.

  {\bf Claim.} Assume that $\psi \in H^4_r (D)$, then there exist constants $C_*, c_*$ depending only the background transonic flows but independent of $\psi$ such that
  \begin{enumerate}[(1)]
    \item $\| \na \psi + \eps \na \psi_0 \|_{L^{\oo}_r (D)} + \| \na^2 \psi + \eps \na^2 \psi_0  \|_{L^{\oo}_r (D)} + \| \f 1r (\p_r \psi + \eps \p_r \psi_0) \|_{H^2_r (D)} \le C_* ( \eps + \| \psi \|_{H^4_r (D)})$;
    \item $\| (\p_r \psi + \eps \p_r \psi_0)^2 \|_{H^3_r (D)} + \| (\p_{x_1} \psi + \eps \p_{x_1} \psi_0)^2 \|_{H^3_r (D)} \le C_* ( \eps + \| \psi \|_{H^4_r (D)})^2$;
    \item $  \| \p_r^2 \psi + \eps \p_r^2 \psi_0 \|_{H^2_r (D)} + \| \p_{x_1}\{(\p_r \psi + \eps \p_r \psi_0)^2\} \|_{H^2_r (D)} \le C_* ( \eps + \| \psi \|_{H^4_r (D)})$;
    \item $c^2 (\rho) - (\p_r \psi + \eps \p_r \psi_0)^2 \ge c_* > 0$, $\| c^2 (\rho) - (\p_r \psi + \eps \p_r \psi_0)^2 \|_{H^3_r (D)} \le C_* $ and
    \begin{equation*}
    \| c^2(\rho) - c^2 (\bar{\rho}) - (\p_r \psi + \eps \p_r \psi_0 )^2 \|_{H^3_r (D)} \le C_* (\eps + \| \psi \|_{H^4_r (D)});
    \end{equation*}
    \item $ \| (\p_r \psi + \eps \p_r \psi_0) \p_r (c^2 (\rho) - (\p_r \psi + \eps \p_r \psi_0)^2) \|_{H^2_r (D)} \le C_* (\eps + \| \psi \|_{H^4_r (D)})^2$.
\end{enumerate}

  Proof of the {\bf Claim.} We separately verify the above estimates. Note that $\f {h_1 (r)}r \in H^2_r ([0,1])$, $h_1' \in H^2_r ([0, 1])$ and $\eta_0 \in C^{\oo} ([L_0, L_1])$, then $\psi_0 (x_1, r) \in H^4_r (D)$ and $\f 1r \p_r (r \p_r \psi_0) \in H^2_r (D)$. Since $\psi \in H^4_r (D)$, then $\psi_1:= \psi + \eps \psi_0 \in H^4_r (D)$.
  \begin{enumerate}[(1)]
    \item It follows from Lemma \ref{g2infinity} that
    \be\no
     \| \na \psi_1 \|_{L^{\oo}_r (D)} + \| \na^2 \psi_1 \|_{L^{\oo}_r (D)} \le C_* ( \eps + \| \na \psi  \|_{L^2_r (D)} + \| \na^2 \psi \|_{L^2_r (D)}\\\no
     + \| \na^3 \psi \|_{L^2_r (D)} + \| \na^4 \psi \|_{L^2_r (D)}) \le C_* ( \eps + \| \psi \|_{H^4_r (D)}).
    \ee
    The inequality $ \| \f 1r \p_r \psi_1 \|_{H^2_r (D)} \le C_* ( \eps + \| \psi \|_{H^4_r (D)}) $ follows from the definition of $H^4_r(D)$.

    \item Since $\psi_1 \in H^4_r (D)$, it is easy to verify that $\|(\p_{x_1} \psi_1 )^2 \|_{H^3_r (D)}\leq \| \p_{x_1} \psi_1 \|_{H^3_r (D)}^2 \leq ( \eps + \| \psi \|_{H^4_r (D)})^2$. In addition,
        \be\no
    && \| (\p_r \psi_1 )^2 \|_{H^3_r (D)} \le \| (\p_r \psi_1 )^2 \|_{L^2_r (D)} + \| \na ((\p_r \psi_1 )^2) \|_{L^2_r (D)} + \| \na^2 ((\p_r \psi_1 )^2) \|_{L^2_r (D)} \\\no
    && \q
    + \| \na^3((\p_r \psi_1 )^2) \|_{L^2_r (D)}
    + \| \f 1r \p_r ((\p_r \psi_1 )^2) \|_{L^2_r (D)} + \| \na (\f 1r \p_r ((\p_r \psi_1 )^2)) \|_{L^2_r (D)} \\\no
    && \le C_* \| \p_r \psi_1 \|_{L^{\oo}_r (D)} (
    \| \p_r \psi_1 \|_{L^2_r (D)} + \| \na \p_r \psi_1 \|_{L^2_r (D)}  +  \| \na^2 \p_r \psi_1 \|_{L^2_r (D)} ) + C_* \| \na \p_r \psi_1  \|_{L^4_r (D)}
     \\\no
    && \q \times (\|\nabla \p_r \psi_1 \|_{L_r^4(D)} + \|\nabla^2 \p_r \psi_1 \|_{L^4_r(D)} )
    + C_* \| \f 1r \p_r \psi_1  \|_{L^{\oo}_r (D)} ( \| \p_r^2 \psi_1 \|_{L^2_r (D)} +  \| \na \p_r^2 \psi_1  \|_{L^2_r (D)})
    \\\no
    &&
    \q
    + C_* \| \na (\f 1r \p_r \psi_1 ) \|_{L^4_r (D)} \| \p_r^2 \psi_1 \|_{L^4_r (D)} \le C_* ( \eps + \| \psi \|_{H^4_r (D)})^2.
    \ee

    \item Note that $ \| \f 1r \p_r(\p_r^2 \psi)\|_{L^2_r (D)} = \| \p_r^2 (\f 1r \p_r \psi ) + \f 2r \p_r (\f 1r \p_r \psi ) \|_{ L^2_r (D)} \le 3\| \psi \|_{H^4_r (D)} $ and
    \begin{equation*}
     \| \f 1r \p_{r}\p_{x_1}\{(\p_r \psi)^2\} \|_{L^2_r (D)} = \| \p_{x_1} (\f 1r \p_r) \{(\p_r \psi)^2\} \|_{L^2_r (D)} \le \| (\p_r \psi)^2 \|_{H^3_r (D)}  ,
     \end{equation*}
     using (2), one gets (3).

    \item According to (2) and (3), we know that there exists a constant $c_*$, such that $c^2(\rho) - (\p_r \psi_1 )^2 = c^2 (\bar{\rho}) 
        - (\ga - 1) \bar{u} \p_{x_1} \psi_1  - \f {\ga - 1}2 (\p_{x_1} \psi_1 )^2 - \f {\ga + 1}2 (\p_r \psi_1 )^2\ge c_* >0$ and $\| c^2(\rho) - (\p_r \psi_1 )^2 \|_{H^3_r (D)} \le C_* $ as well as $\| c^2(\rho) - c^2 (\bar{\rho}) - (\p_r \psi_1 )^2 \|_{H^3_r (D)} \le C_* (\eps + \| \psi \|_{H^4_r (D)}) $.
    \item 
  Using (3) and \eqref{fgh2}, there holds
  \be\no
  && \| \p_r \psi_1 \p_r (c^2 (\rho) - (\p_r \psi_1 )^2) \|_{H^2_r (D)}
   \le C_* ( 
    \| ( \p_{x_1} \psi_1 + \bar{u} ) \p_{x_1 } \{( \p_r \psi_1 )^2 \}\|_{H^2_r (D)}
   + \| (\p_r \psi_1 )^2 ( \p_r^2 \psi_1 ) \|_{H^2_r (D)}) \\\no
  && \le
  C_* (
  \| \p_{x_1 } \{( \p_r \psi_1 )^2 \} \|_{H^2_r (D)}  + \| (\p_r \psi_1 )^2  \|_{H^2_r (D)} \| \p_r^2 \psi_1 \|_{H^2_r (D)})
   \le C_* (\eps + \| \psi \|_{H^4_r (D)})^2 .
  \ee
  \end{enumerate}

The {\bf Claim} is proved.

  {\bf Step 2.} The proof of \eqref{coe11}-\eqref{coe14}. For the convenience of writing, set $ \f 1{\Upsilon} \co c^2 (\rho) - (\p_r \psi_1 )^2$ .

  Applying Lemma \ref{alge-est} and the inequalities (2) and (4) in the {\bf Claim}, one gets
  \be\no
  && \| k_{11} (\na \psi_1 ) - \bar{k}_{11} \|_{H^3_r (D)}\\\no
  &&
  \le \| \Upsilon [(\p_r \psi_1 )^2 - (\p_{x_1} \psi_1 )^2 - 2 \bar{u } ( \p_{x_1} \psi_1 ) ]  \|_{H^3_r (D)}  + \b\| \f { \Upsilon \bar{u}^2 (c^2 (\rho) - c^2 (\bar{\rho}) - (\p_r \psi_1 )^2)}{c^2 (\bar{\rho }) } \b\|_{H^3_r (D)} \\\no
  && \le C_* \| \Upsilon \|_{H^3_r (D)} ( \| (\p_r \psi_1 )^2 \|_{H^3_r (D)} + \| \p_{x_1} \psi_1 \|_{H^3_r (D)}^2
  + \| c^2 (\rho) - c^2 (\bar{\rho}) - (\p_r \psi_1)^2 \|_{H^3_r (D)}  +
   \|  \p_{x_1} \psi_1 \|_{H^3_r (D)}
  ) \\\no
  && \leq C_*(\epsilon+\|\psi\|_{H^4_r(D)})\leq C_* (\eps + \de_0).
  \ee
  Similarly, $\| k_1 (\na \psi_1 ) - \bar{k}_1 \|_{H^3_r (D)} \le C_* (\eps + \de_0)$ can be derived, and hence \eqref{coe11}.

  We next prove \eqref{coe12}. It follows from \eqref{fgh2} and (6) in the {\bf Claim} that
  \be\no
  &&\| k_{12} (\na \psi_1) \|_{L^{\oo}_r (D)} \leq \| \f {k_{12} (\na \psi_1)}r \|_{L^{\oo}_r (D)} \le\| \f {k_{12} (\na \psi_1)}r \|_{H^2_r (D)}\\\label{k12-1}
  && \leq C_* \| \Upsilon (\bar{u} + \p_{x_1} \psi_1 ) \|_{H^2_r (D)} \| \f 1r  \p_r \psi_1  \|_{H^2_r (D)}
  \leq C_* (\eps + \|\psi\|_{H^4_r(D)}) \leq C_*(\epsilon+\de_0).
  \ee
  There also holds that
  \be\no
   && \| \p_r k_{12} \|_{H^2_r (D)} \le  C_* \| \Upsilon \p_{x_1 }\{(\p_r \psi_1 )^2\}\|_{H^2_r (D)}
    + \| \Upsilon (\bar{u} + \p_{x_1} \psi_1 ) \p_r^2 \psi_1  \|_{H^2_r (D)} \\\no
    && \q
   + \| \Upsilon^2 (\bar{u} + \p_{x_1} \psi_1  ) \p_r \psi_1 \p_r (c^2 (\rho) - (\p_r \psi_1 )^2) \|_{H^2_r (D)} \\\no
   && \le C_*
   \| \Upsilon \|_{H^2_r (D)}
   \|\p_{x_1 }\{(\p_r \psi_1 )^2\} \|_{H^2_r (D)} + \| \Upsilon (\bar{u} + \p_{x_1} \psi_1 ) \|_{H^2_r (D)} ( \|  \p_r^2 \psi_1 \|_{H^2_r (D)} \\\label{k12-3}
    && \q
    +  \| \p_r \psi_1 \p_r (c^2 (\rho) - (\p_r \psi_1 )^2) \|_{H^2_r (D)} ) \le C_* (\eps + \de_0),
\\\no
&& \| \p_{x_1} k_{12}  \|_{L^{\oo}_r (D)}
\le \| \Upsilon (\bar{u} + \p_{x_1} \psi_1 ) \|_{H^3_r (D)} (\eps + \| \p_r \psi \|_{L^{\oo}_r (D)} + \| \p_{x_1 r}^2 \psi  \|_{L^{\oo}_r (D)}) \\\label{k12-4}
&& \leq C_* ( \eps + \|\psi\|_{H^4_r(D)}) \le C_* ( \eps + \de_0).
\ee
Therefore, \eqref{coe12} can be obtained from \eqref{k12-1}-\eqref{k12-4}.

Using (1) and (6) in the {\bf Claim}, we verify that
\be\no
&& \| \p_{x_1}^2 k_{12} \|_{L^4_r (D)} \le \| \p_{x_1}^2 (\Upsilon (\bar{u} + \p_{x_1} \psi_1 )) \|_{L^4_r (D)} \| \p_r \psi_1  \|_{L^{\oo}_r (D)}  + C_*  \| \p_{x_1} (\Upsilon (\bar{u} + \p_{x_1} \psi_1 )) \|_{L^4_r (D)} \| \p_{x_1 r}^2 \psi_1  \|_{L^{\oo}_r (D)}\\\label{k12-5}
&& \q
+ \| \Upsilon (\bar{u} + \p_{x_1} \psi_1) \|_{L^{\oo}_r (D)} \| \p_{x_1}^2 \p_r \psi_1 \|_{L^4_r (D)} \leq C_*( \eps + \|\psi\|_{H^4_r(D)}) \leq C_* ( \eps + \de_0), 
\\\no
&& \| \na \p_{x_1}^2 k_{12} \|_{L^2_r (D)} \le \| \na \p_{x_1}^2 (\Upsilon (\bar{u} + \p_{x_1} \psi_1 ) ) \|_{L^2_r (D)} \| \p_r \psi_1 \|_{L^{\oo}_r (D)}
 + C_* \|  \na^2 ( \Upsilon (\bar{u} + \p_{x_1} \psi_1 )) \|_{L^2_r (D)} \| \na \p_r \psi_1 \|_{L^{\oo}_r (D)} \\\no
 && \q
 + C_* \| \na (\Upsilon (\bar{u} + \p_{x_1} \psi_1 ) ) \|_{L^4_r (D)} \| \na \p_{x_1 r}^2 \psi_1  \|_{L^4_r (D)}  + \| \Upsilon (\bar{u} + \p_{x_1} \psi_1 ) \|_{L^{\oo}_r (D)} \| \na \p_{x_1}^2 \p_r \psi_1  \|_{L^2_r (D)} \\\no
 &&
 \le C_* \| \Upsilon (\bar{u} + \p_{x_1} \psi_1 ) \|_{H^3_r (D)} (\eps + \| \p_r \psi \|_{L^{\oo}_r (D)} + \| \na \p_r \psi\|_{L^{\oo}_r (D)} + \| \na \p_{x_1 r}^2 \psi \|_{H^1_r (D)} + \| \na \p_{x_1}^2\p_r \psi \|_{L^2_r (D)} ) \\\label{k12-6}
 &&\leq C_* ( \eps + \|\psi\|_{H^4_r(D)}) \leq C_* ( \eps + \de_0).
\ee
Thus, \eqref{coe10} follows from \eqref{k12-5}-\eqref{k12-6}.

It is obviously that
 \be\no
 && \| k_2 (\na \psi_1 )\|_{L^{\oo}_r (D)}\leq \|\frac{k_2 (\na \psi_1 )}{r}\|_{L^{\oo}_r (D)}\leq\| \f {k_2 (\na \psi_1 )}r \|_{H^2_r (D)} \\\label{k2-1}
 && \le \| \Upsilon \|_{H^2_r (D)} \| \f 1r ( \p_r \psi_1 ) \|_{H^2_r (D)}^2 \le C_*  (\epsilon+\de_0)^2,\\\no
 && \| \p_{x_1} k_2 (\na \psi_1 ) \|_{L^{\oo}_r (D)} \le C_* \| \f 1r \p_r \psi_1 \|_{H^2_r (D)} ( \| \p_{x_1} \Upsilon \|_{H^2_r (D)}  \| \p_r \psi_1 \|_{L^{\oo}_r (D)} \\\label{k2-2}
 && \q + \| \Upsilon \|_{H^3_r (D)} \| \p_{x_1 r}^2 \psi_1 \|_{L^{\oo}_r (D)}) \le C_* (\epsilon+\de_0)^2,
\ee
and
\be\no
  && \| \p_r k_2 (\na \psi_1 )\|_{H^2_r (D)}
 \le
 \| \f {2 \Upsilon  \p_r \psi_1 \p_r^2 \psi_1  }{r } \|_{H^2_r (D)}
 + \| \Upsilon (\f 1r \p_r \psi_1)^2 \|_{H^2_r (D)} \\\no
 && \q
 + \| \Upsilon^2 \f { \p_r \psi_1 }r  \p_r \psi_1  \p_r (c^2 (\rho) - ( \p_r \psi_1 )^2) \|_{H^2_r (D)} \\\no
 && \le  C_* \| \Upsilon \|_{H^2_r (D)} \| \f 1r \p_r \psi_1 \|_{H^2_r (D)} \b( \| \p_r^2 \psi_1 \|_{H^2_r (D)}  + \| \f 1r \p_r \psi_1 \|_{H^2_r (D)} \\\label{k2-3}
 && \q
  + \| \p_r \psi_1 \p_r (c^2 (\rho) - ( \p_r \psi_1)^2)  \|_{H^2_r (D)} \b)
  \le C_* (\eps + \de_0)^2.
\ee
Consequently, we end the proof of \eqref{coe13}.

  Note that
  \be\no
  && \| \p_{x_1}^2 k_2 \|_{L^2_r (D)} \le C_* \| \Upsilon \|_{L^{\oo}_r (D)} ( \| \f {\p_r \psi_1 }r \|_{L^{\oo}_r (D)}  \| \p_{x_1}^2 \p_r \psi_1 \|_{L^2_r (D)}  + \| \p_{x_1 r}^2 \psi_1 \|_{L^{\oo}_r (D)} \| \p_{x_1} (\f {\p_r \psi_1 }r) \|_{L^2_r (D)} ) \\\no
  && \q
  + C_* \| \f {\p_r \psi_1 }r \|_{L^{\oo}_r (D)} ( \| \p_{x_1}\Upsilon \|_{L^{\oo}_r (D)} \| \p_{x_1 r}^2 \psi_1 \|_{L^2_r (D)} +  \|  \p_r \psi_1 \|_{L^{\oo}_r (D)} \| \p_{x_1}^2 \Upsilon \|_{L^2_r (D)} ) \\\no
  && \le C_* \| \Upsilon \|_{H^3_r (D)} ( \eps + \| \psi \|_{H^4_r (D)})^2 \le
  C_*  (\epsilon+\de_0)^2,
 \\\no
 && \| \na \p_{x_1}^2 k_2 \|_{L^2_r (D)} \le C_* \| \Upsilon \|_{L^{\oo}_r (D)} \bigg( \| \p_{x_1 r}^2 \psi_1 \|_{L^{\oo}_r (D)} \| \na \p_{x_1} (\f {\p_r \psi_1}r  ) \|_{L^2_r (D)}
 + \| \na (\f {\p_r \psi_1}r  ) \|_{L^4_r (D)} \| \na \p_{x_1}^2 \p_r \psi_1 \|_{L^4_r (D)}
   \\\no
 && \q
 + \| \f {\p_r \psi_1}r \|_{L^{\oo}_r (D)} \| \na  \p_{x_1}^2 \p_r \psi_1  \|_{L^2_r (D)}\bigg)
 + C_* \| \na \Upsilon \|_{L^4_r (D)} \bigg(\| \na (\f {\p_r \psi_1  }r ) \|_{L^4_r (D)} \| \p_{x_1 r}^2 \psi_1 \|_{L^{\oo}_r (D)} + \| \f {\p_r \psi_1  }r \|_{L^{\oo}_r (D)}  \\\no
 && \q
 \times \| \na \p_{ x_1 r}^2 \psi_1 \|_{L^4_r (D)} \bigg)
  + C_* \| \na^2 \Upsilon \|_{L^2_r (D)} \| \f {\p_r \psi_1  }r \|_{L^{\oo}_r (D)} \| \na \p_r \psi_1 \|_{L^{\oo}_r (D)}
 + C_* \bigg(\| \p_{x_1}^2 \Upsilon \|_{L^4_r (D)} \| \na (\f {\p_r \psi_1  }r ) \|_{L^4_r (D)}  \\\no
 && \q
 + \| \na \p_{x_1}^2 \Upsilon \|_{L^2_r (D)}  \| \f {\p_r \psi_1  }r \|_{L^{\oo}_r (D)}\bigg) \| \p_r \psi_1 \|_{L^{\oo}_r (D)}
 \le C_* \| \Upsilon \|_{H^3_r (D)} ( \eps + \| \psi \|_{H^4_r (D)})^2
 \le C_*  (\epsilon+\de_0)^2.
 \ee
 These give \eqref{coe9}.

 Subsequently, we demonstrate \eqref{coe14}.
  \be\no
 && \| F (\na \psi_1 ) \|_{H^3_r (D)} \le C_* \| \Upsilon \|_{H^3_r (D)} (
  \| \p_{x_1} \psi_1 \|_{H^3_r (D)}^2 + \| (\p_r \psi_1 )^2 \|_{H^3_r (D)} )
 \le C_* ( \eps + \de_0)^2,
 \\\no
 && \| \mc{F} (\na \psi ) \|_{H^2_r (D)} \le 
 C_* \eps (\| k_{11} - \bar{k}_{11} \|_{H^2_r (D)}\| \p_{x_1}^2 \psi_0 \|_{H^2_r (D)} + \| k_1 - \bar{k}_1 \|_{H^2_r (D)} \| \p_{x_1} \psi_0 \|_{H^2_r (D)} \\\no
 && \q
 + \| \bar{k}_{11} \p_{x_1}^2 \psi_0 + \bar{k}_1 \p_{x_1} \psi_0 \|_{H^2_r (D)}  )
 + C_* \eps ( \| \f {k_{12}}r \|_{H^2_r (D)} \| r \p_{x_1 r}^2 \psi_0 \|_{H^2_r (D)} + \| \f {k_2}r \|_{H^2_r (D)} \| r \p_r \psi_0 \|_{H^2_r (D)}) \\\no
 && \q
 + \eps \| \f 1r \p_r (r \p_r \psi_0) \|_{H^2_r (D)}
 \le C_* (\eps (\eps + \de_0) + \eps + (\eps + \de_0)^2 ) \le C_* (\eps + (\eps + \de_0)^2).
 \ee

 Hence, the proof of \eqref{coe14} is completed.

{\bf Step 3.} The proof of boundary conditions \eqref{coe15}-\eqref{coe18}.

Since $\psi \in \Si_{\de_0}$, then $ k_{12} (\na \psi_1 ) (x_1, 0/1)=k_2 (\na \psi_1 )(x_1, 0/1)=0$ for $x_1 \in [L_0, L_1]$ and $\p_r \psi (L_0, r) = 0$ for $r \in [0, 1]$. One can further derive that $k_{12} (\na \psi_1 ) (L_0, r) = 0$ for any $r\in [1-\beta_0,1]$. It follows from \eqref{cp2} that
\begin{equation*}
 \p_r c^2 (\rho) (x_1, 0)  =   \p_r c^2 (\rho) (x_1, 1) = 0, \q \forall x_1 \in [L_0, L_1].
\end{equation*}
Then for all $x_1 \in [L_0, L_1]$, the following conditions hold,
\be\no
&& \p_r k_{11} (\na \psi_1 )(x_1, 0) =  \p_r k_{11} (\na \psi_1 ) (x_1, 1) = 0,\q  \p_r k_1 (\na \psi_1 ) (x_1, 0)  =  \p_r k_1 (\na \psi_1) (x_1, 1)  = 0, \\\no
&& \p_r F (\na \psi_1 ) (x_1, 0) =  \p_r F (\na \psi_1 ) (x_1, 1)  = 0, \q  \p_r^2 k_{12} (\na \psi_1) (x_1, 0)  = \p_r^2 k_{12} (\na \psi_1) (x_1, 1)= 0.
\ee

According to Lemma \ref{g2infinity} and the boundary condition \eqref{coe15}, one has
\be\no
\| \p_r k_{11} \|_{L^{\oo}_r (D)}=\| \p_r (k_{11} - \bar{k}_{11}) \|_{L^{\oo}_r (D)} \le c_* (\| \p_r (k_{11} - \bar{k}_{11}) \|_{L^2_r (D)} + \| \na (\p_r (k_{11} - \bar{k}_{11}) ) \|_{L^2_r (D)} \\\no
+ \| \na^2 (\p_r (k_{11} - \bar{k}_{11}) ) \|_{L^2_r(D)} ) \le c_* \| k_{11} - \bar{k}_{11} \|_{H^3_r (D)} \le C_* (\eps + \| \psi \|_{H^4_r (D)}) \le C_* (\eps + \de_0).
\ee
Similarly, there holds
\begin{equation*}
\| \p_r k_1 \|_{L^{\oo}_r (D)} \le c_* \| k_1 - \bar{k}_1 \|_{H^3_r (D)} \le C_* (\eps + \| \psi \|_{H^4_r (D)}) \le C_* (\eps + \de_0).
\end{equation*}
These give \eqref{coe19}. The proof of Lemma \ref{coe-estimate} is completed.
\end{proof}

{\bf Acknowledgement.} Weng is supported by National Natural Science Foundation of China (Grant No. 12071359, 12221001).

{\bf Statement.} There is no conflict of interests and there is no associated data for this manuscript.

\end{document}